\documentclass[12pt]{article}

\usepackage[T1]{fontenc}
\usepackage{amsmath,amssymb,amsthm}
\usepackage{mathrsfs}
\usepackage{bm}
\usepackage[colorlinks,citecolor=blue,urlcolor=blue,linkcolor=blue]{hyperref}
\usepackage[numbers]{natbib}
\usepackage{graphicx}
\usepackage{booktabs}
\usepackage{float}      
\usepackage{caption}
\usepackage{placeins}
\usepackage{authblk}


\numberwithin{equation}{section}

\newtheorem{theorem}{Theorem}[section]
\newtheorem{lemma}[theorem]{Lemma}

\newtheorem{remark}[theorem]{Remark}

\theoremstyle{definition}
\newtheorem{definition}[theorem]{Definition}
\newtheorem{assumption}[theorem]{Assumption}

\usepackage{algorithm}
\usepackage{algpseudocode}
\usepackage{xcolor}

\title{Strategic Index Reconstitution: Differential Games, Closed-Loop Equilibria and Mean-Field Dynamics}
\author[1]{Lukas-Benedikt Fiechtner}
\author[1,2]{Jose Blanchet}
\affil[1]{Institute for Computational and Mathematical Engineering, Stanford University}
\affil[2]{Department of Management Science and Engineering, Stanford University}
\date{October 20, 2026}

\begin{document}

\vspace{-0.5cm}
\maketitle

\begin{abstract}
We study a continuous-time, multiasset trading game around index reconstitution,
in which opportunistic traders have heterogeneous beliefs about the future
index constitution and interact through transient cross-asset price impact
and instantaneous execution costs.
Opportunistic traders build inventory
based on their individual beliefs about the future index constitution before
a public announcement, adjust their positions to the revealed composition, and unwind them against the indexer's prescribed
rebalancing trades at the implementation date when the new index
composition takes effect. Under a
no-price-manipulation condition, we construct a subgame-perfect
Nash equilibrium. Its affine feedback policies are
computed from a non-standard matrix Riccati equation and linear ODEs whose
number and dimensions are independent of population size.
Under mean-field scaling, we construct a
mean-field equilibrium and obtain quantitative convergence and
approximate-Nash bounds. Numerical illustrations show that greater competition
among opportunists dampens price swings around implementation and that
opportunistic trading lowers the indexer's implementation shortfall unless
price resilience is very low.
Holding the average belief fixed, greater belief dispersion can increase
gross trading volume without changing net aggregate opportunist holdings or
price impact.
\end{abstract}

\begingroup
\small
\begin{description}
\setlength{\itemsep}{0pt}
\setlength{\parsep}{0pt}
\item[Mathematics Subject Classification (2020):] 91A15, 91A16, 91G15, 93E20
\item[JEL Classification:] C73, C61, G11, G14
\item[Keywords:] index rebalancing, stochastic differential games, closed-loop Nash equilibrium, mean-field games, transient market impact, heterogeneous beliefs
\end{description}
\endgroup

\bigskip

\section{Introduction}\label{sec:introduction}

When an index changes its constituents, index funds must adjust their holdings
to track the new portfolio. The changes become public on the announcement date and index funds typically execute the required trades in the closing auction on
the later implementation date, immediately before the new index composition
takes effect. To profit from these required trades, opportunists may build
anticipatory positions before the announcement, revise them once the changes
are announced, and unwind them as index funds rebalance. Their problem is
dynamic: each trade moves prices with impact persisting for some time, and other
opportunists can respond to the resulting changes in prices and inventories.
Press accounts illustrate the financial stakes of anticipating index trades:
The Economist has called index rebalancing ``the biggest event in markets.''\footnote{The Economist,
\href{https://www.economist.com/finance-and-economics/2026/05/12/index-rebalancing-is-now-the-biggest-event-in-markets}
{\emph{Index Rebalancing Is Now the Biggest Event in Markets}}, May 12, 2026.}
Bloomberg reported that two Millennium
index-rebalancing teams
earned about \$3.7 billion in June 2026 and that the same teams had lost about
\$900 million in early 2025.\footnote{Nishant Kumar,
\href{https://www.bloomberg.com/news/articles/2026-07-06/two-millennium-trading-pods-made-about-3-7-billion-last-month}
{\emph{Two Millennium Trading Pods Made About \$3.7 Billion Last Month}},
Bloomberg, July 6, 2026. The report covers several June rebalancing events,
including quarter-end multi-asset rebalancing. Bei Hu, Katherine Burton, and
Nishant Kumar,
\href{https://www.bloomberg.com/news/articles/2025-03-08/millennium-loses-900-million-on-strategy-roiled-by-market-chaos}
{\emph{Millennium Loses \$900 Million on Strategy Roiled by Market Chaos}},
Bloomberg, March 7, 2025.}

The empirical record reinforces the importance of understanding how
opportunists trade around index rebalancing.
Studies of S\&P 500 additions in the late 1970s and early 1980s document
significant abnormal returns of roughly \(3\%\) on the first trading day
after announcement
\citep{Shleifer1986,HarrisGurel1986}, with subsequent reversals \citep{HarrisGurel1986}.
 However, average abnormal returns around announcement and
implementation have declined in recent years: comparing the 1990s with the 2010s,
\citet{GreenwoodSammon2025} report a fall from \(4.94\%\) to \(1.02\%\)
around announcement and from \(3.55\%\) to \(0.26\%\) around
implementation.\footnote{For these event windows, abnormal returns are measured
net of the S\&P 500 return, from the preceding trading day's close to the
following trading day's close.}
On the other hand, abnormal returns over the 100 trading days \emph{before} announcement
increased from \(9.60\%\) to \(18.68\%\) over the same periods
\citep{GreenwoodSammon2025}, suggesting increased anticipatory trading
even before the announcement.\footnote{The increase may not only reflect trading in anticipation of
index inclusion, but also selection of past winners, or both \citep{GreenwoodSammon2025}.}

While abnormal returns around announcement and implementation have declined,
S\&P 500 reconstitution volume relative to normal daily trading has risen
\citep{ChincoSammon2024}. For additions from outside the S\&P index family in
2010--2024, \citet{PegoraroSammonShim2026} report implementation-day turnover
near \(20\%\) of shares outstanding despite an average abnormal return of
\(-0.09\%\). Interestingly, implementation-day volume substantially exceeds the trades
implied by labeled index funds alone, consistent with additional trading by
institutions managing their own index portfolios and active managers closely
tracking benchmarks \citep{ChincoSammon2024}. Evidence from equal-weight
rebalances shows substantial additional displayed buying and selling
interest in the closing auction \citep{Nathan2026}.

Furthermore, for S\&P 500 additions from outside the index family,
\citet{PegoraroSammonShim2026} find short interest rising around implementation
and declining afterward, consistent with opportunists going short to supply
inventory to index funds in addition to selling the inventory they built up
before implementation.
\citet{PegoraroSammonShim2026} also relate the timing of inventory buildup to
predictability: less predictable Russell migrations exhibit more trading
between the rank and implementation dates, which they interpret as consistent
with earlier positioning when membership changes are easier to anticipate.

This empirical evidence motivates us to study how opportunists build
positions before announcement using their probabilities over possible index
reconstitutions, revise their positions when membership is revealed, and
transfer inventory as index funds rebalance.

\paragraph{Model.}
We briefly outline the model here and give the full specification later in
Section~\ref{sec:model}.
We consider \(n\geq2\) opportunists trading \(d\) assets over a finite
horizon \([0,T]\), with announcement at \(T_{\mathrm{ann}}\) and
implementation at \(T_{\mathrm{impl}}\). Before \(T_{\mathrm{ann}}\),
speculators consider \(m\) possible scenarios for the new index constitution.
Each scenario \(k\) determines the indexer's required purchases of added
assets and sales of deleted assets. The prescribed
trading rate \(g^k\) approximates execution in the closing auction by
spreading these purchases and sales uniformly over
\([T_{\mathrm{TWAP}},T_{\mathrm{impl}}]\), with
\(T_{\mathrm{ann}}<T_{\mathrm{TWAP}}<T_{\mathrm{impl}}\).
The realized scenario is announced publicly at \(T_{\mathrm{ann}}\) and determines the realized
trading rate \(g\) followed by the indexer.
Between announcement and implementation, opportunists can adjust their
positions in response to the revealed scenario.
The remaining interval from \(T_{\mathrm{impl}}\) to \(T\) allows opportunists to unwind
positions after implementation.
We allow traders to hold different beliefs about the possible index
constitutions: before the announcement, trader \(i\)'s subjective probability
measure \(\mathbb P_n^i\) assigns probabilities \(\pi_n^i\in\Delta_m\)
to the possible scenarios.
Each trader knows its own prior $\pi^i_n$ and the sample-average prior
\(\bar\pi_n=\frac1n\sum_{j=1}^n\pi_n^j\),
but need not know the other traders' priors. We assume that no further
information about index membership arrives before \(T_{\mathrm{ann}}\), so
beliefs remain fixed until the public announcement.
Trader \(i\)'s inventory \(X_n^i\) evolves as \(dX_n^i(t)=u_n^i(t)dt\) with trading rate \(u_n^i\).
The midprice is \(S=S^0+I\), where \(S^0\) is the fundamental price which is a square-integrable c\`adl\`ag
martingale independent of the realized index scenario under each \(\mathbb P_n^i\) and
\(I\) is the transient impact satisfying
\(dI(t)=\left[-RI(t)+\Gamma\left(\varpi_n\sum_{j=1}^n u_n^j(t)+g(t)\right)\right]dt\).
Here \(R\) governs impact decay, and \(\Gamma\) maps trading rates into
direct and cross-asset impact.
Under the fixed-market-size scaling \(\varpi_n=1\), increasing \(n\)
adds opportunists of fixed size to the same market. Under the mean-field
scaling \(\varpi_n=\varpi/n\), total opportunist mass \(\varpi>0\)
stays fixed as \(n\) increases and individual traders become smaller.
Trader \(i\)'s execution price is
\(
\widehat S^i(t)=S(t)+\Lambda\left(u_n^i(t)
+\chi_u\varpi_n\sum_{j\ne i}u_n^j(t)+\chi_g g(t)\right)
\).
Under \(\mathbb P_n^i\), trader \(i\) minimizes expected trading costs net
of the change in inventory value at fundamental prices, with quadratic
running and terminal inventory penalties.
Traders choose their admissible trading rates using \emph{closed-loop} strategies,
that is, feedback functions of the state history up to the current time.
Here we allow only partial observation of the state, in that each trader
observes its own inventory, average inventory, transient impact, the
fundamental price, and the public announcement, but not the individual
inventories of the other traders.

\paragraph{Main Results and Contributions.}
\begin{enumerate}
\item \textbf{Construction and Global Existence of a Feedback Equilibrium.}
We obtain \emph{subgame-perfect} closed-loop policies that form a Nash
equilibrium from every starting time and state.
The policies are affine, with coefficients found by solving a nonstandard
matrix Riccati ordinary differential equation (ODE) and linear ODEs.
The ODE dimensions do not grow with the number of players.
As long as a solution of the Riccati equation exists on the full horizon,
a verification theorem shows that the policies form a subgame-perfect
Nash equilibrium.
Under natural assumptions ruling out price manipulation, we show that the
Riccati solution does not blow up in finite time.
To our knowledge, this is the first globally solvable continuous-time,
finite-player, closed-loop execution game with transient price impact.

\item \textbf{Mean-Field Approximation.}
Under the mean-field scaling and the no-price-manipulation conditions, we
construct an explicit equilibrium for a possibly continuous distribution of
belief types, unique within the affine construction.
The resulting trading schedules can be implemented without real-time observation
of aggregate inventory or the exact number of competitors.
Let \(\bar\pi\) denote the average belief in the limiting population and
set \(\varepsilon_n=\frac1n+\|\bar\pi_n-\bar\pi\|\).
For zero initial inventories and impact and \(\bar\pi_n\to\bar\pi\),
the finite-player equilibrium trading rates, inventories, and impact
converge to their mean-field counterparts with matching beliefs, with
\(O(\varepsilon_n)\) errors uniformly over time, scenario, and trader.
Using the mean-field trading schedules as feedbacks depending only on time and the
public announcement yields an \(O(\varepsilon_n^2)\)-Nash equilibrium in
the finite game from the same initial state.

\item \textbf{Equilibrium Dynamics and Belief Heterogeneity.}
We show that the finite-player and mean-field equilibrium trajectories
exhibit the following decomposition according to traders' beliefs.
Before announcement, the common state (average inventory and impact) and
individual deviations from average inventory satisfy decoupled linear
systems driven by the average belief and individual belief
deviations from the average belief, respectively.
After announcement, each trader's deviation from average inventory follows
the same path in every scenario.

\item \textbf{Numerical Illustrations.}
We obtain the following economic insights from a stylized five-asset
index rebalance under the fixed-market-size scaling.
Adding opportunists of fixed size dampens price swings around
implementation and generally lowers the indexer's expected implementation
shortfall with diminishing marginal reductions, because opportunists trade
against the indexer during implementation.
Relative to trading without opportunists, the indexer's expected
implementation shortfall is reduced most at intermediate resilience but
can be higher than without opportunists when the impact of opportunists'
earlier trades decays very slowly.
In added stocks, opportunists sell more than their accumulated inventory
and subsequently cover the resulting aggregate short positions, consistent
with the short-interest evidence \citep{PegoraroSammonShim2026}.
With the average belief fixed, greater dispersion of traders' beliefs
raises gross trading through offsetting purchases and sales among
opportunists, without changing net aggregate inventory or price impact.
\end{enumerate}

\paragraph{Related Literature.}
The closest economic model is that of
\citet{PegoraroSammonShim2026}.  They study a two-stock model with
three trading dates. Index membership is public before trading begins (i.e. there is no pre-announcement trading), and an
indexer chooses its schedule against multiple identical Cournot
speculators, either without commitment or as a Stackelberg leader.  Purchases
of the incoming stock and sales of the outgoing stock use the same schedule,
and impact is asset-separable, with permanent and instantaneous components.
\citet{Tasitsiomi2025} studies a
complementary one-stock model with one manager and one anticipatory trader
after an index addition has been announced. Both papers study inventory
acquisition \emph{after} announcement and liquidity provision at implementation
using \emph{open-loop} trading schedules. Our \emph{closed-loop} game
with transient cross-asset price impact includes inventory buildup \emph{before} announcement
under differing beliefs about index membership and position adjustments
after membership is revealed.
Our model also allows an arbitrary number of assets, with purchase and sale
quantities determined by the indexer's initial and target holdings rather
than constrained to match in shares.

Our transient-impact dynamics are motivated by the resilient
limit-order-book model of \citet{OBIZHAEVA20131}.
For conditions ruling out price manipulation in one- and multiasset
transient-impact models, see, for example,
\citep{gatheral_2012,alfonsi2016,Schneider02012019,jaber2026}.
In our model, the no-price-manipulation assumption yields a priori bounds
that prevent Riccati blow-up and thus guarantee existence of a Nash equilibrium.

\emph{Market-impact games} model strategic interaction among traders whose orders
affect both their own and their rivals' execution prices.
In discrete time, a two-player game with linear permanent and exponentially
decaying transient impact is studied under hard liquidation in both open-loop
and closed-loop formulations in \citep[Chapter~9]{schoeneborn2008trade}.
With general impact kernels and quadratic transaction costs, the
open-loop game is solved for two risk-neutral players
\citep{schied_zhang_2019}, extended to \(n\) risk-averse players
\citep{luo_schied_2019}, and further extended to multiple assets with
cross-impact \citep{CordoniLillo2024}.
For one asset and exponential decay, the high-frequency limit of the
risk-neutral game is studied for two players
\citep{schied_strehle_zhang_2017} and then for \(n\) players
\citep{NutzProsperi2025}.

In contrast to the preceding discrete-time models,
\citet{schied_zhang_2017} and \citet{DrapeauLuoSchiedXiong2021} formulate
continuous-time games directly in terms of trading rates and obtain
open-loop equilibria in market impact games with permanent impact and instantaneous execution costs.
For transient impact, \citet{strehle_2017} establishes existence and
uniqueness of an open-loop Nash equilibrium with a general positive-type
kernel and strictly positive quadratic trading costs.
\citet{campbell2026optimalexecutionntraders} give explicit equilibria for
\(N\) traders in a one-asset Obizhaeva--Wang model with positive quadratic
trading-rate costs, both with finite terminal penalties and under hard
liquidation.
In a companion paper, \citet{CampbellNutzRandomization} show that private randomization cannot
restore existence when no pure equilibrium exists and prove that there is
at most one equilibrium in a broad class of transient-impact games.
\citet{CarteaJaimungalSanchezBetancourt2026} characterize a one-asset game
between a broker and an informed trader by coupled forward--backward stochastic
differential equations and a matrix Riccati equation, with a horizon bound
in their general existence result for transient impact.

The literature on closed-loop market-impact games is much smaller than its
open-loop counterpart. In open-loop equilibria, rivals' trading processes
remain fixed under a unilateral deviation. In closed-loop equilibria, rivals
retain their feedback functions, so their realized trading rates can respond to
the changed state. In continuous time, \citet{carlin2007} derive a feedback
HJB/Riccati system with permanent impact and instantaneous execution costs.
\citet{CarmonaYang2008PredatoryTrading} formulate a two-player feedback
game with permanent impact and instantaneous execution costs and analyze its
degenerate HJB system numerically. \citet{micheli2023} give a
continuous-time, finite-player closed-loop equilibrium with a common signal,
but their one-asset infinite-horizon model has no transient-impact state and
requires price impact to be sufficiently small. With transient price impact,
\citet{Ohnishi02102020,Ohnishi03072024} obtain Markov-perfect equilibria with
two players in discrete time. The missing case is a continuous-time,
finite-player transient-impact game with
subgame-perfect strategies and full-horizon solvability.  No prior result
supplies this conjunction even for one asset. Our theorem allows matrix-valued
cross-impact across \(d\) assets and differing beliefs.

Closed-loop trading is also studied in a different class of models with
endogenous market-clearing prices.
\citet{ChoiLarsenSeppi2019} characterize a discrete-time Bayesian Nash
equilibrium with closed-loop trading strategies.
In continuous time,
\citet{ChoiLarsenSeppi2021Benchmarks,ChenChoiLarsenSeppi2022} establish
finite-player subgame-perfect equilibria.
In a related formulation based on perceived linear trading responses,
\citet{ChenChoiLarsenSeppi2023} prove global existence of a continuous-time
financial-market Nash equilibrium for arbitrary finite horizons.

Mean-field versions of market-impact games are studied by
\citet{CardaliaguetLehalle2018,LehalleMouzouni2019,NeumanVoss2023}
and, within a general functional linear--quadratic framework, by
\citet{AbiJaberNeumanVoss2023}.
Belief heterogeneity enters through different subjective probability
measures in \citet{Casgrain2020}, who also allow heterogeneous trading
objectives.
Other forms of heterogeneity include risk tolerances and relative-performance
concerns in portfolio games without price impact
\citep{Lacker2019,Lacker2020}, and trading horizons
\citep{LeclereRosenbaum2026}.
In our model, traders differ in their prior probabilities of possible index
constitutions, with the realized constitution publicly revealed at
announcement.

Mean-field convergence results connect finite-player market-impact games
to limiting games with a continuum of traders.
\citet{NeumanVoss2023} construct finite-player and
mean-field open-loop equilibria under transient impact and a common predictive
signal, and prove convergence and approximate-Nash bounds.  The
linear--quadratic game over trading paths of
\citet{AbiJaberNeumanVoss2023} also yields open-loop finite-player convergence
and approximate equilibria. Our analysis instead establishes convergence of
explicitly constructed finite-player subgame perfect \emph{closed-loop} equilibria and approximate-Nash
bounds for the mean-field policies. Our estimates also separate the finite-player
term \(1/n\) from the discrepancy between the finite-player and limiting
average beliefs. These model-specific results complement general closed-loop
mean-field convergence theory \citep{Lacker2020ClosedLoop,LackerLeFlem2023}
and quantitative convergence results for games with interaction through
controls \citep{LauriereTangpi2022}.

Our paper also connects to the literature on major--minor mean-field games.
Under our mean-field scaling, the indexer plays the role of a major agent
because its exogenously prescribed trading flow remains unscaled, whereas
the opportunists are minor players whose individual market impact vanishes
as the population grows.
\citet{Huang04032019} study anticipatory trading in a major--minor game in
which both the institutional trader and HFTs trade continuously.
\citet{ChengWangXu2024} introduce a trading-speed asymmetry: the large trader
trades at discrete dates and HFTs trade continuously, with both prescribed
and optimized large-trader schedules considered.
\citet{chen2024periodictradingactivitiesfinancial} study an open-loop
major--minor liquidation game in which the major trader balances execution
costs against inventory deviations from a prescribed periodic trading
benchmark.
These papers construct mean-field equilibria and establish finite-population
approximate-Nash properties. We explicitly solve the finite-player feedback
game, define and solve the corresponding mean-field game, and then prove
that the constructed finite-player equilibria converge to the mean-field
equilibrium.

Finally, the roles of predation and liquidity provision around a predictable
order have a long history.  Strategic traders may amplify a forced order or
supply inventory against it
\citep{brunnermeier_2005,carlin2007,schoeneborn2009liquidation,strehle_2017}.
\citet{BessembinderCarrionTuttleVenkataraman2016} show, theoretically and
empirically, how strategic trading can improve execution around predictable
ETF roll orders when impact is sufficiently resilient.
\citet{Herzing2026} studies a two-player continuous-time open-loop game with
uncertainty about opponents' required trading volumes and no belief updating.
Further evidence on execution costs,
predictability, and crowding around index events appears in
\citep{Petajisto2011,Li2021,MicheliNeuman2022,JiangWuYao2025}.  Our
contribution is to place these mechanisms inside one finite-player feedback
game that spans uncertainty, public revelation, implementation, and subsequent
unwinding, and to prove that the game is solvable on the entire horizon.

The remainder of the paper is organized as follows.
Section~\ref{sec:model} presents the model.
Section~\ref{sec:solution} constructs the finite-player equilibrium.
Section~\ref{sec:mean_field_limit} studies the mean-field limit.
Section~\ref{sec:experiments} presents numerical illustrations.
The proofs are collected in the appendices.

\paragraph{Notation.}
The symbol \(\odot\) denotes componentwise (Hadamard) multiplication.
We write \(0_r\in\mathbb R^r\) and \(0_{r\times s}\in\mathbb R^{r\times s}\)
for the zero vector and zero matrix, respectively, and omit the subscripts
when their dimensions are clear from context.
We write \(\Delta_m:=\{p\in[0,1]^m:\sum_{k=1}^m p^k=1\}\) for the
probability simplex.
For any dimension \(r\), \(\mathbb S^r\) denotes the real symmetric
\(r\times r\) matrices, with \(\mathbb S_+^r\) and \(\mathbb S_{++}^r\)
their positive semidefinite and positive definite subsets.
For symmetric matrices, \(\mathsf M\preceq\mathsf N\)
(respectively, \(\mathsf M\prec\mathsf N\)) means that
\(\mathsf N-\mathsf M\) is positive semidefinite (respectively, positive
definite).
We use \(\|\cdot\|\) for the Euclidean norm of vectors and any fixed matrix
norm of matrices.

\section{Model}\label{sec:model}

\subsection{Index Constitution and Indexer Trades}

We consider an index reconstitution over a deterministic horizon $[0,T]$,
with relevant times $0<T_{\mathrm{ann}}<T_{\mathrm{impl}}<T$.
At the announcement time $T_{\mathrm{ann}}$, the new index constitution is publicly announced.
The date on which the index adopts the announced constitution is called the effective date.
Indexers usually execute their rebalancing trades in the closing auction on the preceding trading day.
In our model, we identify the effective date with the implementation time
$T_{\mathrm{impl}}$ of this closing auction.
We work on a finite eligible universe of $d$
tradable stocks. Here, $d$ denotes the size of the eligible universe, which
contains the current index constituents and the set of plausible entrants.

The new index reconstitution is characterized by two quantities, both of which
are made public at \(T_{\mathrm{ann}}\).
The first announced quantity is the vector of \emph{post-announcement membership
indicators} $\xi\in\{0,1\}^d\setminus\{0_d\}$, where $\xi_a=1$ means that
asset $a$ is included in the new index constitution and $\xi_a=0$ means that it is not.
The second announced quantity is a vector of \emph{floating share counts}
$f\in\mathbb{R}_{++}^d$, where $f_a$ is the number of floating shares of stock
$a$ used to compute float-adjusted market-capitalization weights.
Once the realized membership vector \(\xi\) and the vector of floating shares \(f\) have been announced,
the required trade and resulting flow of an indexer replicating the
index can be computed. To do so, let
$h^0\in\mathbb{R}_+^d\setminus\{0_d\}$ denote the indexer's pre-announcement holdings
\emph{measured in shares}, so that $h_a^0=0$ for names not held
before the event. After the announcement, the indexer's target stock holdings
are proportional to $f\odot \xi$, i.e.,
\(h=\kappa\,(f\odot \xi)\)
for some positive proportionality factor $\kappa$.
In our model, $\kappa$ is chosen using a
fixed reference-price approximation: Fix
\(s^{\mathrm{ref}}\in\mathbb{R}_{++}^d\), which we usually choose as the initial
midprice \(S(0)\), with \(S(t)\) defined later in \eqref{eq:midprice}, and impose a
self-financing rebalance at this reference price.
Writing \(\Delta h:=h-h^0\) for the indexer's total
rebalance quantity, the self-financing condition is
\(0 = (s^{\mathrm{ref}})^\top \Delta h
  = (s^{\mathrm{ref}})^\top\big(\kappa\,(f\odot \xi)-h^0\big)\).
Solving for $\kappa$ gives
\begin{equation}
\kappa
:=\frac{(s^{\mathrm{ref}})^\top h^0}
{(s^{\mathrm{ref}})^\top(f\odot \xi)}
\label{eq:kappa_ref}
\end{equation}
Substituting this expression into \(\Delta h\) gives the indexer's required
rebalance quantity:
\begin{equation}\label{eq:order_quantity}
\Delta h
=\frac{(s^{\mathrm{ref}})^\top h^0}
{(s^{\mathrm{ref}})^\top(f\odot \xi)}(f\odot \xi)-h^0.
\end{equation}
We do not model the indexer as a player choosing how to execute this order.
Instead, we approximate the concentrated closing-auction trade by prescribing
a deterministic TWAP that executes the quantity \(\Delta h\) uniformly
over the short interval \([T_{\mathrm{TWAP}},T_{\mathrm{impl}}]\), where
\(T_{\mathrm{ann}}<T_{\mathrm{TWAP}}<T_{\mathrm{impl}}\). With \(\Delta_{\mathrm{TWAP}}:=T_{\mathrm{impl}}-T_{\mathrm{TWAP}}\)
its trading rate is
\begin{equation}\label{eq:twap_rate}
g(t)
:=
\frac{\Delta h}{\Delta_{\mathrm{TWAP}}}
\mathbf 1_{[T_{\mathrm{TWAP}},T_{\mathrm{impl}}]}(t).
\end{equation}
We call \(T_{\mathrm{TWAP}}\) and \(T_{\mathrm{impl}}\), where
\(g\) switches on and off, the switching times.

\begin{remark}[Reference-Price and TWAP Approximations]%
\label{rem:twap_approximation}\label{rem:discrete_closing_auction}
As described above, indexers usually execute most or all of their rebalance in the closing auction
on the trading day before the effective date, with many public-index-tracking
ETFs executing entirely at the auction price
\citep{Li2021,JiangWuYao2025}. They do so because closing-auction execution aligns the indexer's
transaction prices with the closing prices and thus eliminates tracking error
from execution-price differences.
A more literal treatment of the implementation-date auction in our model would thus
let the indexer trade only in a discrete closing auction. Trader \(i\)
would submit an auction quantity \(u_{\mathrm A}^i\in\mathbb R^d\), with
\(U_{\mathrm A}:=\sum_{j=1}^n u_{\mathrm A}^j\) denoting the aggregate
auction quantity of all $n$ opportunists, and the indexer would submit its order $g_A$.
The effect of total auction imbalance on the auction price could, for example, be modeled by
\begin{equation}\label{eq:discrete_auction_price}
P_{\mathrm{auction}}
=
S(T_{\mathrm{impl}}-)
+
H_{\mathrm A}
\bigl(
U_{\mathrm A}+g_{\mathrm A}+\xi_{\mathrm A}
\bigr),
\end{equation}
where \(S\) is the midprice defined later in \eqref{eq:midprice},
\(H_{\mathrm A}\) is a symmetric positive definite auction-impact matrix,
and
\(\xi_{\mathrm A}\) is mean-zero exogenous auction imbalance.
The indexer's auction order is again \(g_{\mathrm A}=\kappa(f\odot\xi)-h^0\).
Exact self-financing would determine \(\kappa\) from \eqref{eq:kappa_ref}
with \(s^{\mathrm{ref}}\) replaced by \(P_{\mathrm{auction}}\).
Since \(P_{\mathrm{auction}}\) depends on \(\kappa\) through \(g_{\mathrm A}\),
this gives a fixed-point equation for \(\kappa\) which, for fixed opportunist orders with
\(p_{\mathrm A}:=S(T_{\mathrm{impl}}-)+H_{\mathrm A}(U_{\mathrm A}+\xi_{\mathrm A}-h^0)>0\),
has a unique positive root giving the indexer's auction quantity
\begin{equation}\label{eq:auction_kappa_quadratic}
g_{\mathrm A}
=\frac{-b_\kappa+\sqrt{b_\kappa^2+4a_\kappa c_\kappa}}{2a_\kappa}
(f\odot\xi)-h^0,
\end{equation}
where \(a_\kappa:=v_{\mathrm A}^\top H_{\mathrm A}v_{\mathrm A}\),
\(b_\kappa:=p_{\mathrm A}^\top v_{\mathrm A}-(h^0)^\top H_{\mathrm A}v_{\mathrm A}\),
and \(c_\kappa:=p_{\mathrm A}^\top h^0\),
with \(v_{\mathrm A}:=f\odot\xi\).
The indexer's order in \eqref{eq:auction_kappa_quadratic} generally depends
nonlinearly on the opportunists' auction quantities $u_A^i$, so once the order
is substituted into \eqref{eq:discrete_auction_price}, trader \(i\)'s auction
payment \(P_{\mathrm{auction}}^\top u_A^i\) is generally nonquadratic in
\(u_A^i\). The nonlinear dependence of the indexer's auction order
on the pre-auction midprice also makes auction payments generally nonquadratic
in earlier opportunist trades,
since those trades affect the midprice through market impact as specified later
in \eqref{eq:midprice}.
To preserve the LQ structure, we thus determine \(\kappa\) from the fixed exogenous
reference price \(s^{\mathrm{ref}}\), as in \eqref{eq:kappa_ref}.
With \(g_{\mathrm A}=\Delta h\) fixed, trader \(i\)'s auction payment
\(P_{\mathrm{auction}}^\top u_{\mathrm A}^i\) then is quadratic jointly in the
submitted quantities and the pre-auction midprice.

With the reference-price approximation in place, the backward procedure
described later in Section~\ref{sec:solution} for the game with TWAP execution
could be adapted to a game with a discrete closing auction. One would link
three continuous-time LQ games---before the announcement, between announcement
and auction, and after the auction---through the announcement and the discrete
LQ auction game at \(T_{\mathrm{impl}}\).
Working backward, one first solves the post-auction
continuous-time game and uses its continuation value in the auction game.
The auction equilibrium value then supplies the
quadratic terminal condition for the continuous game preceding the auction,
which must in turn be linked through the announcement to the pre-announcement game.
If the backward construction is well defined, adapting the verification
argument presented later in Section~\ref{sec:solution} yields a Nash
equilibrium. Proving a priori that the Riccati solutions remain finite
after the auction changes their terminal matrices would be more involved
than for TWAP. Although we expect such a proof to be feasible, we retain
the TWAP approximation \eqref{eq:twap_rate} because it permits a simpler
a priori solvability proof based on a single Riccati equation over the full horizon.
\end{remark}

\begin{remark}[Simultaneous Rebalancing Across Indices]
In practice, a stock added to one index may be deleted from another at the same
time which is called a migration. The indexer tracking the first index then buys the stock, while the
indexer tracking the second sells it, so their orders partly offset. 
In our model such simultaneous rebalancing can be incorporated 
by defining \(g(t)\) as the net deterministic flow
obtained by summing the orders of all affected indexers.
\end{remark}

\subsection{Membership Scenarios and Trader Beliefs}

Let $(\Omega,\Sigma)$ be a measurable space on which each of the $n$
opportunistic traders has a probability measure $\mathbb P_n^i$, $i=1,\ldots,n$.
Since the membership vector is uncertain before
the announcement, we model it as a random variable $\Xi$ with values in the set 
$\mathcal X:=\{\xi^1,\ldots,\xi^m\}\subset\{0,1\}^d\setminus\{0_d\}$ of $m$ possible scenarios.
In our model, we treat \(f\) as fixed and known to all players, so \(\Xi\) is the only
uncertain component of the announced index constitution.
Before the announcement, trader \(i\) evaluates membership
uncertainty under \(\mathbb P_n^i\), with
\[
\mathbb P_n^i(\Xi=\xi^k)=\pi_n^{i,k},
\qquad k=1,\dots,m.
\]
We write \(\pi_n^i=(\pi_n^{i,1},\dots,\pi_n^{i,m})\in\Delta_m\) for trader \(i\)'s
\emph{scenario prior} vector, which may differ across traders, and define the \emph{finite-population
average prior} by
\begin{equation}\label{eq:finite_population_average_prior}
\bar\pi_n:=\frac{1}{n}\sum_{j=1}^n\pi_n^j.
\end{equation}
We assume that \(\bar\pi_n\) is common knowledge and that each trader \(i\) knows its own
prior \(\pi_n^i\), but not the individual priors of the other traders.
Traders receive no additional information about
\(\Xi\) before \(T_{\mathrm{ann}}\) and stick to their individual beliefs about
scenario probabilities. At
\(T_{\mathrm{ann}}\), the realized membership vector is publicly revealed, so
all traders then know which scenario has occurred.
We encode this information arrival through the public announcement process
\(A(t):=\Xi\mathbf 1_{\{t\geq T_{\mathrm{ann}}\}}\) for \(0\leq t\leq T\),
which takes values in \(\{0_d\}\cup\mathcal X\) with $0_d$ denoting that the constitution has not been announced.
For scenario \(k\), denote the target holdings, required order, and corresponding
indexer flow for \(\xi=\xi^k\) by $h^k$, $\Delta h^k$, and $g^k(t)$, respectively.
The realized indexer-flow process is
\(g(t)=\sum_{k=1}^m \mathbf 1_{\{\Xi=\xi^k\}}\,g^k(t)\).

\begin{remark}[Predictability of Index Membership]\label{rem:membership_predictability}
Our model is particularly suited to discretionary index reconstitutions when
little information about the selection decision arrives before the announcement.
For example, S\&P 500 selection involves committee discretion, so
observing information like market capitalizations and eligibility does not fully resolve membership
uncertainty before the announcement.\footnote{For example, Datadog's July 2025 addition to the
S\&P 500 surprised some investors who had expected Robinhood or AppLovin
instead. See Silin Chen,
\href{https://www.thestreet.com/investing/analyst-resets-datadog-stock-price-target-after-surprise-addition-to-s-p-500}
{\emph{Analyst Resets Datadog Stock Price Target After Surprise Addition to
S\&P 500}}, TheStreet, July 4, 2025.}
Fixed pre-announcement priors approximate the limited pre-announcement learning in
such settings. Traders observing the same public information may also differ
in their assessments of the committee's eventual decision which is captured by our trader-specific
scenario priors.
By contrast, membership in rule-based indices such as Russell is determined
by published rules applied to observable data, including market
capitalizations. Thus, observing the relevant data over time allows traders to
update membership forecasts before the announcement.
\end{remark}

\subsection{Trading Dynamics and Objectives}\label{sec:opportunists}

\paragraph{Price and Inventory Dynamics.} Let \(S^0=(S^0(t))_{0\le t\le T}\) be an exogenous
\(\mathbb R^d\)-valued process, which we call the \emph{fundamental price}.
Under each trader's subjective probability \(\mathbb P_n^i\), we assume that
\(S^0\) is a square-integrable càdlàg martingale with respect to its natural
filtration and is independent of \(\Xi\), and hence of the announcement process
\(A\).
Trader \(i\) chooses a trading rate
\(u_n^i(t)\in\mathbb R^d\). Its inventory \(X_n^i(t)\in\mathbb R^d\) evolves as
\begin{equation}\label{eq:inventory_dynamics}
    dX_n^i(t)=u_n^i(t)\,dt,
    \qquad
    X_n^i(0)=x_{n,0}^i
\end{equation}
where we usually choose \(x_{n,0}^i=0\) so that the opportunists start flat.
We denote the average opportunist inventory by
\(\bar X_n(t):=\frac{1}{n}\sum_{j=1}^n X_n^j(t)\),
with initial value \(\bar x_{n,0}:=\bar X_n(0)\).
The aggregate opportunist trading rate and the aggregate
rate of all opportunists other than trader \(i\) are given by
\begin{equation}\label{eq:aggregate_opportunist_rates}
U_n(t):=\sum_{j=1}^n u_n^j(t),
\qquad
U_n^{-i}(t):=\sum_{j\ne i}u_n^j(t),
\end{equation}
respectively. The \emph{midprice} \(S(t)\in\mathbb{R}^d\) is given by
\begin{equation}\label{eq:midprice}
S(t)=S^0(t)+I(t).
\end{equation}
Here, \(I(t)\in\mathbb R^d\) is the \emph{transient impact} state with dynamics
\begin{equation}\label{eq:transient_impact}
dI(t)
=
\left[
-R I(t)
+\Gamma\left(\varpi_n U_n(t)+g(t)\right)
\right]dt,
\qquad
I(0)=\iota_0
\end{equation}
where \(R, \Gamma\in\mathbb{R}^{d\times d}\) and \(\varpi_n>0\).
We usually choose \(\iota_0=0\), so that the initial impact is zero.
The matrix \(R\) governs the decay of transient impact, while \(\Gamma\) maps
the net trading rate \(\varpi_nU_n+g\) into its contribution to the rate of
change of \(I\). The population
weight \(\varpi_n\) determines the contribution of each opportunist's rate
to transient impact. Typical choices are \(\varpi_n=\varpi\) for a finite-player
game with fixed player size and \(\varpi_n=\varpi/n\) for a model with mean-field scaling,
with \(\varpi>0\) independent of \(n\). We discuss population scaling further in
Subsection~\ref{subsec:parameter_choices_and_restrictions}.

For trader \(i\), trading generates the \emph{instantaneous impact}
\(\Lambda\bigl(u_n^i(t)+\chi_u\varpi_n U_n^{-i}(t)+\chi_g g(t)\bigr)\) where \(\chi_u,\chi_g\geq0\), so
the \emph{execution price} for player \(i\) is
\begin{equation}\label{eq:execution_price}
\widehat S^i(t)
=
S(t)+\Lambda\left(
u_n^i(t)+\chi_u\varpi_n U_n^{-i}(t)+\chi_g g(t)
\right).
\end{equation}
The parameters \(\chi_u\) and \(\chi_g\)
measure the execution-price loadings on the contemporaneous flows of the
other opportunists and the indexer, respectively.
The choice \(\varpi_n=\chi_u=\chi_g=1\) gives all opportunist and indexer
flows equal weight in the execution price, so opposing buy and sell rates
offset. The execution price then is
\[
\widehat S^i(t)
=
S(t)+\Lambda\bigl(U_n(t)+g(t)\bigr),
\]
which is the same for every trader. The scaling and interpretation of instantaneous impact are discussed further
later in Subsection~\ref{subsec:parameter_choices_and_restrictions}.

\begin{remark}[Symmetry and Diagonality of the Execution-Cost Matrix]
The symmetry assumption on \(\Lambda\) can be relaxed: the analysis in
Sections~\ref{sec:solution} and \ref{sec:mean_field_limit} extends to
nonsymmetric \(\Lambda\in\mathbb R^{d\times d}\) whose symmetric part
\(\Lambda_{\mathrm{s}}:=\frac{1}{2}(\Lambda+\Lambda^\top)\) is positive
definite.
However, the symmetric specification is also empirically motivated. In particular, as
summarized in Remark~2.1 of \citet{jaber2026}, the findings of
\citet{capponi_cont_2020}, \citet{cont_cucuringu_zhang_2023}, and
\citet{lecoz_mastromatteo_challet_benzaquen_2024} motivate the more restrictive
choice of a diagonal \(\Lambda\).
\end{remark}

\paragraph{Objective.} For a trading-rate profile
\(u_n=(u_n^1,\ldots,u_n^n)\), write
\(u_n^{-i}:=(u_n^j)_{j\ne i}\) and define trader \(i\)'s cost by
\begin{equation}\label{eq:overview_player_objective}
\begin{aligned}
J_n^i(u_n^i;u_n^{-i})
:=
\mathbb E_n^i\Bigg[
&\int_0^T \big(\widehat S^i(t)\big)^\top u_n^i(t)\,dt
-\left((S^0(T))^\top X_n^i(T)-(S^0(0))^\top x_{n,0}^i\right)\\
&+\int_0^T \frac{1}{2}(X_n^i(t))^\top Q X_n^i(t)\,dt
+\frac{1}{2}(X_n^i(T))^\top Q_T X_n^i(T)
\Bigg],
\end{aligned}
\end{equation}
where \(Q,Q_T\in\mathbb S_+^d\), and \(\mathbb E_n^i\) denotes expectation
under \(\mathbb P_n^i\). The first integral is the cumulative cash paid for
trading. From this amount we subtract the change in the fundamental value of
the trader's inventory. The final two terms penalize running and terminal
inventory, respectively. We typically choose the terminal penalty sufficiently
large that it effectively forces terminal inventories to be very close to
zero.

\begin{remark}[Fundamental-Price Terminal Valuation]
We value terminal inventory at \(S^0(T)\), rather than at the impacted midprice
\(S(T)\). This avoids crediting a trader with self-generated terminal price
impact that cannot be realized without selling and thereby reversing it. The
convention also underlies the a priori Riccati bounds in
Theorem~\ref{thm:global_riccati_solvability}, which guarantee full-horizon
existence of the closed-loop Nash equilibrium. Our proof of these bounds does
not extend to midprice marking, as explained later in
Remark~\ref{rem:impacted_terminal_mark_solvability}. In the related open-loop
games of \citet{NeumanVoss2023} and
\citet{campbell2026optimalexecutionntraders},
fundamental-price
marking preserves convexity of costs
making first-order conditions sufficient and supporting uniqueness. In practice,
however, we typically choose \(Q_T\) large enough to keep terminal inventories
close to zero, so the choice of terminal valuation is expected to have little
quantitative effect.
\end{remark}

\subsection{Parameter Restrictions and Population Scaling}%
\label{subsec:parameter_choices_and_restrictions}

In this subsection, we describe some assumptions we make on the market impact parameters and
also comment on the population weight \(\varpi_n\) in our model.

\paragraph{No Price Manipulation.} We choose the impact parameters \(\Gamma\) and \(R\) to rule out price
manipulation in the sense of \citet{huberman_stanzl}. Such manipulation means
profiting from price movements generated solely by one's own trades while
ending with the initial portfolio. To isolate the effect of a trader's own
trades, we assume that all other opportunists and the indexer are inactive and
that the initial transient impact is zero.
Formally, a deterministic trading rate
\(v\in L^2([0,T],\mathbb R^d)\) is a \emph{round trip} if
\(\int_0^T v(t)\,dt=0\). With \(\Lambda\in\mathbb S_{++}^d\) denoting the
execution-cost matrix in \eqref{eq:execution_price}, its total self-induced
cost is
\begin{equation}\label{eq:transient_self_impact_cost}
\mathcal C_{\mathrm{tot}}(v)
:=
\int_0^T v(t)^\top\Lambda v(t)\,dt+\varpi_n\mathcal C(v),\quad\mathcal C(v)
:=
\int_0^T v(t)^\top
\left(\int_0^t e^{-R(t-s)}\Gamma v(s)\,ds\right)dt.
\end{equation}
Absence of price manipulation requires \(\mathcal C_{\mathrm{tot}}(v)\geq 0\)
for every round trip. We impose the stronger condition
\(\mathcal C(v)\geq 0\) for every
\(v\in L^2([0,T],\mathbb R^d)\). Sufficient conditions for this requirement are
\begin{equation}\label{eq:transient_kernel_parameter_conditions}
R\in\mathbb S_+^d,
\qquad
\Gamma\in\mathbb S_+^d,
\qquad
R\Gamma=\Gamma R.
\end{equation}
See \citet{jaber2026}, \citet{alfonsi2016}, and
\citet{gatheral_2012} for conditions that preclude price
manipulation in more general transient-impact models.

\paragraph{Interpretation of the Population Scaling.}
The weight $\varpi_n$ represents each opportunist's mass. The model variables
$X_n^i$, $u_n^i$, and $J_n^i$ are quantities per unit of trader mass. Their
physical counterparts are
\(X_{n,\mathrm{phys}}^i:=\varpi_nX_n^i\),
\(u_{n,\mathrm{phys}}^i:=\varpi_nu_n^i\), and
\(J_{n,\mathrm{phys}}^i:=\varpi_nJ_n^i\).
Consequently, $dX_{n,\mathrm{phys}}^i(t)
=u_{n,\mathrm{phys}}^i(t)\,dt$ and
$\varpi_nU_n=\sum_{j=1}^n u_{n,\mathrm{phys}}^j$. The transient-impact
dynamics can therefore be written in physical units as
\[
dI(t)
=
\left[
-RI(t)
+
\Gamma\left(
\sum_{j=1}^n u_{n,\mathrm{phys}}^j(t)+g(t)
\right)
\right]dt.
\]
Thus, \(R\) and \(\Gamma\) are interpreted in physical units.
The interpretation of the private instantaneous and inventory costs is
different. In physical units, the execution price becomes
\[
\widehat S^i(t)
=
S(t)
+
\Lambda\left(
\frac{1}{\varpi_n}u_{n,\mathrm{phys}}^i(t)
+
\chi_u\sum_{j\ne i}u_{n,\mathrm{phys}}^j(t)
+
\chi_g g(t)
\right).
\]
Moreover, multiplying the objective by $\varpi_n$ and substituting
$X_n^i=X_{n,\mathrm{phys}}^i/\varpi_n$ and
$u_n^i=u_{n,\mathrm{phys}}^i/\varpi_n$ gives
\begin{align*}
J_{n,\mathrm{phys}}^i
&=
\mathbb E_n^i\Bigg[
\int_0^T
\bigl(\widehat S^i(t)\bigr)^\top
u_{n,\mathrm{phys}}^i(t)\,dt
-
\left(
(S^0(T))^\top X_{n,\mathrm{phys}}^i(T)
-(S^0(0))^\top X_{n,\mathrm{phys}}^i(0)
\right)
\\
&\quad
+
\frac{1}{2\varpi_n}\int_0^T
\bigl(X_{n,\mathrm{phys}}^i(t)\bigr)^\top
QX_{n,\mathrm{phys}}^i(t)\,dt
+
\frac{1}{2\varpi_n}
\bigl(X_{n,\mathrm{phys}}^i(T)\bigr)^\top
Q_TX_{n,\mathrm{phys}}^i(T)
\Bigg].
\end{align*}
Hence a trader of mass $\varpi_n$ faces coefficients
$\Lambda/\varpi_n$, $Q/\varpi_n$, and $Q_T/\varpi_n$
on its physical trading rate and inventory. The coefficient
$\Lambda/\varpi_n$ makes a fixed absolute trading rate more expensive as the
trader becomes smaller, as would be the case if execution capacity were
proportional to trader mass. For the inventory terms, $Q/\varpi_n$ and $Q_T/\varpi_n$ likewise
make a fixed absolute position more expensive which amounts to greater
effective risk aversion in physical units and is reasonable under the assumption
that absolute risk-bearing capacity is proportional to trader mass.
If instead one requires the coefficients on absolute physical rates and
inventories to remain constant, one must take
$\Lambda_n=\varpi_n\Lambda_{\mathrm{phys}}$,
$Q_n=\varpi_nQ_{\mathrm{phys}}$, and
$Q_{T,n}=\varpi_nQ_{T,\mathrm{phys}}$. Under $\varpi_n=\varpi/n$,
these coefficients vanish as $n\to\infty$. In particular, the representative-agent Hamiltonian loses its
strictly convex quadratic term in the trading rate, while the running and
terminal inventory-restoring terms disappear. The limiting mean-field control
problem is therefore generally singular.
We therefore keep $\Lambda$, $Q$, and $Q_T$ fixed per unit of trader mass,
preserving strict convexity in the trading rate and the inventory penalties
in the mean-field limit.

As discussed in Remark~\ref{rem:twap_approximation}, modeling the indexer's execution as a TWAP
is a simplification. In practice, index funds concentrate much of their
rebalance trading in the closing auction, where opposing buy and sell orders
offset in the auction imbalance. To approximate the offsetting of buy and sell 
orders in the closing auction within our continuous-trading game we therefore typically
set $\chi_u=\chi_g=1$.
In the atomic specification \(\varpi_n=1\), this choice gives
$\widehat S^i=S+\Lambda(U_n+g)$, so all traders face the same execution price
and opposite physical trading rates in the same asset cancel.
For fixed $\varpi_n=\varpi\ne1$, opposing contributions can still offset, but
the trader's own physical rate has coefficient $\Lambda/\varpi$, so execution
prices generally differ.
Under the mean-field scaling \(\varpi_n=\varpi/n\), setting
$\chi_u=\chi_g=1$ gives the limiting execution price
\(\widehat S^i(t)=S(t)+\Lambda\bigl(u^i(t)+\varpi\bar u(t)+g(t)\bigr)\),
where \(\bar u(t)=\lim_{n\to\infty}\frac{1}{n}U_n(t)\).
Unlike in the atomic specification \(\varpi_n=1\), execution prices in the
mean-field model generally differ across traders, even when \(\varpi=1\).
Each trader's own rate \(u^i\) enters per unit of trader mass, whereas
\(\varpi\bar u+g\) is the aggregate physical flow of the other opportunists
and the indexer. This formulation avoids the singular mean-field limit
discussed earlier and retains offsetting of opposing contributions, with
\(\varpi\) controlling the relative weight of aggregate opportunist flow.

\paragraph{Relative Weights in the Execution Price.}
For each finite-player game, we impose $\chi_u\varpi_n\leq1$, so each other
opportunist's rate enters \eqref{eq:execution_price} with no greater weight
than the trader's own rate. Usually one chooses \(\chi_u\in\{0, 1\}\)
and either \(\varpi_n=1\) or \(\varpi_n=\varpi/n\). The restriction \(\chi_u\varpi_n\leq1\)
thus only becomes relevant in the mean-field scaling setting with \(\chi_u=1\) where it is equivalent to
\(n\geq\varpi\).

We now collect all our standing assumptions:

\begin{assumption}[Standing Assumptions]\label{ass:standing_model}
Under every \(\mathbb P_n^i\), \(S^0\) is a square-integrable càdlàg martingale
independent of \(\Xi\). Moreover, \(\Lambda\in\mathbb S_{++}^d\),
\(Q,Q_T\in\mathbb S_+^d\), \(\mathcal C(v)\geq0\) for every deterministic
\(v\in L^2([0,T],\mathbb R^d)\), \(\chi_u,\chi_g\geq0\), and
\(\varpi_n>0\). In each finite-player game, assume
\(\chi_u\varpi_n\leq1\).
\end{assumption}
\subsection{Admissible Policies and Equilibrium Concepts}

We next define the admissible controls, focusing on \emph{closed-loop policies} which 
map the history of the processes observed by a
player up to time \(t\) to that player's trading rate. In our model, trader
\(i\) observes their own inventory \(X_n^i\), the average inventory \(\bar X_n\),
the transient impact \(I\), the announcement process \(A\), and the
fundamental price \(S^0\). Importantly, traders do
not observe the individual inventories of the other traders.

\begin{definition}[Admissible Closed-Loop Policies]
\label{def:admissible_feedback_policies}
A closed-loop policy for trader \(i\) is a Borel measurable,
non-anticipative map
\begin{equation}\label{eq:finite_feedback_policy_maps}
\Phi_n^i:
[0,T]\times D([0,T];\mathbb R^d)^5\times\Delta_m^2
\longrightarrow\mathbb R^d,
\end{equation}
where \(D([0,T];\mathbb R^d)\) denotes the space of \(\mathbb R^d\)-valued
càdlàg paths, equipped with the Skorokhod topology. Nonanticipativity means
that for every \(t\in[0,T]\) the map's value at
time \(t\) depends on each path only through its restriction to \([0,t]\).
A profile \(\Phi_n=(\Phi_n^1,\ldots,\Phi_n^n)\) induces the trading rates
\begin{equation}\label{eq:finite_feedback_induced_rate}
u_n^i(t)
:=
\Phi_n^i\left(
t,X_n^i,\bar X_n,I,A,S^0;\bar\pi_n,\pi_n^i
\right),
\qquad i=1,\ldots,n.
\end{equation}
The profile \(\Phi_n\) is admissible if the inventory and transient-impact
equations \eqref{eq:inventory_dynamics} and \eqref{eq:transient_impact},
with trading rates induced by \(\Phi_n\), admit a unique strong solution
and the induced rates are square-integrable under every trader's subjective
probability:
\[
\max_{i=1,\ldots, n}\mathbb E_n^i\left[
\int_0^T \sum_{j=1}^n \|u_n^j(t)\|^2\,dt
\right]
<\infty.
\]
The set of admissible closed-loop policy profiles is denoted by \(\mathcal U_n\).
\end{definition}

For an admissible profile \(\Phi_n\), we write
\(J_n^i(\Phi_n^i;\Phi_n^{-i})\) for the objective in
\eqref{eq:overview_player_objective} evaluated at the rates and states induced by $\Phi_n$.
An admissible deviation by trader \(i\) is a policy \(\Psi_n^i\) such that
the modified profile remains admissible.

\begin{definition}[Nash Equilibrium]
\label{def:nash_equilibrium}
An admissible policy profile \(\Phi_n^*\) is a closed-loop Nash equilibrium
if, for every trader \(i\) and every admissible deviation
\(\Psi_n^i\),
\begin{equation}\label{eq:ne_definition}
J_n^i(\Phi_n^{i,*};\Phi_n^{-i,*})
\leq
J_n^i(\Psi_n^i;\Phi_n^{-i,*}).
\end{equation}
\end{definition}

The closed-loop Nash equilibria in Definition~\ref{def:nash_equilibrium}
satisfy optimality from the given initial time and state. They need not
remain equilibria when restarted from an alternative history. The stronger
concept of subgame perfection requires optimality in every continuation game.
For any prescribed histories \((X_n,I,A,S^0)_{[0,t]}\), the continuation
objective corresponding to \eqref{eq:overview_player_objective} is\footnote{%
We use a regular conditional law of \((S^0,\Xi)\) under \(\mathbb P_n^i\),
given \(S^0|_{[0,t]}\) and \(A(t)\). We choose a version that,
for every prescribed price history and announcement state, assigns probability one
to paths extending that history and makes \(S^0\) a square-integrable
martingale on \([t,T]\) in the filtration generated by \((S^0,A)\).
Independently of the price path, \(\Xi\) has probabilities \(\pi_n^{i,k}\)
before the announcement and equals \(A(t)\) afterward, including scenarios
with zero prior probability.}
\begin{equation}\label{eq:continuation_objective}
\begin{aligned}
&J_n^i\left(\Phi_n^i;\Phi_n^{-i}\mid t,(X_n,I,A,S^0)_{[0,t]}\right)
\\
&\quad:=\mathbb E_n^i\Bigg[
\int_t^T \bigl(\widehat S^i(\tau)\bigr)^\top u_n^i(\tau)\,d\tau
-\left((S^0(T))^\top X_n^i(T)-(S^0(t))^\top X_n^i(t)\right)
\\
&\qquad+\int_t^T \frac{1}{2}(X_n^i(\tau))^\top QX_n^i(\tau)\,d\tau
+\frac{1}{2}(X_n^i(T))^\top Q_TX_n^i(T)
\,\Bigm|\, S^0|_{[0,t]},\ A(t)
\Bigg].
\end{aligned}
\end{equation}
The rates and controlled states on \([t,T]\) are induced by the policy maps
\(\Phi_n\), starting from the states at time \(t\) in the prescribed histories.
The maps are evaluated on the prescribed histories extended by the observations
generated after \(t\).
Admissibility and unilateral deviations are defined as before
on \([t,T]\), under every trader's subjective continuation law.

\begin{definition}[Subgame-Perfect Nash Equilibrium]
\label{def:subgame_perfect_nash_equilibrium}
An admissible profile \(\Phi_n^*\in\mathcal U_n\) is a subgame-perfect
closed-loop Nash equilibrium if, for every \(t\in[0,T]\) and every prescribed history
\((X_n,I,A,S^0)_{[0,t]}\), its restriction to \([t,T]\) is admissible
and, for every trader \(i\) and every admissible unilateral continuation
deviation \(\Psi_n^i\),
\begin{equation}\label{eq:subgame_perfect_ne}
J_n^i\left(
\Phi_n^{i,*};\Phi_n^{-i,*}
\mid t,(X_n,I,A,S^0)_{[0,t]}
\right)
\leq
J_n^i\left(
\Psi_n^i;\Phi_n^{-i,*}
\mid t,(X_n,I,A,S^0)_{[0,t]}
\right).
\end{equation}
The left-hand side of \eqref{eq:subgame_perfect_ne} is called trader \(i\)'s
value of the continuation game and is denoted by
\(V_n^i(t,(X_n,I,A,S^0)_{[0,t]})\).
\end{definition}

In particular, subgame perfection requires continuation optimality after any
announced constitution \(\xi^k\), even when trader \(i\) assigned it zero prior
probability, \(\pi_n^{i,k}=0\).

\begin{remark}[Open-Loop, Closed-Loop, and Subgame-Perfect Equilibria]
Open-loop controls, commonly used in market impact games, are induced by
policies \(\Phi_n^i\) that depend only on the histories of the exogenous
processes \((S^0,A)\). Admissibility and Nash equilibria are defined as
before, with profiles and deviations restricted to this class.
Since neither \(S^0\) nor \(A\) is affected by the traders' actions, the
realized trading rates of the nondeviating traders remain unchanged when
another trader deviates.
In our closed-loop formulation, by contrast, the nondeviating traders retain
the same policy maps, but a deviation can change the controlled-state
histories fed into them and hence their realized trading rates.

Nevertheless, the policy maps of an open-loop Nash equilibrium also form
a closed-loop Nash equilibrium in Definition~\ref{def:nash_equilibrium}
from the specified initial condition. Against these fixed plans, the unique
strong solution makes any
admissible closed-loop deviation's induced rate a nonanticipative function
of \((S^0,A)\), which an open-loop deviation can reproduce with the same
states and cost. However, this initial-condition guarantee does not establish
the stronger property of subgame perfection.

The Markov strategies we construct later in Section~\ref{sec:solution}
form a subgame-perfect Nash equilibrium. In contrast, the rates extracted
from the mean-field game in Section~\ref{sec:mean_field_limit} form
approximate closed-loop Nash
equilibria from zero initial inventories and impact, but need not be optimal
in every continuation game. However, implementing these rates requires
neither observation of the average inventory nor knowledge of the exact
number of players; see Remark~\ref{rem:average_inventory_proxy}.
\end{remark}

\begin{remark}[A Market-Data Proxy for Average Inventory]
\label{rem:average_inventory_proxy}
In the definition of the admissible policies we treated \(\bar X_n\)
as observed by traders. In practice, however, \(\bar X_n\) is not directly
reported. Nevertheless, a reasonable proxy can often still be
constructed from market data if one is willing to make certain assumptions.
We briefly describe a crude procedure to compute such a proxy:
Suppose that the opportunists trade exclusively through liquidity-taking
market orders. Their aggregate inventory change can then be approximated by cumulative signed
liquidity-taking order flow. This flow can be estimated from consolidated
transaction and quotation data using the heuristic of \citet{lee_ready_1991},
which classifies trades above the prevailing quote midpoint as buyer-initiated,
trades below it as seller-initiated, and midpoint trades using a tick test.
The proxy relies on opportunists trading exclusively through liquidity-taking
market orders and on there being no other liquidity takers.
Related empirical work estimates the flows of particular investor
groups by combining signed
transactions with observable transaction characteristics or lower-frequency
holdings data
\citep{campbell_ramadorai_schwartz_2009,boehmer_jones_zhang_zhang_2021}.
\end{remark}

\section{Finite-Player Equilibrium}\label{sec:solution}

We derive a subgame-perfect closed-loop Nash equilibrium consisting of
Markov policies that are affine in the state.

\subsection{Reduction and HJB Verification}
\label{sec:fundamental_price_reduction}

We first reformulate the continuation objective \eqref{eq:continuation_objective}
to remove its explicit dependence on the fundamental price.
For prescribed histories \((X_n,I,A,S^0)_{[0,t]}\) and an admissible
continuation profile \(\Phi_n\), integration by parts and the martingale
property of \(S^0\) under the continuation law give the following
representation of trader \(i\)'s objective:
\begin{align}
J_n^i\left(
\Phi_n^i;\Phi_n^{-i}
\mid t,(X_n,I,A,S^0)_{[0,t]}
\right)
&=\mathbb E_n^i\Bigg[
\int_t^T
\mathcal L\bigl(
X_n^i(\tau),u_n^i(\tau),
I(\tau),\varpi_nU_n^{-i}(\tau),g(\tau)
\bigr)\,d\tau\notag
\\
&\quad
+\frac{1}{2}(X_n^i(T))^\top Q_TX_n^i(T)
\,\Bigm|\,
S^0|_{[0,t]},\ A(t)
\Bigg]
\label{eq:reduced_continuation_objective}
\end{align}
where
\begin{equation}\label{eq:reduced_running_cost}
\mathcal L(x,u,\iota,\mu,g)
:=
\bigl(\iota+\Lambda(u+\chi_u\mu+\chi_g g)\bigr)^\top u
+\frac{1}{2}x^\top Q x.
\end{equation}

The running and terminal costs in \eqref{eq:reduced_continuation_objective}
contain no explicit dependence on \(S^0\), and neither do the inventory
and impact dynamics. This suggests looking for a Nash equilibrium whose
policies ignore \(S^0\). For such a profile, the expectation in
\eqref{eq:reduced_continuation_objective} is only over \(\Xi\).
However, verifying a Nash equilibrium still requires allowing deviating
traders to use \(S^0\) as an external randomization device, although the
verification theorem below shows that this randomization cannot lower a
trader's cost against the constructed equilibrium.

We seek subgame-perfect equilibria with Markov feedbacks of the form
\begin{equation}\label{eq:markov_feedback}
\phi_n^{i,k}\bigl(
t,X_n^i(t)-\bar X_n(t),\bar X_n(t),I(t);\bar\pi_n,\pi_n^i
\bigr).
\end{equation}
Here \(k=0\) denotes the pre-announcement interval
\(t\in[0,T_{\mathrm{ann}})\), and \(k=1,\ldots,m\) the announced scenario
for \(t\in[T_{\mathrm{ann}},T]\) with \(A(t)=\xi^k\).
Write \(\phi_n^i:=(\phi_n^{i,k})_{k=0}^m\) and
\(\phi_n:=(\phi_n^1,\ldots,\phi_n^n)\), identified with their induced
closed-loop policies on arbitrary input histories.
We collect the three state inputs of the feedbacks into trader \(i\)'s
\emph{reduced state}
\(Z_n^i(t):=z^i(X_n^i(t),\bar X_n(t),I(t))\), where
\(z^i(x^i,\bar x,\iota):=
\bigl((x^i-\bar x)^\top,\bar x^\top,\iota^\top\bigr)^\top
\in\mathbb R^{3d}\).
We call \(\delta_n^i(t):=X_n^i(t)-\bar X_n(t)\) trader \(i\)'s
\emph{deviation state} and \((\bar X_n(t),I(t))\) the \emph{common state}.
We write \(\phi_n^{i,k}(t,z^i;\bar\pi_n,\pi_n^i)\) for the same feedback
in stacked coordinates.
For \(x_n=(x_n^1,\ldots,x_n^n)\), write
\(\bar x_n:=n^{-1}\sum_{j=1}^n x_n^j\) and
for the Markov feedbacks in \eqref{eq:markov_feedback}, write
\(\phi_n^{-i,k}:=(\phi_n^{j,k})_{j\ne i}\). Define the aggregate feedback rate of trader
\(i\)'s opponents in each post-announcement scenario \(k=1,\ldots,m\) by
\(
U_n^{-i,k}(t,x_n,\iota)
:=\sum_{j\ne i}\phi_n^{j,k}
\left(t,z^j(x_n^j,\bar x_n,\iota);\bar\pi_n,\pi_n^j\right),
\)
and define \(U_n^{-i,0}\) analogously using the pre-announcement feedbacks
\(\phi_n^{j,0}\).
A full inventory vector \(x_n\) is \emph{compatible}
with a generic reduced state
\(z^i=((\delta^i)^\top,\bar x^\top,\iota^\top)^\top\) if
\(x_n^i=\delta^i+\bar x\) and \(\bar x_n=\bar x\). 

In each post-announcement
scenario \(k=1,\ldots,m\), the dynamics of \(Z_n^i\) are
\[
\dot Z_n^i(t)
=\beta_n^{i,k}\left(t,X_n(t),I(t),u_n^i(t)\right),
\]
where, for a full inventory vector \(x_n\), an impact state \(\iota\), and a
candidate trading rate \(v\in\mathbb R^d\), the drift is
\begin{equation}\label{eq:reduced_state_drift}
\beta_n^{i,k}(t,x_n,\iota,v)
:=\begin{bmatrix}
\frac{n-1}{n}v-\frac{1}{n}U_n^{-i,k}(t,x_n,\iota)\\[2pt]
\frac{1}{n}\bigl(v+U_n^{-i,k}(t,x_n,\iota)\bigr)\\[2pt]
-R\iota+\Gamma\bigl(
\varpi_n(v+U_n^{-i,k}(t,x_n,\iota))+g^k(t)
\bigr)
\end{bmatrix}.
\end{equation}
The same dynamics hold before the announcement with \(k=0\), under the
convention \(g^0\equiv0\).
Against fixed opponents' feedbacks \(\phi_n^{-i,k}\), the running cost
and reduced-state dynamics define trader \(i\)'s best-response Hamiltonian
for \(p\in\mathbb R^{3d}\):
\begin{equation}\label{eq:stitched_hamiltonian}
\begin{aligned}
\mathscr H_n^{i,k}(t,x_n,\iota,v,p;\phi_n^{-i,k})
:={}&
\mathcal L\left(
x_n^i,v,\iota,
\varpi_nU_n^{-i,k}(t,x_n,\iota),
g^k(t)
\right)
\\
&+p^\top\beta_n^{i,k}(t,x_n,\iota,v).
\end{aligned}
\end{equation}
Dynamic programming gives an HJB equation for each trader's best-response
problem. Requiring every trader's feedback to minimize their best-response
Hamiltonian over their own trading rate yields a coupled HJB system.
The post-announcement value functions
have terminal conditions given by the terminal inventory penalty. At the
announcement, each trader's pre-announcement value matches the
prior-weighted average of its post-announcement values.
The following verification theorem makes this precise.

\begin{theorem}[HJB Verification]
\label{thm:stitched_hjb_verification}
Fix a Markov feedback profile
\(\phi_n\) in the parametrization of \eqref{eq:markov_feedback} that is
admissible in every continuation game. Suppose that, for every trader \(i\),
there exist candidate functions
\(V_n^{i,0}:[0,T_{\mathrm{ann}}]\times\mathbb R^{3d}\to\mathbb R\) and
\(V_n^{i,k}:[T_{\mathrm{ann}},T]\times\mathbb R^{3d}\to\mathbb R\),
\(k=1,\ldots,m\), such that every \(V_n^{i,k}\), \(k=0,\ldots,m\), is
continuous on its domain and \(C^1\) except at the switching times
\(T_{\mathrm{TWAP}}\) and \(T_{\mathrm{impl}}\) for \(k\geq1\).
For every trader
\(i\), assume the HJB and minimizing-feedback conditions below hold on the
respective time intervals, except at the switching times for \(k\geq1\), for every
\(z^i=((\delta^i)^\top,\bar x^\top,\iota^\top)^\top\in\mathbb R^{3d}\) and
every \(x_n\) compatible with \(z^i\), and the boundary conditions hold for
every \(z^i\in\mathbb R^{3d}\):
\begin{subequations}\label{eq:stitched_hjb_system}
\begin{equation}\label{eq:stitched_hjb}
0
=
\partial_tV_n^{i,k}(t,z^i)
+\min_{v\in\mathbb R^d}
\mathscr H_n^{i,k}\left(
t,x_n,\iota,v,
\nabla V_n^{i,k}(t,z^i);\phi_n^{-i,k}
\right),
\end{equation}
\begin{align}
\phi_n^{i,k}
\left(t,z^i;\bar\pi_n,\pi_n^i\right)
&=
\operatorname*{argmin}_{v\in\mathbb R^d}
\mathscr H_n^{i,k}\left(
t,x_n,\iota,v,
\nabla V_n^{i,k}(t,z^i);\phi_n^{-i,k}
\right),
\label{eq:stitched_hjb_argmin}\\
V_n^{i,k}(T,z^i)
&=
\frac{1}{2}(\delta^i+\bar x)^\top Q_T(\delta^i+\bar x),
\qquad k=1,\ldots,m,
\label{eq:stitched_hjb_terminal}\\
V_n^{i,0}(T_{\mathrm{ann}},z^i)
&=
\sum_{k=1}^m\pi_n^{i,k}
V_n^{i,k}(T_{\mathrm{ann}},z^i).
\label{eq:stitched_hjb_value_matching}
\end{align}
\end{subequations}
Then \(\phi_n\) is a Markov subgame-perfect Nash equilibrium. For any
continuation history up to \(s\), with \(X_n(s)=x_n\), \(I(s)=\iota\),
and \(A(s)=a\), its continuation value is
\[
V_n^i\left(s,(X_n,I,A,S^0)_{[0,s]}\right)
=
\begin{cases}
V_n^{i,0}\left(s,z^i(x_n^i,\bar x_n,\iota)\right),&a=0,\\
V_n^{i,k}\left(s,z^i(x_n^i,\bar x_n,\iota)\right),&a=\xi^k.
\end{cases}
\]
\end{theorem}

Since the Hessian with
respect to \(v\) is \(2\Lambda\succ0\), the minimizer is unique.
The traders' information restriction is encoded in both the candidate
functions and the HJB system. For each \(k=0,\ldots,m\), \(V_n^{i,k}\) and
\(\phi_n^{i,k}\) depend on the current state only through the reduced state
\(z^i\). The minimum and argmin in
\eqref{eq:stitched_hjb}--\eqref{eq:stitched_hjb_argmin} range over all
\(v\in\mathbb R^d\) and because the Hamiltonian depends on \((x_n,\iota)\), its
minimizer could a priori be a general function of these variables.
Equation~\eqref{eq:stitched_hjb_argmin} instead requires the minimizer to be
the same for every \(x_n\) compatible with a given \(z^i\), so it depends on
the current state only through \(z^i\). Equation~\eqref{eq:stitched_hjb}
imposes the analogous restriction on the minimum value.
Note that the theorem also shows that \(\phi_n\) remains a
subgame-perfect Nash equilibrium under the larger information structure where
each trader observes every other
trader's inventory.

\subsection{Quadratic Equilibrium Construction}
\label{sec:fixed_interval_lq_game}
\label{sec:full_procedure}

We now explicitly solve the coupled HJB system in
Theorem~\ref{thm:stitched_hjb_verification} using the quadratic ansatz
\begin{equation}\label{eq:quadratic_value_ansatz}
\begin{aligned}
V_n^{i,0}(t,z^i)
&=
\frac{1}{2}(z^i)^\top H_n(t)z^i
+(r_n^{i,0}(t))^\top z^i+c_n^{i,0}(t),\\
V_n^{i,k}(t,z^i)
&=
\frac{1}{2}(z^i)^\top H_n(t)z^i
+(r_n^k(t))^\top z^i+c_n^k(t),
\qquad k=1,\ldots,m.
\end{aligned}
\end{equation}
In this ansatz, \(H_n\) is common across traders and post-announcement
scenarios and is also the pre-announcement quadratic coefficient. The pairs
\((r_n^k,c_n^k)\) are scenario-specific after the announcement, whereas
\((r_n^{i,0},c_n^{i,0})\) are trader-specific before it.

For compact notation, we rewrite the reduced-state dynamics in
\eqref{eq:reduced_state_drift} in matrix form
\begin{equation}\label{eq:fixed_interval_reduced_dynamics}
\dot Z_n^i=A_zZ_n^i+B_nu_n^i+B_n^{-}U_n^{-i}+B_gg,
\end{equation}
where
\[
A_z:=\begin{bmatrix}0&0&0\\0&0&0\\0&0&-R\end{bmatrix},
\quad
B_n:=\begin{bmatrix}\frac{n-1}{n}\operatorname{Id}_d\\
\frac{1}{n}\operatorname{Id}_d\\ \varpi_n\Gamma\end{bmatrix},
\quad
B_n^{-}:=\begin{bmatrix}-\frac{1}{n}\operatorname{Id}_d\\
\frac{1}{n}\operatorname{Id}_d\\ \varpi_n\Gamma\end{bmatrix},
\quad
B_g:=\begin{bmatrix}0\\0\\\Gamma\end{bmatrix}.
\]
With
\(E_x:=[\operatorname{Id}_d\ \operatorname{Id}_d\ 0]\) and
\(E_\iota:=[0\ 0\ \operatorname{Id}_d]\), one has
\(X_n^i=E_xZ_n^i\) and \(I=E_\iota Z_n^i\). Consequently, the inventory
penalty in the running cost \(\mathcal L\) from
\eqref{eq:reduced_running_cost} and the terminal inventory penalty in
\eqref{eq:reduced_continuation_objective} are represented in the reduced state by
\(Q_z\) and \(H_T\), respectively, where
\begin{equation}\label{eq:fixed_interval_final_boundary}
Q_z:=E_x^\top QE_x,
\qquad
H_T:=E_x^\top Q_TE_x.
\end{equation}

Additionally, we introduce some notation required to compactly state the
ODE system satisfied by the coefficients in \eqref{eq:quadratic_value_ansatz}.
Set $\vartheta_n:=\chi_u\varpi_n$, $a_n:=2-\vartheta_n$, $d_n:=2+\vartheta_n(n-1)$.
Assumption~\ref{ass:standing_model} gives \(0\leq\vartheta_n\leq1\), so
\(1\leq a_n\leq2\) and \(d_n\geq2\). Define
\begin{equation}\label{eq:fixed_interval_D_def}
D_n:=\operatorname{diag}\left(
\tfrac{1}{a_n}\operatorname{Id}_d,
\tfrac{1}{d_n}\operatorname{Id}_d,
\tfrac{1}{d_n}\operatorname{Id}_d\right),
\end{equation}
\begin{equation}\label{eq:fixed_interval_F_def}
F_n:=\operatorname{diag}\left(
\tfrac{1}{a_n}\operatorname{Id}_d,
-\tfrac{n-1}{d_n}\operatorname{Id}_d,
-\tfrac{n-1}{d_n}\operatorname{Id}_d\right)
\end{equation}
and for \(H\in\mathbb S^{3d}\)
\begin{equation}\label{eq:fixed_interval_Theta_def}
\Theta_n(H):=E_\iota+B_n^\top H,
\end{equation}
\begin{equation}\label{eq:fixed_interval_M_def}
M_n(H):=A_z+B_n^{-}\Lambda^{-1}\Theta_n(H)F_n,
\end{equation}
and 
\begin{subequations}\label{eq:pre_ann_matrix_maps_def}
\begin{align}
\mathsf A_n^{\mathrm c}(H)
&:=M_n(H)^\top-\tfrac{1}{d_n}
\bigl((n-1)HB_n^{-}+2D_n^\top\Theta_n(H)^\top\bigr)
\Lambda^{-1}B_n^\top,\\
\mathsf A_n^{\mathrm d}(H)
&:=M_n(H)^\top+\tfrac{1}{a_n}
\bigl(HB_n^{-}-2D_n^\top\Theta_n(H)^\top\bigr)
\Lambda^{-1}B_n^\top.
\end{align}
\end{subequations}

We now state the ODE system for the coefficients in
\eqref{eq:quadratic_value_ansatz}.
Theorem~\ref{thm:full_horizon_closed_loop_equilibrium} below then verifies
that these indeed give the value functions and also gives the
corresponding equilibrium feedbacks. Time arguments are suppressed where no
ambiguity arises.

\begin{itemize}
\item \textbf{Quadratic coefficient.}
On \([0,T]\), \(H_n\) solves the non-standard Riccati ODE
\begin{equation}\label{eq:fixed_interval_riccati_P}
\begin{aligned}
\dot H_n={}&-Q_z-M_n(H_n)^\top H_n-H_nM_n(H_n)\\[-2pt]
&+2D_n^\top\Theta_n(H_n)^\top\Lambda^{-1}\Theta_n(H_n)D_n,
\\[-2pt]
H_n(T)={}&H_T.
\end{aligned}
\end{equation}

\item \textbf{Linear coefficients.}
\begin{subequations}\label{eq:full_horizon_linear_coefficients}
For \(k=1,\ldots,m\), \(r_n^k\) is continuous on
\([T_{\mathrm{ann}},T]\) and satisfies the following linear ODE away from
\(T_{\mathrm{TWAP}}\) and \(T_{\mathrm{impl}}\):
\begin{equation}\label{eq:post_ann_common_r}
\begin{aligned}
\dot r_n^k={}&-\mathsf A_n^{\mathrm c}(H_n)r_n^k-H_nB_gg^k
\\[-2pt]
&+\tfrac{\chi_g}{d_n}
\bigl((n-1)H_nB_n^{-}+2D_n^\top\Theta_n(H_n)^\top\bigr)g^k,
\\[-2pt]
r_n^k(T)={}&0,
\end{aligned}
\end{equation}
On \([0,T_{\mathrm{ann}}]\), 
\(r_n^0:=\tfrac{1}{n}\sum_{j=1}^n r_n^{j,0}\)
solves the linear ODE
\begin{equation}\label{eq:pre_ann_common_r_ode}
\dot r_n^0
=-\mathsf A_n^{\mathrm c}(H_n)r_n^0,
\qquad
r_n^0(T_{\mathrm{ann}})
=\sum_{k=1}^m\bar\pi_n^kr_n^k(T_{\mathrm{ann}}).
\end{equation}
Given \(r_n^0\), trader \(i\)'s coefficient \(r_n^{i,0}\) solves
\begin{equation}\label{eq:pre_ann_individual_r_ode}
\begin{aligned}
\dot r_n^{i,0}
={}&-\mathsf A_n^{\mathrm d}(H_n)r_n^{i,0}
+\bigl(\mathsf A_n^{\mathrm d}(H_n)-\mathsf A_n^{\mathrm c}(H_n)\bigr)r_n^0,
\\[-2pt]
r_n^{i,0}(T_{\mathrm{ann}})
={}&\sum_{k=1}^m\pi_n^{i,k}r_n^k(T_{\mathrm{ann}}),
\qquad i=1,\ldots,n.
\end{aligned}
\end{equation}
Thus, trader \(i\) can compute \(r_n^{i,0}\) from \((\bar\pi_n,\pi_n^i)\) by first solving the ODE
for $r^0_n$ and then the one for $r^{i, 0}_n$.
\end{subequations}

\item \textbf{Constant coefficients.}
\begin{subequations}\label{eq:full_horizon_constant_coefficients}
For \(k=1,\ldots,m\), \(c_n^k\) is continuous on
\([T_{\mathrm{ann}},T]\) and satisfies the following ODE away from
\(T_{\mathrm{TWAP}}\) and \(T_{\mathrm{impl}}\):
\begin{equation}\label{eq:post_ann_common_alpha}
\begin{aligned}
\dot c_n^k={}&-\left(
B_gg^k-\tfrac{n-1}{d_n}B_n^{-}\Lambda^{-1}
\bigl(B_n^\top r_n^k+\chi_g\Lambda g^k\bigr)
\right)^\top r_n^k
\\[-2pt]
&+\tfrac{1}{d_n^2}\bigl(B_n^\top r_n^k+\chi_g\Lambda g^k\bigr)^\top
\Lambda^{-1}\bigl(B_n^\top r_n^k+\chi_g\Lambda g^k\bigr),
\\[-2pt]
c_n^k(T)={}&0,
\end{aligned}
\end{equation}
On \([0,T_{\mathrm{ann}}]\), \(i=1,\ldots,n\), \(c_n^{i,0}\) solves
\begin{equation}\label{eq:pre_ann_constant_ode}
\begin{aligned}
\dot c_n^{i,0}
={}&-\left[
B_n^{-}\Lambda^{-1}
\left(
\tfrac{1}{a_n}B_n^\top r_n^{i,0}
-\tfrac{2n}{a_nd_n}B_n^\top r_n^0
\right)
\right]^\top r_n^{i,0}
\\[-2pt]
&+\left(
\tfrac{1}{a_n}B_n^\top r_n^{i,0}
-\tfrac{n\vartheta_n}{a_nd_n}B_n^\top r_n^0
\right)^\top
\Lambda^{-1}
\left(
\tfrac{1}{a_n}B_n^\top r_n^{i,0}
-\tfrac{n\vartheta_n}{a_nd_n}B_n^\top r_n^0
\right),
\\[-2pt]
c_n^{i,0}(T_{\mathrm{ann}})
={}&\sum_{k=1}^m\pi_n^{i,k}c_n^k(T_{\mathrm{ann}}).
\end{aligned}
\end{equation}
\end{subequations}
\end{itemize}

\begin{theorem}[Explicit Quadratic Equilibrium]
\label{thm:full_horizon_closed_loop_equilibrium}
Suppose that the Riccati terminal-value problem
\eqref{eq:fixed_interval_riccati_P} admits a solution
\(H_n\in C^1([0,T];\mathbb S^{3d})\). Then the coefficient equations
\eqref{eq:full_horizon_linear_coefficients} and
\eqref{eq:full_horizon_constant_coefficients} admit unique finite solutions.
Define the pre-announcement
feedback on \([0,T_{\mathrm{ann}}]\) and, for \(k=1,\ldots,m\), the
scenario-\(k\) feedback on \([T_{\mathrm{ann}},T]\) by
\begin{subequations}\label{eq:full_horizon_markov_feedback}
\begin{align}
\phi_n^{i,0}(t,z^i;\bar\pi_n,\pi_n^i)
={}&-\Lambda^{-1}\Bigl[
\Theta_n(H_n(t))D_nz^i
+\tfrac{1}{a_n}B_n^\top r_n^{i,0}(t)
-\tfrac{n\vartheta_n}{a_nd_n}B_n^\top r_n^0(t)
\Bigr],\label{eq:full_horizon_markov_feedback_pre}\\[4pt]
\phi_n^{i,k}(t,z^i;\bar\pi_n,\pi_n^i)
={}&-\Lambda^{-1}\left[
\Theta_n(H_n(t))D_nz^i
+\tfrac{1}{d_n}
B_n^\top r_n^k(t)+\tfrac{1}{d_n}\chi_g\Lambda g^k(t)
\right].\label{eq:full_horizon_markov_feedback_post}
\end{align}
\end{subequations}
The value functions in \eqref{eq:quadratic_value_ansatz} and the feedbacks in
\eqref{eq:full_horizon_markov_feedback} satisfy
\eqref{eq:stitched_hjb_system}. The Markov feedback family \(\phi_n\) defined
by \eqref{eq:full_horizon_markov_feedback} is admissible in every continuation
game. Consequently, \(\phi_n\) is a subgame-perfect Nash equilibrium.
\end{theorem}

Note that the Riccati equation has dimension independent of the number of traders. This
scalable structure is related to the decomposition of discrete-time LQ
deep-structured games in
\citet{arabneydi2020explicitsequentialequilibrialq}.

\subsection{Global Solvability}\label{sec:solvability_structure}

Theorem~\ref{thm:full_horizon_closed_loop_equilibrium} constructs a
subgame-perfect closed-loop Nash equilibrium conditional on the Riccati
equation admitting a solution on \([0,T]\). We now show that
Assumption~\ref{ass:standing_model} guarantees this full-horizon solvability a
priori. The essential ingredient here apart from the strict convexity of the execution cost is
the exclusion of price-manipulation.

\begin{theorem}[Uniform Global Solvability]
\label{thm:global_riccati_solvability}
Under Assumption~\ref{ass:standing_model}, for every \(n\ge2\) satisfying
\(\chi_u\varpi_n\le1\), the Riccati terminal-value problem
\eqref{eq:fixed_interval_riccati_P} has a unique
solution \(H_n\in C^1([0,T];\mathbb S^{3d})\).
More precisely, for \(t\in[0,T]\), define
\[
G_Q(t):=Q_T+(T-t)Q,
\qquad
G_R(t)
:=
\int_t^T
e^{-R^\top(s-t)}\Lambda^{-1}e^{-R(s-t)}\,ds,
\]
and set
\[
\overline H(t)
:=
\begin{bmatrix}
G_Q(t)&G_Q(t)&0\\
G_Q(t)&G_Q(t)&0\\
0&0&0
\end{bmatrix},
\qquad
\underline H(t)
:=
\left(
2\operatorname{tr}(G_Q(t))
+\frac12\operatorname{tr}(G_R(t))
\right)\operatorname{Id}_{3d}.
\]
Then, uniformly over \(n\ge2\) satisfying \(\chi_u\varpi_n\le1\),
\begin{equation}
\label{eq:global_riccati_two_sided_bound}
-\underline H(t)
\preceq
H_n(t)
\preceq
\overline H(t),
\qquad
0\le t\le T.
\end{equation}
\end{theorem}

\begin{remark}[Impacted Terminal Marks and Solvability]
\label{rem:impacted_terminal_mark_solvability}
The global-solvability proof relies on marking terminal inventory to the
fundamental price. Marking it instead to \(S(T)=S^0(T)+I(T)\) leads to a
modified Riccati equation, but the a priori argument for global existence no
longer applies. Indeed, one can construct parameter values satisfying our
standing assumptions for which the modified Riccati solution blows up in finite
time. We omit the construction because it is not needed for the subsequent
analysis.

Both \citet{NeumanVoss2023} and
\citet{campbell2026optimalexecutionntraders} mark terminal inventory at the
fundamental price to retain convexity of each player's optimization problem,
making their first-order conditions sufficient. We use the same valuation
convention, but for a different reason: it provides the a priori bounds that
prevent Riccati blow-up and guarantee full-horizon existence of the value
functions and the closed-loop Nash equilibrium.
\end{remark}

\subsection{Equilibrium Dynamics and Belief Heterogeneity}
\label{sec:equilibrium_dynamics}

Let \(X_n^{i,0}\), \(\bar X_n^0\), and \(I_n^0\) denote the equilibrium
individual inventory, average inventory, and impact before the announcement,
and let \(X_n^{i,k}\), \(\bar X_n^k\), and \(I_n^k\) denote their
continuations after scenario \(k\) is announced. Write
\(\delta_n^{i,k}:=X_n^{i,k}-\bar X_n^k\) for \(k=0,\ldots,m\).
Let \(K_n^\delta(t)\) and \(K_n(t)\) denote the upper-left \(d\times d\)
and lower-right \(2d\times2d\) blocks, respectively, of
\(M_n(H_n(t))-B_n\Lambda^{-1}\Theta_n(H_n(t))D_n\).
The equilibrium feedbacks
\eqref{eq:full_horizon_markov_feedback} give the linear common-state ODE
\begin{equation}\label{eq:common_state_dynamics}
\begin{aligned}
\begin{bmatrix}
\dot{\bar X}_n^k\\
\dot I_n^k
\end{bmatrix}
={}&
K_n(t)
\begin{bmatrix}
\bar X_n^k\\
I_n^k
\end{bmatrix}
+
\begin{bmatrix}
0_d\\
\Gamma g^k
\end{bmatrix}
\\
&-\frac{1}{d_n}
\begin{bmatrix}
\operatorname{Id}_d\\
n\varpi_n\Gamma
\end{bmatrix}
\Lambda^{-1}
\left(
B_n^\top r_n^k+\chi_g\Lambda g^k
\right),
\end{aligned}
\end{equation}
for \(k=0\) before the announcement and \(k=1,\ldots,m\) afterward,
using the convention \(g^0\equiv0\). Before the announcement, trader \(i\)'s inventory deviation satisfies
\begin{equation}\label{eq:pre_ann_delta_dynamics}
\dot\delta_n^{i,0}
=K_n^\delta(t)\delta_n^{i,0}
-\frac{1}{a_n}\Lambda^{-1}B_n^\top\bigl(r_n^{i,0}-r_n^0\bigr),
\qquad
\delta_n^{i,0}(0)
=x_{n,0}^i-\bar x_{n,0},
\end{equation}
where \(K_n^\delta(t)=-\frac{1}{a_n}\Lambda^{-1}
[\Theta_n(H_n(t))]_\delta\), with the subscript \(\delta\) selecting the
first \(d\) columns.
After scenario \(k\) is announced, it satisfies
\begin{equation}\label{eq:post_ann_delta_dynamics}
\dot\delta_n^{i,k}
=K_n^\delta(t)\delta_n^{i,k},
\qquad
\delta_n^{i,k}(T_{\mathrm{ann}})
=\delta_n^{i,0}(T_{\mathrm{ann}}).
\end{equation}
The common-state and individual-deviation equations
\eqref{eq:common_state_dynamics}--\eqref{eq:post_ann_delta_dynamics}
therefore decouple both before and after the announcement. Before the
announcement, the deviation is driven by
\(B_n^\top(r_n^{i,0}-r_n^0)\), which depends on the prior difference
\(\pi_n^i-\bar\pi_n\). After the announcement, neither the deviation
equation \eqref{eq:post_ann_delta_dynamics} nor its initial condition depends
on \(k\), so uniqueness implies that \(\delta_n^{i,k}\) is identical across
scenarios. The realized scenario affects only the common state through
\(r_n^k\) and \(g^k\).

\section{Mean-Field Limit}\label{sec:mean_field_limit}

The finite-player closed-loop strategies constructed in
Section~\ref{sec:solution} are attractive because they respond to the realized
state. They can, however, be difficult to implement when an opportunist knows
only that the number of relevant players is large, rather than the exact value
of \(n\), and cannot observe the current average inventory \(\bar X_n\); see
Remark~\ref{rem:average_inventory_proxy}. To obtain trading schedules that can be
implemented with this limited information, we consider a mean-field
approximation.

We define a mean-field game motivated by the \(n\to\infty\) limit of the
finite-player model under the mean-field scaling \(\varpi_n=\varpi/n\),
and construct an equilibrium. For zero initial inventories and impact and
convergence of the finite-player average prior to the mean-field population
average, the constructed finite-player equilibrium rates,
inventories, and impact converge uniformly to their mean-field counterparts
by Theorem~\ref{thm:finite_to_mf_convergence}. The mean-field trading schedules form
approximate closed-loop Nash equilibria in the finite-player game from zero
initial inventories and impact, without a subgame-perfection guarantee
(Theorem~\ref{thm:mf_approximate_nash}). These schedules require
neither observation of the average inventory nor knowledge of the exact
\(n\), while retaining dependence on the individual and population-average
priors.

Closed-loop Nash equilibria can converge to weak mean-field equilibria with
common randomness beyond the model's prescribed common noise
\citep{Lacker2020ClosedLoop,LackerLeFlem2023}. For our constructed sequence,
the limiting population flow depends only on the public announcement and is
deterministic conditional on \(\Xi\).

\subsection{Definition of the Mean-Field Game}

Let \(\Pi\) be a probability distribution on \(\Delta_m\). A representative
player of type \(\pi\in\Delta_m\) evaluates the announcement under the
subjective probability
\(\mathbb P^\pi(\Xi=\xi^k)=\pi^k.\)
The \emph{population-average prior} is \(\bar\pi:=\int_{\Delta_m}\pi\,\Pi(d\pi)\)
which is the population analogue of the finite-population average prior
\(\bar\pi_n\) in \eqref{eq:finite_population_average_prior}.
In our earlier equilibrium construction for the finite-player game we could let the
fundamental price drop out completely. This motivates the following mean-field game, in which the
announcement is the only source of randomness.

\begin{definition}[Mean-Field Equilibrium]\label{def:mf_equilibrium}
A candidate mean flow is a collection
\[
\bar u^0\in L^2([0,T_{\mathrm{ann}}],\mathbb R^d),
\qquad
\bar u^k\in L^2([T_{\mathrm{ann}},T],\mathbb R^d),
\quad k=1,\ldots,m.
\]
The corresponding impact components satisfy on $[0,T_{\mathrm{ann}}]$ and $[T_{\mathrm{ann}},T]$, respectively,
\begin{equation}\label{eq:mf_component_impact_dynamics}
\begin{aligned}
\dot I^0(t)
&=-RI^0(t)+\varpi\Gamma\bar u^0(t),
& I^0(0)&=0,
\\
\dot I^k(t)
&=-RI^k(t)+\Gamma\bigl(\varpi\bar u^k(t)+g^k(t)\bigr),
& I^k(T_{\mathrm{ann}})&=I^0(T_{\mathrm{ann}}),
\quad k=1,\ldots,m.
\end{aligned}
\end{equation}
For a type \(\pi\), an admissible trading schedule
\(u^\pi=(u^{\pi,0},\ldots,u^{\pi,m})\) consists of
\[
u^{\pi,0}\in L^2([0,T_{\mathrm{ann}}],\mathbb R^d),
\qquad
u^{\pi,k}\in L^2([T_{\mathrm{ann}},T],\mathbb R^d),
\quad k=1,\ldots,m.
\]
The inventories satisfy
\(\dot X^{\pi,0}(t)=u^{\pi,0}(t)\) on
\([0,T_{\mathrm{ann}}]\) and, for \(k=1,\ldots,m\),
\(\dot X^{\pi,k}(t)=u^{\pi,k}(t)\) on \([T_{\mathrm{ann}},T]\), with
\(X^{\pi,0}(0)=0\) and
\(X^{\pi,k}(T_{\mathrm{ann}})=X^{\pi,0}(T_{\mathrm{ann}})\).
A pair \((u,\bar u)\), consisting of a family \(u=(u^\pi)_{\pi\in\Delta_m}\) of admissible trading schedules such that \((\pi,t)\mapsto u^{\pi,k}(t)\) is jointly measurable for each \(k=0,\ldots,m\), and a candidate mean flow \(\bar u\), is a mean-field equilibrium if the following two conditions hold:

\begin{itemize}
    \item \textbf{Consistency:} For a.e. time \(t\) on the respective intervals,
    \begin{equation}\label{eq:mf_average_consistency}
        \bar u^k(t)
        =
        \int_{\Delta_m}
        u^{\pi,k}(t)\,\Pi(d\pi),
        \quad k=0,\ldots,m.
    \end{equation}
    \item \textbf{Optimality:} For \(\Pi\)-a.e. \(\pi\),
    \(
    u^\pi\in\operatorname*{argmin}_{v^\pi}J^\pi(v^\pi;\bar u),
    \)
    where
    \begin{equation}\label{eq:mf_type_objective}
    \begin{aligned}
    J^\pi(u^\pi;\bar u)
    :={}&\int_0^{T_{\mathrm{ann}}}
    \mathcal L\bigl(X^{\pi,0}(t),u^{\pi,0}(t),I^0(t),
    \varpi\bar u^0(t),0\bigr)\,dt
    \\
    &+\sum_{k=1}^m\pi^k
    \int_{T_{\mathrm{ann}}}^T
    \mathcal L\bigl(X^{\pi,k}(t),u^{\pi,k}(t),I^k(t),
    \varpi\bar u^k(t),g^k(t)\bigr)\,dt
    \\
    &+\frac{1}{2}\sum_{k=1}^m\pi^k
    (X^{\pi,k}(T))^\top Q_T X^{\pi,k}(T).
    \end{aligned}
    \end{equation}
    The minimum is taken over admissible trading schedules \(v^\pi\), with
    the mean flow \(\bar u\) and its impact components held fixed.
\end{itemize}

\end{definition}

We make a few remarks about this definition. First, admissibility requires
rates to be defined and square-integrable even in
post-announcement scenarios which a player of type $\pi$ assigns zero probability.
This ensures that the rates entering the population average are defined in every scenario
and that the trader has an executable policy even in post-announcement scenarios that he considers impossible to occur.
In the next subsection, we construct the rates through
subgame-perfect feedback responses to \(\bar u\). These feedbacks are optimal from every starting time
and inventory, including in post-announcement scenarios with \(\pi^k=0\).
Second, we impose \(X^{\pi,0}(0)=0\) for every type and \(I^0(0)=0\).
The zero initial inventories are natural for
traders who participate solely to profit from the indexer flow.
In our game, heterogeneity comes from beliefs about index membership,
whereas \citet{NeumanVoss2023} study a liquidation game with heterogeneous
deterministic initial positions.

Our model combines two features: heterogeneous beliefs and mean-field
consistency obtained by averaging over a distribution of belief types.
\citet{Casgrain2020} also allow agents to disagree about the dynamics of the
asset price and its latent factors, but divide the population into finitely
many positive-mass subpopulations, each sharing the same beliefs and objective.
The papers \citep{Lacker2019,Lacker2020} instead formulate
mean-field games with a general distribution of types, such as risk
tolerances, relative-performance preferences, and market parameters. Both
\citet{Casgrain2020} and \citet{Lacker2019,Lacker2020} retain stochastic
asset-price dynamics in their mean-field games. By contrast, \(S^0\) drops out
of our mean-field objective and state dynamics, so the common announcement
\(\Xi\) is the only remaining source of uncertainty and the dynamics are
deterministic conditional on its realization.

\subsection{Construction of the Mean-Field Equilibrium}

We now construct a mean-field equilibrium by first using dynamic programming
to solve the control problem of a representative player of type \(\pi\) for
a fixed mean flow \(\bar u\) and then imposing the consistency condition on the coefficients of the
resulting optimal feedback policy.
This approach yields quadratic value functions before and after the
announcement, on the time intervals \([0,T_{\mathrm{ann}}]\) and
\([T_{\mathrm{ann}},T]\), of the form
\begin{equation}\label{eq:mf_quadratic_value_ansatz}
\begin{aligned}
V^{\pi,0}(t,x)
&=\frac12x^\top P(t)x+(\ell^{\pi,0}(t))^\top x+c^{\pi,0}(t),
\\
V^k(t,x)
&=\frac12x^\top P(t)x+(\ell^k(t))^\top x+c^k(t),
\qquad k=1,\ldots,m.
\end{aligned}
\end{equation}
These are the value functions of the representative type-\(\pi\) player's
control problem starting from inventory \(x\) at time \(t\), with
\(\bar u\) fixed to the \emph{equilibrium} mean flow.
The quadratic coefficient is common across types, intervals, and scenarios.
The affine and constant coefficients depend on the scenario after the
announcement and on the type before the announcement.

We now state the system of equations satisfied by the coefficients in
\eqref{eq:mf_quadratic_value_ansatz}. Theorem~\ref{thm:mf_solution} below and its proof then make this
construction precise and also state the corresponding optimal controls.
To state the ODEs satisfied by the coefficients we need to introduce some auxiliary processes.
Define the type-average pre-announcement affine coefficient by
\(
\ell^0(t):=
\int_{\Delta_m}\ell^{\pi,0}(t)\,\Pi(d\pi).
\)
Also, let \(\bar X^0\) and \(I^0\) denote the equilibrium mean-inventory and
impact-state paths before the announcement, and let \(\bar X^k\) and \(I^k\)
denote their counterparts in post-announcement scenario \(k\). Define the
augmented mean-field common-state paths by $Y^k(t):=(\bar X^k(t),I^k(t),\ell^k(t))^\top$ for $k=0,...,m$.

\begin{itemize}
\item \textbf{Quadratic Coefficient.}

On \([0,T]\), \(P\) solves
\begin{equation}\label{eq:mf_riccati}
\dot P(t)
=
- Q+\frac12P(t)\Lambda^{-1}P(t),
\qquad
P(T)=Q_T.
\end{equation}
Since \(\Lambda\succ0\) and \(Q,Q_T\succeq0\), standard finite-horizon LQR
theory gives a unique continuous symmetric nonnegative solution on
\([0,T]\).

\item \textbf{Augmented Mean-Field Common-State System.}

The paths \(Y^0,\ldots,Y^m\) are continuous on their respective intervals
and satisfy the following ODEs, with the equations for \(k=1,\ldots,m\)
imposed away from
\(T_{\mathrm{TWAP}}\) and \(T_{\mathrm{impl}}\):
\begin{subequations}\label{eq:mf_consistency_system}
\begin{equation}\label{eq:mf_y_system}
\begin{aligned}
\dot Y^0(t)&=\mathcal A(t)Y^0(t),
&&t\in[0,T_{\mathrm{ann}}],
\\
\dot Y^k(t)&=\mathcal A(t)Y^k(t)+b^k(t),
&&t\in[T_{\mathrm{ann}},T],
\quad k=1,\ldots,m.
\end{aligned}
\end{equation}
with
\begin{equation}\label{eq:mf_pre_consistency_boundary}
\bar X^0(0)=0,
\qquad
I^0(0)=0,
\qquad
\ell^0(T_{\mathrm{ann}})
=\sum_{k=1}^m\bar\pi^k\ell^k(T_{\mathrm{ann}}),
\end{equation}
and, for \(k=1,\ldots,m\),
\begin{equation}\label{eq:mf_post_consistency_boundary}
\bar X^k(T_{\mathrm{ann}})=\bar X^0(T_{\mathrm{ann}}),
\qquad
I^k(T_{\mathrm{ann}})=I^0(T_{\mathrm{ann}}),
\qquad
\ell^k(T)=0.
\end{equation}
\end{subequations}
Here
\begin{equation}\label{eq:mf_A_matrix}
\mathcal A(t)
=
\frac{1}{2+\varpi\chi_u}
\begin{pmatrix}
-\Lambda^{-1}P(t)&-\Lambda^{-1}&-\Lambda^{-1}\\
-\varpi\Gamma\Lambda^{-1}P(t)&
-(2+\varpi\chi_u)R-\varpi\Gamma\Lambda^{-1}&
-\varpi\Gamma\Lambda^{-1}\\
-\frac{\varpi\chi_u}{2}P(t)\Lambda^{-1}P(t)&
P(t)\Lambda^{-1}&P(t)\Lambda^{-1}
\end{pmatrix},
\end{equation}
and \(b^0(t):=0\), while, for \(k=1,\ldots,m\),
\begin{equation}\label{eq:mf_b_vector}
b^k(t)
=
\frac{1}{2+\varpi\chi_u}
\begin{pmatrix}
-\chi_g g^k(t)\\
(2+\varpi\chi_u-\varpi\chi_g)\Gamma g^k(t)\\
\chi_g P(t)g^k(t)
\end{pmatrix}.
\end{equation}
Note that the conditions for \(\bar X^0\) and \(I^0\) are imposed at \(t=0\), whereas
the conditions for \(\ell^k\), \(k=1,\ldots,m\), are imposed at \(t=T\). The conditions at
\(T_{\mathrm{ann}}\) match the pre- and post-announcement mean inventories
and impact states and link \(\ell^0\) to the post-announcement affine
coefficients $\ell^k$. Thus, \eqref{eq:mf_consistency_system} is a coupled
boundary-value problem rather than an initial-value problem, and its
solvability is not immediate. The equilibrium mean flows are
\begin{subequations}\label{eq:mf_mu_closed_form}
\begin{align}
\bar u^0(t)
&=
-\frac{1}{2+\varpi\chi_u}\Lambda^{-1}
\left(P(t)\bar X^0(t)+I^0(t)+\ell^0(t)\right),
\label{eq:mf_pre_mu_closed_form}\\
\bar u^k(t)
&=
-\frac{1}{2+\varpi\chi_u}\Lambda^{-1}
\left(P(t)\bar X^k(t)+I^k(t)+\ell^k(t)+\Lambda\chi_g g^k(t)\right).
\label{eq:mf_post_mu_closed_form}
\end{align}
\end{subequations}

\item \textbf{Type-Specific Affine Coefficient.}

Given the solution of \eqref{eq:mf_consistency_system}, the coefficient
\(\ell^{\pi,0}\) for \(\pi\in\Delta_m\) solves on \([0,T_{\mathrm{ann}}]\)
\begin{equation}\label{eq:mf_pre_ell}
\begin{aligned}
\dot\ell^{\pi,0}(t)
={}&
\frac12P(t)\Lambda^{-1}
\left(I^0(t)+\Lambda\varpi\chi_u\bar u^0(t)+\ell^{\pi,0}(t)\right),
\\
\ell^{\pi,0}(T_{\mathrm{ann}})
={}&\sum_{k=1}^m\pi^k\ell^k(T_{\mathrm{ann}}).
\end{aligned}
\end{equation}

\item \textbf{Constant Coefficients.}

\begin{subequations}\label{eq:mf_constant_coefficients}
For \(k=1,\ldots,m\), \(c^k\) is continuous on
\([T_{\mathrm{ann}},T]\) and satisfies the following ODE away from
\(T_{\mathrm{TWAP}}\) and \(T_{\mathrm{impl}}\):
\begin{equation}\label{eq:mf_post_alpha}
\begin{aligned}
\dot c^k(t)
={}&
\frac14
\left(I^k(t)+\Lambda(\varpi\chi_u\bar u^k(t)+\chi_g g^k(t))+\ell^k(t)\right)^\top
\\
&\qquad
\Lambda^{-1}
\left(I^k(t)+\Lambda(\varpi\chi_u\bar u^k(t)+\chi_g g^k(t))+\ell^k(t)\right),
\\
c^k(T)={}&0.
\end{aligned}
\end{equation}
For \(\pi\in\Delta_m\), \(c^{\pi,0}\) solves on
\([0,T_{\mathrm{ann}}]\)
\begin{equation}\label{eq:mf_pre_alpha}
\begin{aligned}
\dot c^{\pi,0}(t)
={}&
\frac14
\left(I^0(t)+\Lambda\varpi\chi_u\bar u^0(t)+\ell^{\pi,0}(t)\right)^\top
\\
&\qquad
\Lambda^{-1}
\left(I^0(t)+\Lambda\varpi\chi_u\bar u^0(t)+\ell^{\pi,0}(t)\right),
\\
c^{\pi,0}(T_{\mathrm{ann}})
={}&\sum_{k=1}^m\pi^kc^k(T_{\mathrm{ann}}).
\end{aligned}
\end{equation}
\end{subequations}
\end{itemize}

\begin{theorem}[Explicit Mean-Field Equilibrium]\label{thm:mf_solution}
Let \(P\) be the solution of
\eqref{eq:mf_riccati}. Suppose that the augmented mean-field common-state
boundary-value problem
\eqref{eq:mf_consistency_system} admits a solution, and fix one such solution
\((Y^0,Y^1,\ldots,Y^m)\). Let \(\ell^{\pi,0}\), \(c^k\), and
\(c^{\pi,0}\) be determined by \eqref{eq:mf_pre_ell},
\eqref{eq:mf_post_alpha}, and \eqref{eq:mf_pre_alpha}.
For \(x\in\mathbb R^d\), define the feedbacks
\(\phi^{\pi,0}\) on \([0,T_{\mathrm{ann}}]\) and
\(\phi^{k}\), $k=1,...,m$, on \([T_{\mathrm{ann}},T]\) by
\begin{subequations}\label{eq:mf_equilibrium_feedbacks}
\begin{align}
\phi^{\pi,0}(t,x)
&=
-\frac12\Lambda^{-1}
\left(P(t)x+\ell^{\pi,0}(t)+I^0(t)+\Lambda\varpi\chi_u\bar u^0(t)\right),
\label{eq:mf_pre_feedback}\\
\phi^{k}(t,x)
&=
-\frac12\Lambda^{-1}
\left(P(t)x+\ell^k(t)+I^k(t)
+\Lambda(\varpi\chi_u\bar u^k(t)+\chi_g g^k(t))\right).
\label{eq:mf_post_feedback}
\end{align}
\end{subequations}
For the inventory paths obtained by solving the inventory equations under
these feedbacks, with
\(X^{\pi,0}(0)=0\) and
\(X^{\pi,k}(T_{\mathrm{ann}})=X^{\pi,0}(T_{\mathrm{ann}})\) for
\(k=1,\ldots,m\), the trading rates
\begin{equation}\label{eq:mf_induced_rates}
u^{\pi,0}(t):=\phi^{\pi,0}(t,X^{\pi,0}(t)),
\qquad
u^{\pi,k}(t):=\phi^k(t,X^{\pi,k}(t)),
\quad k=1,\ldots,m,
\end{equation}
together with the mean flow \eqref{eq:mf_mu_closed_form}, form a mean-field
equilibrium with value functions
\eqref{eq:mf_quadratic_value_ansatz}. The \(\bar X\)- and \(I\)-components
of \(Y^k\) give the equilibrium paths of the population-average inventory
and impact state, respectively.
\end{theorem}

\begin{remark}[State Decomposition]
\label{rem:mf_deviation_policy}
For the equilibrium trajectories, write
\(\delta^{\pi,k}:=X^{\pi,k}-\bar X^k\) for \(k=0,\ldots,m\).
The augmented common-state paths $Y^k$ satisfy \eqref{eq:mf_consistency_system},
whose \(\bar X\)- and \(I\)-equations are the mean-field counterpart of the
finite-player common-state dynamics \eqref{eq:common_state_dynamics}.
The equilibrium feedbacks \eqref{eq:mf_equilibrium_feedbacks} give the
pre-announcement deviation dynamics
\begin{equation}\label{eq:mf_pre_deviation_dynamics}
\dot\delta^{\pi,0}(t)
=
-\frac12\Lambda^{-1}
\left(
P(t)\delta^{\pi,0}(t)+\ell^{\pi,0}(t)-\ell^0(t)
\right),
\qquad
\delta^{\pi,0}(0)=0.
\end{equation}
After scenario \(k\) is announced, the deviation satisfies
\begin{equation}\label{eq:mf_post_deviation_dynamics}
\dot\delta^{\pi,k}(t)
=
-\frac12\Lambda^{-1}P(t)\delta^{\pi,k}(t),
\qquad
\delta^{\pi,k}(T_{\mathrm{ann}})
=\delta^{\pi,0}(T_{\mathrm{ann}}).
\end{equation}
Thus, once the affine coefficients are determined, the common-state and
individual-deviation dynamics decouple. Before the announcement, the deviation
is driven by \(\ell^{\pi,0}-\ell^0\), whose terminal value depends on
\(\pi-\bar\pi\). As in the finite-player game, after the announcement neither
\eqref{eq:mf_post_deviation_dynamics} nor its initial condition depends on
\(k\), so uniqueness implies that \(\delta^{\pi,k}\) is identical across
scenarios. The realized scenario affects only the common state through
\(g^k\).
\end{remark}

\subsection{Solvability and Computation}

The equilibrium construction in Theorem~\ref{thm:mf_solution} is conditional
on solvability of the augmented mean-field common-state boundary-value problem
\eqref{eq:mf_consistency_system}. We now show that this problem is solvable
under Assumption~\ref{ass:standing_model} and show how it can be solved. For
\(0\le s,t\le T\), let
\(\mathcal E(t,s)\in\mathbb R^{3d\times3d}\) be the state-transition matrix
of the homogeneous system
\(\dot{Y}(t)=\mathcal A(t)Y(t)\). That is, for each fixed \(s\),
\(\frac{\partial}{\partial t}\mathcal E(t,s)=\mathcal A(t)\mathcal E(t,s)\)
and \(\mathcal E(s,s)=\operatorname{Id}_{3d}.\)
We label the \(d\times d\) blocks of \(\mathcal E(t,s)\) by row and column
subscripts \(X,I,\ell\), following the coordinate order of \(Y^k\).
Let \(\mathcal E_{\ell:}(t,s)\) denote the last \(d\) rows of
\(\mathcal E(t,s)\).
The following lemma expresses the solution in terms of these transition blocks
and identifies the invertibility conditions that ensure uniqueness.

\begin{lemma}
\label{lem:mf_state_transition_representation}
Under Assumption~\ref{ass:standing_model}, if \(\mathcal E_{\ell\ell}(T,0)\) is invertible,
the initial affine coefficient is
\begin{equation}\label{eq:mf_pre_matching}
\ell^0(0)
=
-(\mathcal E_{\ell\ell}(T,0))^{-1}\sum_{k=1}^m\bar\pi^k
\int_{T_{\mathrm{ann}}}^T\mathcal E_{\ell:}(T,s)b^k(s)\,ds.
\end{equation}
For \(t\in[0,T_{\mathrm{ann}}]\), the pre-announcement solution is
\[
\begin{aligned}
\bar X^0(t)
=\mathcal E_{X\ell}(t,0)\ell^0(0),
\qquad
I^0(t)
=\mathcal E_{I\ell}(t,0)\ell^0(0),
\qquad
\ell^0(t)
=\mathcal E_{\ell\ell}(t,0)\ell^0(0).
\end{aligned}
\]
If \(\mathcal E_{\ell\ell}(T,T_{\mathrm{ann}})\) is invertible, then, using the
values at \(T_{\mathrm{ann}}\), the post-announcement initial affine coefficient for
scenario \(k=1,\ldots,m\) is
\begin{equation}\label{eq:mf_post_transition_ell}
\begin{aligned}
\ell^k(T_{\mathrm{ann}})
={}&
-(\mathcal E_{\ell\ell}(T,T_{\mathrm{ann}}))^{-1}
\Bigl(\mathcal E_{\ell X}(T,T_{\mathrm{ann}})\bar X^0(T_{\mathrm{ann}})
\\
&\qquad+\mathcal E_{\ell I}(T,T_{\mathrm{ann}})I^0(T_{\mathrm{ann}})
+\int_{T_{\mathrm{ann}}}^T\mathcal E_{\ell:}(T,s)b^k(s)\,ds\Bigr).
\end{aligned}
\end{equation}
Finally, the post-announcement solution in scenario \(k\) is
\begin{equation}\label{eq:mf_post_forward_solution}
Y^k(t)
=
\mathcal E(t,T_{\mathrm{ann}})
\begin{pmatrix}
\bar X^0(T_{\mathrm{ann}})\\
I^0(T_{\mathrm{ann}})\\
\ell^k(T_{\mathrm{ann}})
\end{pmatrix}
+
\int_{T_{\mathrm{ann}}}^t
\mathcal E(t,s)b^k(s)\,ds,
\qquad
t\in[T_{\mathrm{ann}},T].
\end{equation}
If both \(\mathcal E_{\ell\ell}(T,0)\) and \(\mathcal E_{\ell\ell}(T,T_{\mathrm{ann}})\) are
invertible, these formulas determine a unique solution of the mean-field
boundary-value problem \eqref{eq:mf_consistency_system}.
\end{lemma}

We now show that the procedure in
Lemma~\ref{lem:mf_state_transition_representation} can be carried out under the
market-parameter assumptions of Subsection~\ref{subsec:parameter_choices_and_restrictions}.

\begin{theorem}[Implementability of the Mean-Field Construction]
\label{thm:mf_implementability}
Under Assumption~\ref{ass:standing_model},
\(\mathcal E_{\ell\ell}(T,T_{\mathrm{ann}})\) and
\(\mathcal E_{\ell\ell}(T,0)\) in
Lemma~\ref{lem:mf_state_transition_representation} are
invertible. Hence the mean-field construction is implementable and yields a
unique equilibrium within the affine construction of
Theorem~\ref{thm:mf_solution}.
\end{theorem}

\subsection{Convergence and Approximate Nash Equilibria}

We consider finite-player games with \(x_{n,0}^i=0\) for every player and
\(\iota_0=0\) and show that the equilibria constructed in
Section~\ref{sec:full_procedure} under the scaling \(\varpi_n=\varpi/n\) converge to
the mean-field equilibrium constructed in Theorem~\ref{thm:mf_solution}.
For \(n\) traders with priors \(\pi_n^1,\ldots,\pi_n^n\), let
\(u_n^{i,0}\) and \(u_n^{i,k}\), \(k=1,\ldots,m\), denote trader \(i\)'s
equilibrium rates before the announcement and after scenario \(k\) is
announced, respectively, obtained by evaluating
\eqref{eq:full_horizon_markov_feedback} along the equilibrium trajectories.
The matching mean-field rates \(u^{\pi_n^i,0}\) and \(u^{\pi_n^i,k}\) are
those from Theorem~\ref{thm:mf_solution} with \(\pi=\pi_n^i\).
Write \(\bar u_n^k:=\frac{1}{n}\sum_{i=1}^n u_n^{i,k}\), \(k=0,\ldots,m\),
for the finite-player average rates.

\begin{theorem}[Convergence to the Mean-Field Equilibrium]
\label{thm:finite_to_mf_convergence}
Under Assumption~\ref{ass:standing_model}, set
\(n_0:=\max\{2,\lceil\varpi\chi_u\rceil\}\). For each \(n\ge n_0\), consider the
finite-player game with population weight
\(\varpi_n=\varpi/n\) and
fixed instantaneous-flow loadings \(\chi_u,\chi_g\),
with \(x_{n,0}^i=0\) for every player \(i\) and \(\iota_0=0\). Let the
finite-player equilibrium be the one constructed in
Section~\ref{sec:full_procedure}.
Set \(\varepsilon_n:=\frac1n+\|\bar\pi_n-\bar\pi\|\).
There is a
constant \(C\), independent of \(n\) and of the finite belief profile, such
that, for every \(n\ge n_0\),
\begin{equation}\label{eq:finite_to_mf_common_estimate}
\max\left\{
\sup_{t\in[0,T_{\mathrm{ann}}]}
\|\bar u_n^{0}(t)-\bar u^0(t)\|,
\max_{k=1,\ldots,m}\sup_{t\in[T_{\mathrm{ann}},T]}
\|\bar u_n^{k}(t)-\bar u^k(t)\|
\right\}
\le
C\varepsilon_n.
\end{equation}
Moreover, the player-level rate estimate holds uniformly over the finite
population:
\begin{equation}\label{eq:finite_to_mf_control_estimate}
\max_{1\le i\le n}\;
\max\left\{
\begin{aligned}
&\sup_{t\in[0,T_{\mathrm{ann}}]}
\|u_n^{i,0}(t)-u^{\pi_n^i,0}(t)\|,\\
&\max_{k=1,\ldots,m}\sup_{t\in[T_{\mathrm{ann}},T]}
\|u_n^{i,k}(t)-u^{\pi_n^i,k}(t)\|
\end{aligned}
\right\}
\le
C\varepsilon_n.
\end{equation}
Estimate \eqref{eq:finite_to_mf_common_estimate} also holds for the average
inventories and impact states. Estimate \eqref{eq:finite_to_mf_control_estimate}
also holds for individual inventories, uniformly over traders.
\end{theorem}

We next show that using the mean-field rates in the finite-player game
yields an approximate Nash equilibrium.

\begin{theorem}[Approximate Nash Equilibrium from Mean-Field Rates]
\label{thm:mf_approximate_nash}
Adopt the assumptions and notation of
Theorem~\ref{thm:finite_to_mf_convergence}.
For each \(n\ge n_0\), let \(\widehat\Phi_n\) prescribe to each trader \(i\)
the mean-field trading schedule \(u^{\pi_n^i}\) from Theorem~\ref{thm:mf_solution},
using feedbacks that depend only on time and the announcement state.
Then \(\widehat\Phi_n\in\mathcal U_n\) and there is a constant \(C\),
independent of \(n\) and of the finite belief profile, such that
\begin{equation}\label{eq:mf_profile_approximate_nash}
J_n^i(
\widehat\Phi_n^i;\widehat\Phi_n^{-i}
)
\leq
J_n^i(
\Psi_n^i;\widehat\Phi_n^{-i}
)
+
C\varepsilon_n^2
\end{equation}
for every \(i=1,\ldots,n\) and every policy \(\Psi_n^i\) satisfying
\((\Psi_n^i,\widehat\Phi_n^{-i})\in\mathcal U_n\).
\end{theorem}

The approximate-Nash guarantee applies from \(x_{n,0}^i=0\) for every player and
\(\iota_0=0\) and is not asserted for continuation games after arbitrary
off-equilibrium histories.
\begin{remark}[Random and Deterministic Belief Populations]\label{rem:mf_iid_beliefs}
Theorems~\ref{thm:finite_to_mf_convergence} and
\ref{thm:mf_approximate_nash} apply to every fixed belief profile.
If the profile is generated by i.i.d. sampling from \(\Pi\), then
\(\|\bar\pi_n-\bar\pi\|=O_p(n^{-1/2})\). The uniform path approximation
error is therefore \(O_p(n^{-1/2})\), and the profile \(\widehat\Phi_n\) is an
\(O_p(n^{-1})\)-Nash equilibrium.
Alternatively, deterministic belief profiles satisfying
\(\|\bar\pi_n-\bar\pi\|=O(n^{-1})\) yield a uniform path approximation
error of \(O(n^{-1})\) and an \(O(n^{-2})\)-Nash equilibrium.
For comparison, \citet[Corollary~4.12]{NeumanVoss2023} obtain an
\(O(n^{-1})\) approximate-Nash guarantee in their liquidation game, over
uniformly \(L^2\)-bounded deviations and under a parameter-dependent
horizon restriction.
\end{remark}

\section{Illustrations}\label{sec:experiments}

We illustrate the model without seeking a complete empirical fit to real-world
index rebalances as the goal here is to showcase the qualitative properties
of the model. 

\subsection{Illustrative Setup}\label{sec:experiments_setup}

\paragraph{Market Parameters.}

We consider an index reconstitution involving \(d=5\) assets, with event times
\(T_{\mathrm{ann}}=10\), \(T_{\mathrm{TWAP}}=19\), \(T_{\mathrm{impl}}=20\), and \(T=22.\)
Time is measured in trading days, with one trading day equal to \(6.5\) trading
hours. Share quantities are
measured in millions of shares, and trading rates in millions of shares per
trading day. We assume that all coefficient matrices are diagonal:
\[
\Lambda=\operatorname{diag}(\lambda_1,\ldots,\lambda_d),
\qquad
\Gamma=\operatorname{diag}(\gamma_1,\ldots,\gamma_d),
\qquad
R=\operatorname{diag}(\rho_1,\ldots,\rho_d),
\]
\[
Q=\operatorname{diag}(q_1,\ldots,q_d),
\qquad
Q_T=\operatorname{diag}(q_{T,1},\ldots,q_{T,d}).
\]

We parameterize \(\Lambda\) and \(\Gamma\) using the estimates reported by
\cite[Table~A.7]{cartea_jaimungal_2016}. The instantaneous-impact coefficient
\(\hat k_a\) in \cite{cartea_jaimungal_2016} corresponds to our \(\lambda_a\), but because they measure trading
rates in shares per hour whereas we measure them in millions of shares per
trading day, we set \(\lambda_a=10^6\hat k_a/6.5\). Their permanent-impact
coefficient \(\hat b_a\) corresponds to our \(\gamma_a\), but because they
measure share quantities in individual shares whereas we measure them in
millions of shares, we set \(\gamma_a=10^6\hat b_a\). Their permanent-impact
model corresponds to \(R=0\) in our parameterization. To obtain transient
impact, we instead set \(\rho_a=2\log 2\) for every asset, giving the impact
state a half-life of half a trading day.

To choose \(q_a\), we suppress the impact state and strategic interactions
and consider an auxiliary scalar infinite-horizon problem with
\(\dot x_a=u_a\) and running cost
\(\lambda_a u_a^2+\frac{1}{2}q_a x_a^2\). Its optimal
feedback is \(u_a=-\sqrt{q_a/(2\lambda_a)}\,x_a\). We set
\(q_a/\lambda_a=0.25\,\mathrm{day}^{-2}\), which gives an inventory half-life
of approximately \(1.96\) trading days in the auxiliary problem. The
finite-horizon illustrations separately include a terminal penalty. We set
\(q_{T,a}/\lambda_a=13{,}000\,\mathrm{day}^{-1}\) to make terminal inventories
negligible.

Table~\ref{tab:experiment_asset_parameters} summarizes the data and
coefficients for the five selected stocks: Intel (INTC), Cisco (CSCO), eBay
(EBAY), Dow (DOW), and Costco (COST), with \(s_a^0\) denoting asset \(a\)'s
initial price. We set \(f_a\) to \(100\) days of average daily volume (ADV).

\begin{table}[!tbp]
\centering
\caption{Asset-Level Model Inputs and Coefficients}
\label{tab:experiment_asset_parameters}
\scriptsize
\resizebox{\textwidth}{!}{%
\begin{tabular}{clrrrrrrr}
\toprule
Asset & Ticker & \(s_a^0\) & \(\mathrm{ADV}_a\)
& \(\lambda_a\) & \(\gamma_a\) & \(\rho_a\) & \(q_a\) & \(q_{T,a}\) \\
\midrule
1 & INTC &  30.26 & 5.489503 & 0.067692 &  1.09 & 1.386294 & 0.016923 &    880 \\
2 & CSCO &  24.21 & 5.598870 & 0.037538 &  0.63 & 1.386294 & 0.009385 &    488 \\
3 & EBAY &  53.41 & 2.506048 & 0.384615 &  4.50 & 1.386294 & 0.096154 &  5,000 \\
4 & DOW  &  49.48 & 1.453332 & 0.463077 &  6.90 & 1.386294 & 0.115769 &  6,020 \\
5 & COST & 121.37 & 0.540008 & 1.630769 & 25.30 & 1.386294 & 0.407692 & 21,200 \\
\bottomrule
\end{tabular}}
\par\vspace{0.5\baselineskip}
\begin{minipage}{0.98\textwidth}
\footnotesize
\raggedright
\emph{Units.} The initial price \(s_a^0\) is measured in USD per share,
\(\mathrm{ADV}_a\) in millions of shares per trading day. The units of
\(\lambda_a\), \(\gamma_a\), \(\rho_a\),
\(q_a\), and \(q_{T,a}\) are, respectively,
\((\mathrm{USD/share})/(\mathrm{million\ shares/day})\),
\((\mathrm{USD/share})/(\mathrm{million\ shares})\),
\(\mathrm{day}^{-1}\),
\((\mathrm{USD/share})/(\mathrm{million\ shares}\cdot\mathrm{day})\), and
\((\mathrm{USD/share})/(\mathrm{million\ shares})\).
\par
\end{minipage}
\end{table}

We set \(x_{n,0}^i=0\) for every opportunist, \(\iota_0=0\), and
\(\chi_u=\chi_g=1\).
Further, we take \(S^0(t)\equiv s^0\), thus suppressing fluctuations that would
affect realized execution prices and wealth but not the equilibrium strategies
or impact state.

\paragraph{Membership Scenarios and Beliefs.}
The index contains exactly two of the five assets. Assets \(1\) and \(2\) are
the pre-rebalance constituents, and every two-asset subset is a possible
post-announcement scenario so that \(m=10\).
We model the population distribution of beliefs by
\(\Pi=\operatorname{Dirichlet}(\zeta\bar\pi)\) with
concentration parameter \(\zeta=50\). To construct the mean
\(\bar\pi\), we assign asset \(a\) the score \(\varsigma_a=6-a\) and compute
the unnormalized weight of pair \(\{a,a'\}\) as
\(\varsigma_a\varsigma_{a'}\). Normalizing the ten pair weights then gives
\(\bar\pi\). For a finite game with \(n\) opportunists, we draw their beliefs
independently from \(\Pi\).
We set \(h_a^0=0.07f_a\) for \(a=1,2\), corresponding to seven days of
ADV in each pre-rebalance constituent and USD~\(2.112\) billion in assets
under management. Under the population-average belief \(\bar\pi\), expected
indexer flow during the TWAP window is negative for assets \(1\) and \(2\)
and positive for assets \(3\), \(4\), and \(5\).

\subsection{Atomic-Player Model}\label{sec:experiments_finite_player}

\begingroup
\renewcommand{\floatpagefraction}{0.85}
\setlength{\emergencystretch}{2em}

We begin with the finite-player model with
\(\varpi_n=1\) and \(n=20\) opportunists whose beliefs
\(\pi_n^i\), \(i=1,\ldots,n\), are i.i.d. samples from \(\Pi\).

For scenario \(k\), let
\(L_n^k(t):=\allowbreak\sum_{i=1}^n X_n^{i,k}(t)\) denote the aggregate inventory
vector.
Figure~\ref{fig:experiments_aggregate_inventory} shows the aggregate inventory
paths \(L_n^k(t)\) for a few selected scenarios.
Before the announcement, opportunists in aggregate have short positions in assets \(1\)
and \(2\) and long positions in assets \(3\), \(4\),
and \(5\), reflecting the expected indexer flow under \(\bar\pi_n\). After the
announcement, the aggregate inventory adjusts scenario by scenario to the realized indexer order.
Notably, for every asset added to or removed from the index, the aggregate inventory in that asset
passes through zero during the TWAP window, temporarily taking the opposite
sign, before being flattened.

Figure~\ref{fig:experiments_impact} shows the impact state for the same
selected scenarios. The impact state changes rapidly at the start of trading
and after the announcement. Away from these changes, opportunist trading
before the TWAP largely offsets impact decay in
\eqref{eq:transient_impact}, allowing impact to remain nearly constant while
inventories continue to change.

\begin{figure}[!htbp]
\centering
\includegraphics[width=\textwidth]{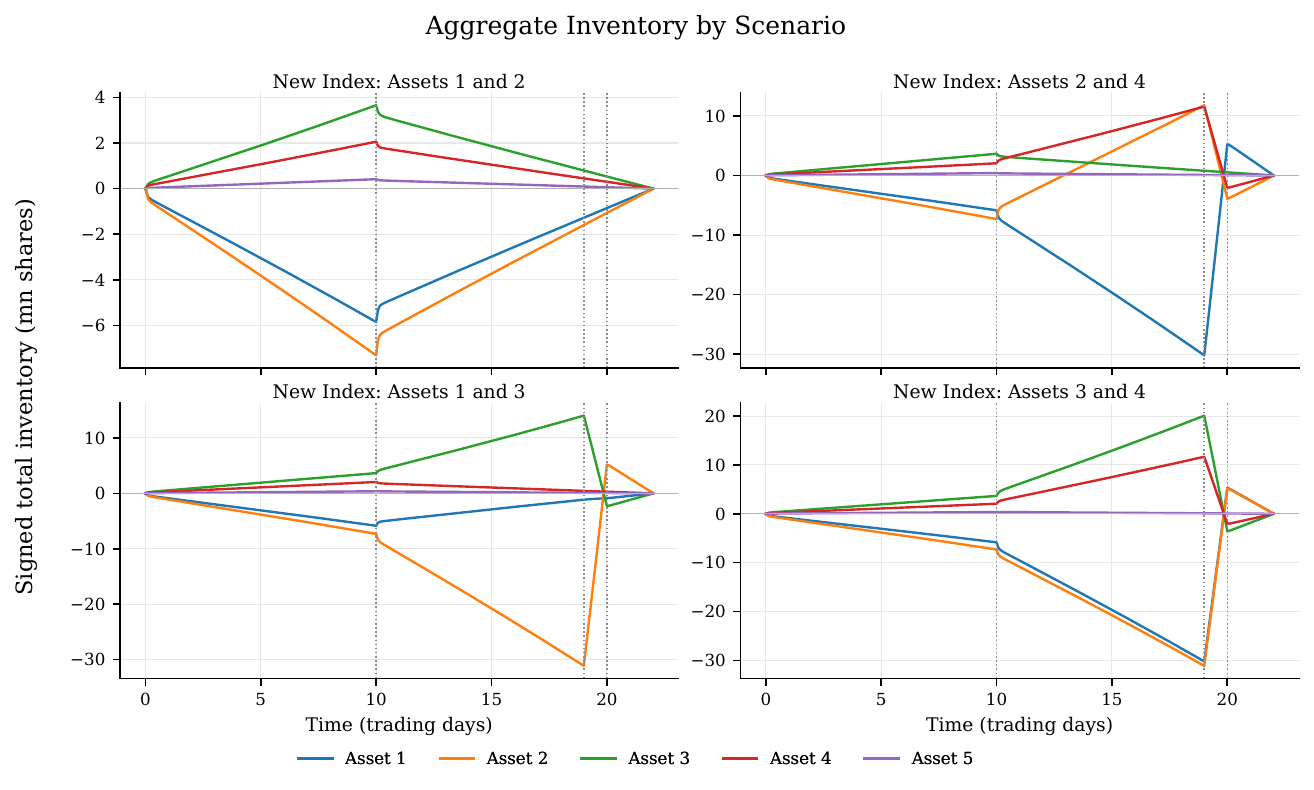}
\caption{Aggregate inventories in selected scenarios.}
\label{fig:experiments_aggregate_inventory}
\par\vspace{\baselineskip}
\includegraphics[width=\textwidth]{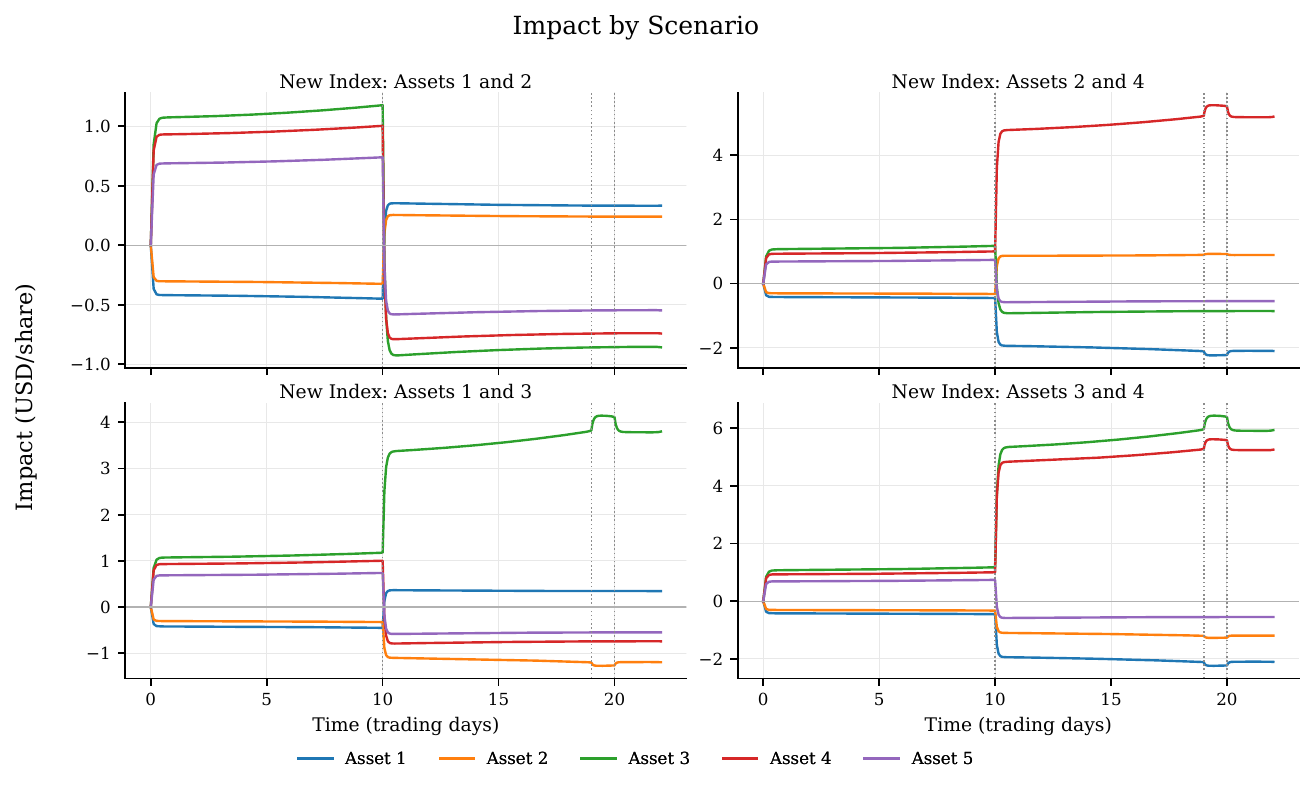}
\caption{Price impact in selected scenarios.}
\label{fig:experiments_impact}
\end{figure}

Figure~\ref{fig:experiments_player_asset2}
shows the asset \(2\) inventories of
the players whose expected indexer flows are lowest, closest to the expected
flow under \(\bar\pi_n\), and highest for that asset. Even the player with
the highest expected indexer flow expects negative flow for asset \(2\).
However, because this expectation is less negative than the expected flow under
\(\bar\pi_n\), the player carries a positive inventory into the
announcement. Once the scenario is announced, the players' pre-announcement beliefs no longer
enter the continuation problem directly, but their effect persists through the
different inventories carried into the announcement.

\begin{figure}[!htbp]
\centering
\includegraphics[width=\textwidth]{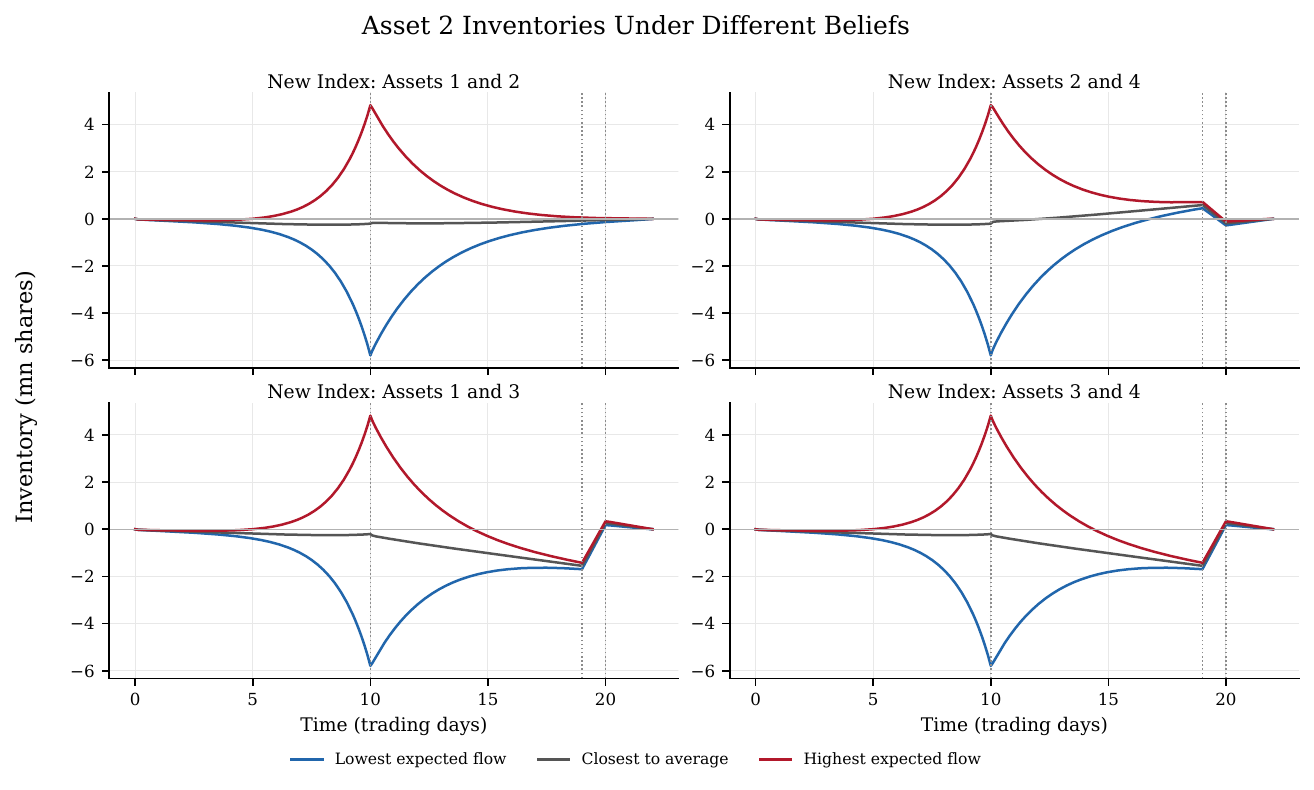}
\caption{Asset \(2\) inventories of the opportunists whose expected indexer flow is lowest, closest to the average under \(\bar\pi_n\), and highest.}
\label{fig:experiments_player_asset2}
\end{figure}

In Figure~\ref{fig:experiments_player_asset2}, some opportunists buy asset
\(2\) while others sell it, so gross trading volume includes purchases and
sales that offset in the aggregate. To isolate the effect of belief
dispersion, we replace the sampled beliefs by
\(\pi_n^i(\delta)=(1-\delta)\bar\pi_n+\delta\pi_n^i\) for \(\delta\in[0,1]\).
Here, \(\delta=1\) gives the sampled beliefs and \(\delta=0\)
gives identical beliefs and varying \(\delta\) preserves \(\bar\pi_n\).
Writing \(u_n^{i,k}(t;\delta)\)
for the resulting control, define expected total gross volume by
\begin{equation}\label{eq:experiments_gross_volume}
G_{n,\delta}(T)
:=\sum_{k=1}^m\bar\pi_n^k\int_0^T
\sum_{i=1}^n\lVert u_n^{i,k}(s;\delta)\rVert_1\,ds.
\end{equation}
Gross volume counts all opportunists' purchases and sales positively, so
opposing trades contribute even when they cancel in the aggregate.
Figure~\ref{fig:experiments_volume_dispersion} plots \(G_{n,\delta}(T)\)
(left) and the increase per opportunist relative to identical beliefs,
\([G_{n,\delta}(T)-G_{n,0}(T)]/n\) (right), against the root-mean-square
belief dispersion
\(\mathcal D_{n,\delta}:=\allowbreak\left(\frac{1}{n}\sum_{i=1}^n
\lVert\pi_n^i(\delta)-\bar\pi_n\rVert_2^2\right)^{1/2}\).
Greater dispersion can generate substantial additional gross volume despite
unchanged aggregate inventory and impact. At the same positive dispersion,
the additional volume per opportunist is larger for higher \(n\) in these
experiments. At larger \(n\), each opportunist's share of aggregate trading
is smaller. Belief differences therefore more readily turn buying less than
average into selling, or selling less into buying. The resulting opposing
trades increase gross volume per opportunist while canceling in aggregate.

\begin{figure}[!htbp]
\centering
\includegraphics[width=\textwidth]{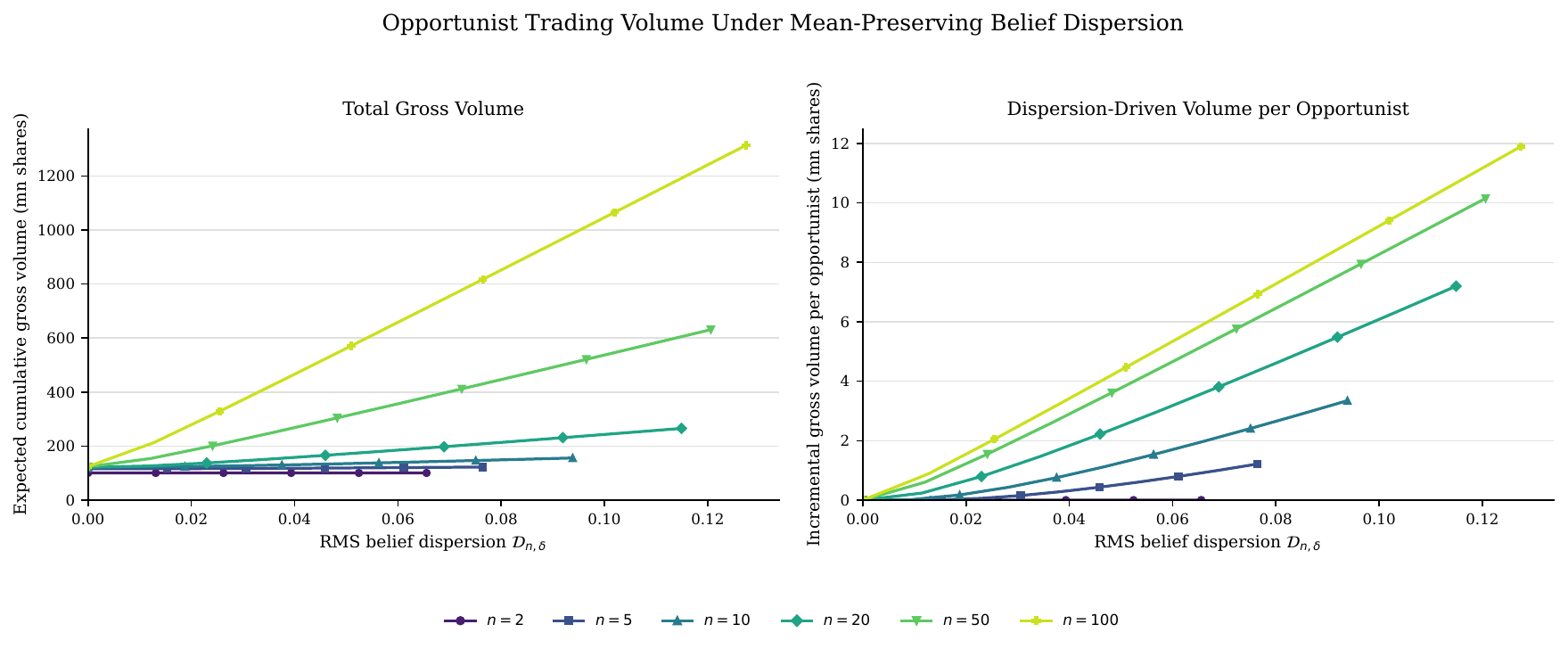}
\caption{Gross volume and additional volume per opportunist versus belief dispersion.}
\label{fig:experiments_volume_dispersion}
\end{figure}

Figure~\ref{fig:experiments_inventory_timing} plots three aggregate inventory
measures.
The first quantity is absolute inventory at \(T_{\mathrm{ann}}\),
averaged under the population-average belief \(\bar\pi\),
\(\mathsf{L}_n^{\mathrm{ann}}:=\sum_{k=1}^m\bar\pi^k\lVert L_n^k(T_{\mathrm{ann}})\rVert_1\). The second is
the corresponding quantity at \(T_{\mathrm{TWAP}}\),
\(\mathsf{L}_n^{\mathrm{TWAP}}:=\sum_{k=1}^m\bar\pi^k\lVert L_n^k(T_{\mathrm{TWAP}})\rVert_1\). The third,
\(\mathsf{L}_n^{\mathrm{chg}}:=\sum_{k=1}^m\bar\pi^k\lVert
L_n^k(T_{\mathrm{impl}})-\allowbreak L_n^k(T_{\mathrm{TWAP}})\rVert_1\), is the
absolute inventory change during the indexer's TWAP, averaged under
\(\bar\pi\).
It measures net aggregate inventory change during the TWAP.
Panel (a) varies the number of opportunists \(n\).\footnote{For the \(n=1\)
benchmark, the game reduces to a linear--quadratic control problem, solved by
dynamic programming with the pre-announcement value matched at
\(T_{\mathrm{ann}}\) to the trader's prior-weighted post-announcement values.}
Aggregate inventory
generally increases with \(n\), with the largest increases at small \(n\).
With multiple opportunists,
\(\mathsf{L}_n^{\mathrm{chg}}>\mathsf{L}_n^{\mathrm{TWAP}}\), reflecting the
inventory sign reversals during the TWAP in
Figure~\ref{fig:experiments_aggregate_inventory}. A single opportunistic trader
does not fully unwind its position during the TWAP.
Panel (b) varies resilience \(\rho\), holding all other inputs fixed. Faster impact decay
reduces all three measures, with a much larger reduction in announcement
inventory than in the subsequent TWAP inventory change.
As resilience increases, less of the impact from early trades survives
into the TWAP, making early accumulation less profitable. Opportunists can
shift their trades closer to \(T_{\mathrm{TWAP}}\), but executing the same
order quantity over a shorter interval raises quadratic execution costs.
This limits how far accumulation can profitably be postponed, so
increasing resilience reduces the amount of inventory accumulated.

Figure~\ref{fig:experiments_impact_sensitivity} shows the corresponding
sensitivity of expected
absolute impact magnitude \(\sum_{k=1}^m\bar\pi^k\lVert I_n^k(t)\rVert_1\)
at \(T_{\mathrm{ann}}\), \(T_{\mathrm{TWAP}}\), \(T_{\mathrm{impl}}\), and
\(T\). In panel (a), greater competition as \(n\) increases narrows the
differences in impact magnitude between the TWAP start, implementation, and
terminal time, thus smoothing out price swings.
The increase from announcement to implementation also shrinks.
In panel (b), higher resilience \(\rho\) lowers impact directly through
faster decay and also indirectly as it discourages inventory buildup, as shown in
Figure~\ref{fig:experiments_inventory_timing}(b).

\begin{figure}[!htbp]
\centering
\includegraphics[width=\textwidth]{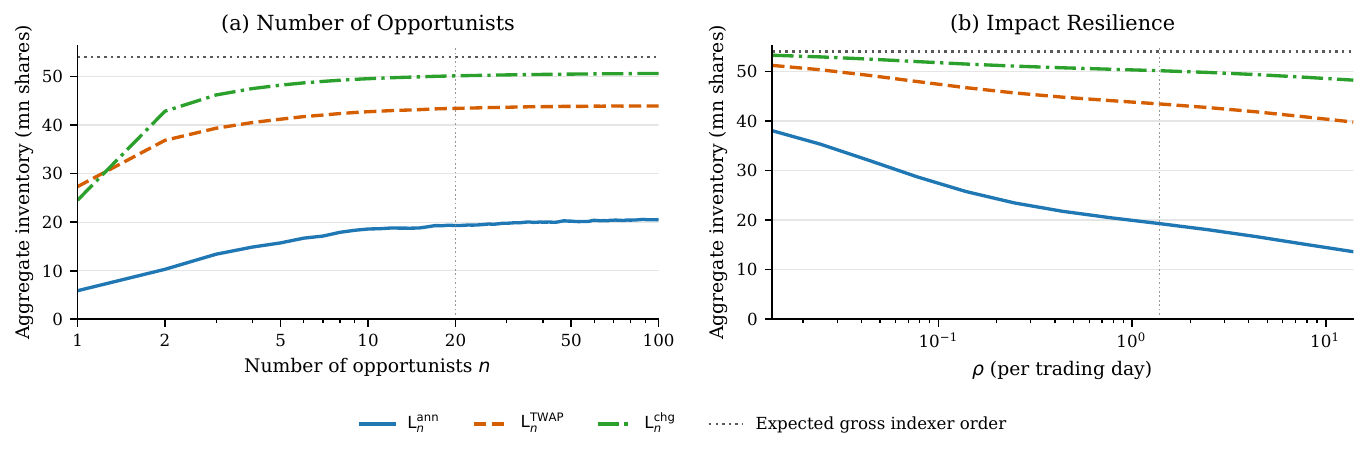}
\caption{Inventory sensitivity to \(n\) and \(\rho\). Dotted lines mark baseline values (vertical) and the expected gross indexer order (horizontal).}
\label{fig:experiments_inventory_timing}
\par\vspace{\baselineskip}
\includegraphics[width=\textwidth]{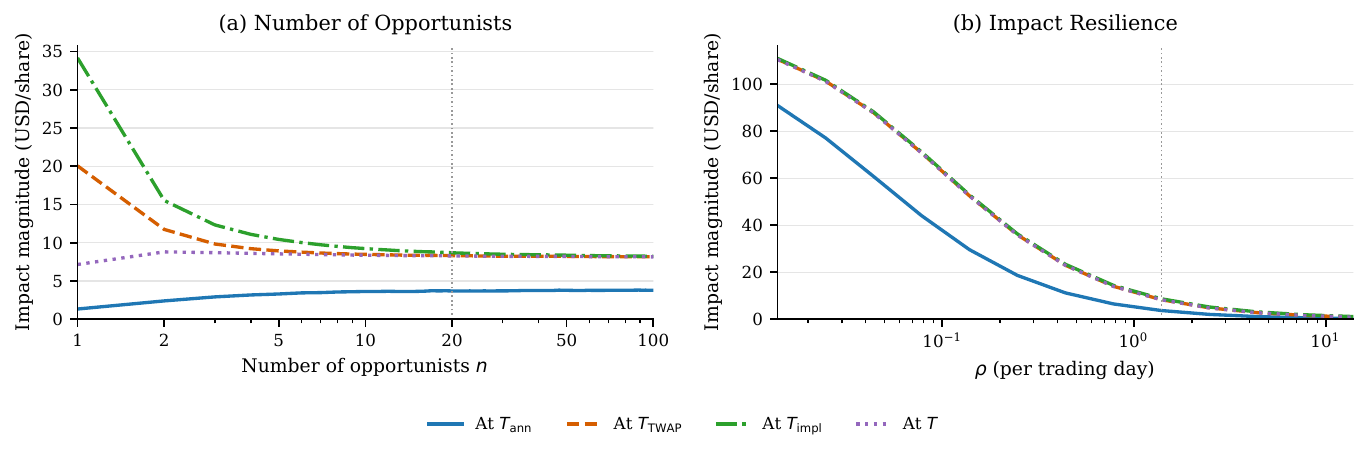}
\caption{Impact sensitivity to \(n\) and \(\rho\). Vertical dotted lines mark baseline values.}
\label{fig:experiments_impact_sensitivity}
\end{figure}

Finally, we compare the indexer's cost with \(n\) opportunists to the cost
without opportunists. For scenario \(k\), we define its implementation
shortfall relative to the fundamental price by
\begin{equation}\label{eq:experiments_indexer_cost}
C_{n,\mathrm{idx}}^k
:=
\int_{T_{\mathrm{TWAP}}}^{T_{\mathrm{impl}}}
(g^k(t))^\top
\left[
I_n^k(t)+\Lambda\bigl(U_n^k(t)+g^k(t)\bigr)
\right]dt.
\end{equation}
Write \(\bar C_{n,\mathrm{idx}}:=\sum_{k=1}^m\bar\pi^k
C_{n,\mathrm{idx}}^k\), and let \(\bar C_{0,\mathrm{idx}}\) denote the
corresponding cost without opportunists.
Figure~\ref{fig:experiments_indexer_cost} plots expected cost savings
\(\bar C_{0,\mathrm{idx}}-\bar C_{n,\mathrm{idx}}\).
In panel (a), savings generally rise with the number of opportunists \(n\), mainly at
small \(n\), as additional opportunists offset more of the indexer's trading
and leave less adverse impact at the TWAP start.
Although not plotted here, both total and per-player expected terminal trading wealth
(cash plus inventory valued at \(s^0\), before inventory penalties) also
decline with \(n\).
Panel (b) varies resilience \(\rho\) and shows a non-monotone pattern in
cost savings. This pattern reflects two competing effects. First,
opportunists' inventory buildup before the TWAP moves prices against the
indexer, increasing the indexer's cost. Second, their offsetting trades during the
TWAP reduce both transient price impact and instantaneous execution costs.
At low \(\rho\), larger inventory buildup supports more trades opposing
the indexer's flow during the TWAP, as reflected in the larger TWAP
inventory change in Figure~\ref{fig:experiments_inventory_timing}(b).
However, the greater adverse impact at the TWAP start, reflected in
Figure~\ref{fig:experiments_impact_sensitivity}(b), makes the indexer's
entry prices less favorable. At low \(\rho\), this inherited impact also
persists longer, becoming permanent at \(\rho=0\), and its cost outweighs
the benefit of offsetting trades, giving negative savings in
Figure~\ref{fig:experiments_indexer_cost}(b).
At the other extreme, very high \(\rho\) discourages inventory buildup,
leaving less opportunist flow to offset the indexer's trades. Transient
impact then decays so quickly that the cost difference is driven mainly by
instantaneous execution costs. The smaller offsetting volume therefore
yields lower cost savings.
The peak at intermediate resilience combines substantial inventory buildup
and opposing trades with enough decay to limit adverse impact inherited
from earlier trading. In these experiments, opportunists therefore lower
indexer costs over most of the plotted range, but raise them when resilience
is very low.

\begin{figure}[!htbp]
\centering
\includegraphics[width=\textwidth]{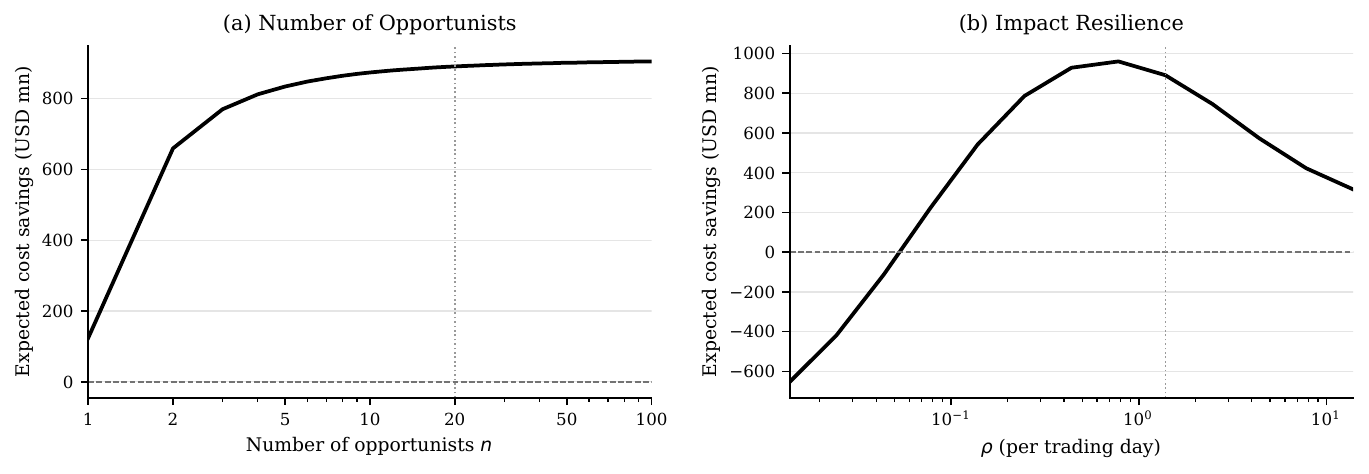}
\caption{Expected indexer cost savings relative to no opportunists.}
\label{fig:experiments_indexer_cost}
\end{figure}

\subsection{Mean-Field Limit}\label{sec:experiments_mean_field}

We compare the finite-player model with \(\varpi_n=30/n\)
and its mean-field limit. The restriction \(\chi_u\varpi_n\leq1\) in
Assumption~\ref{ass:standing_model} requires \(n\geq30\) because \(\chi_u=1\).
The game with \(n\) players uses the first \(n\)
draws from a single sequence of independent beliefs with law \(\Pi\).
Figures~\ref{fig:experiments_mf_inventory_scenario2} and
\ref{fig:experiments_mf_impact_scenario2} compare the finite-player and
mean-field inventory and impact paths when the new index contains assets
\(1\) and \(3\), for
\(n=30,60,150,300\). The paths are close to their mean-field limits at
the larger population sizes. The midprice reacts
almost immediately when trading starts and when the membership scenario
is announced. The legend reports
\(\bar\pi_n^k\) for each finite-player curve and \(\bar\pi^k\) for the
mean-field curve.
\begin{figure}[!htbp]
\centering
\includegraphics[width=\textwidth]{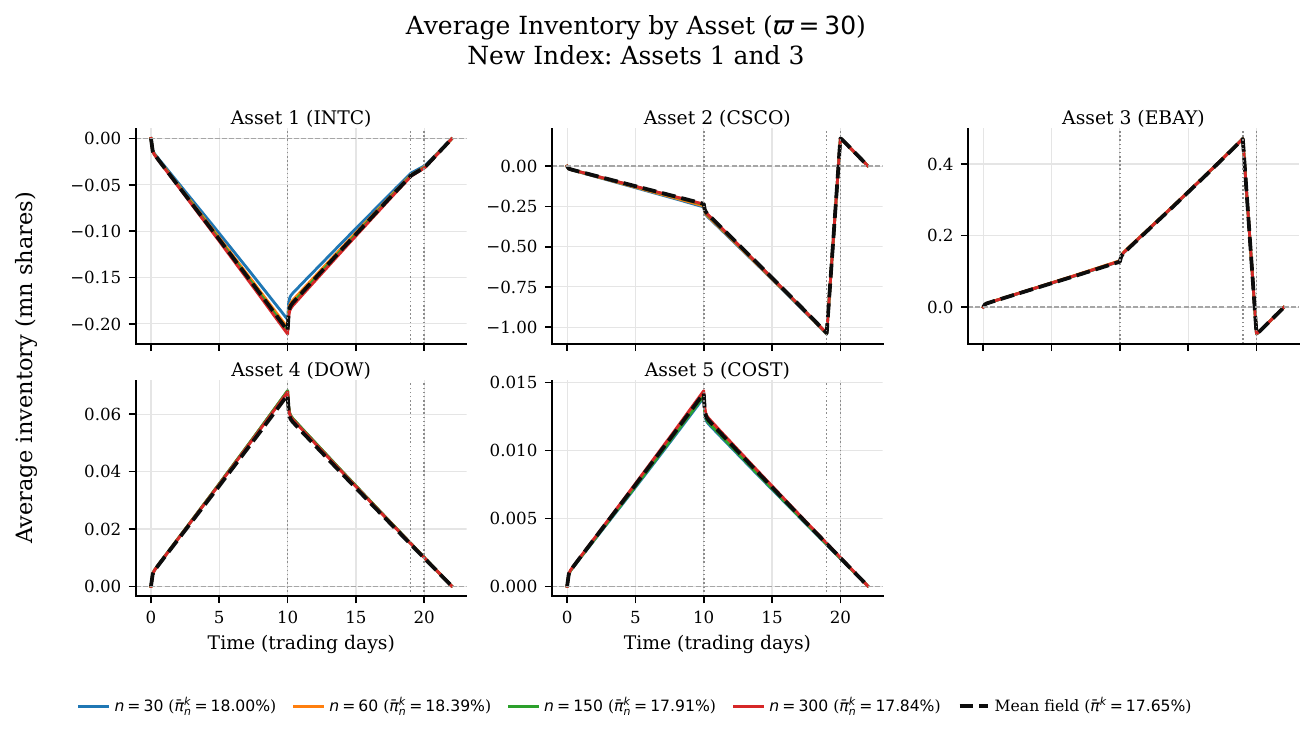}
\caption{Mean-field convergence of average inventory.}
\label{fig:experiments_mf_inventory_scenario2}
\end{figure}

\begin{figure}[!htbp]
\centering
\includegraphics[width=\textwidth]{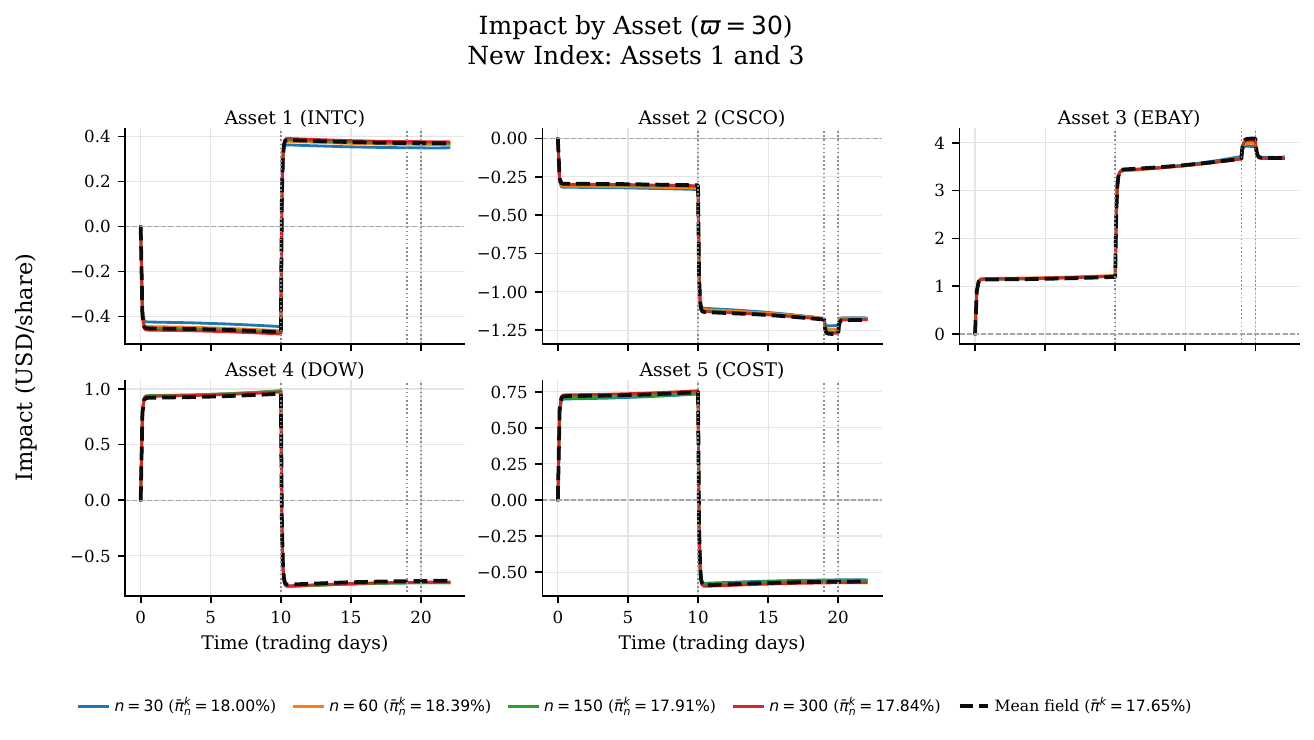}
\caption{Mean-field convergence of price impact.}
\label{fig:experiments_mf_impact_scenario2}
\end{figure}

\endgroup

\FloatBarrier
\section*{Acknowledgements}
J. Blanchet gratefully acknowledges support from ONR under award N00014-24-1-2655, and the National Science Foundation (NSF) under grants 2312204 and 2403007.

\section*{Use of Artificial Intelligence}
We used OpenAI's GPT models in 2026, primarily \mbox{GPT-5.4}, with
\mbox{GPT-5.5}, \mbox{GPT-5.6}, and \mbox{GPT-6} used for later revisions.
These tools assisted with literature
searches, revising author-written drafts, checking mathematical computations
and derivations, carrying out selected matrix calculations, and developing
the numerical implementation from the models and equations in the paper.
The proofs were developed by the authors. We reviewed and checked the
AI-assisted material incorporated into the manuscript and accompanying code
and take full responsibility for their content and accuracy.

\bibliographystyle{plainnat}
\setlength{\bibsep}{6pt}
\bibliography{refs}

@article{micheli2023,
author = {Micheli, Alessandro and Muhle-Karbe, Johannes and Neuman, Eyal},
title = {Closed-Loop {Nash} Competition for Liquidity},
journal = {Mathematical Finance},
volume = {33},
number = {4},
pages = {1082--1118},
doi = {10.1111/mafi.12409},
url = {https://onlinelibrary.wiley.com/doi/abs/10.1111/mafi.12409},
eprint = {https://onlinelibrary.wiley.com/doi/pdf/10.1111/mafi.12409},
year = {2023}
}

@article{lee_ready_1991,
author = {Lee, Charles M. C. and Ready, Mark J.},
title = {Inferring Trade Direction from Intraday Data},
journal = {The Journal of Finance},
volume = {46},
number = {2},
pages = {733--746},
year = {1991},
doi = {10.1111/j.1540-6261.1991.tb02683.x},
url = {https://doi.org/10.1111/j.1540-6261.1991.tb02683.x}
}

@article{campbell_ramadorai_schwartz_2009,
author = {Campbell, John Y. and Ramadorai, Tarun and Schwartz, Allie},
title = {Caught on Tape: Institutional Trading, Stock Returns, and Earnings Announcements},
journal = {Journal of Financial Economics},
volume = {92},
number = {1},
pages = {66--91},
year = {2009},
doi = {10.1016/j.jfineco.2008.03.006},
url = {https://doi.org/10.1016/j.jfineco.2008.03.006}
}

@article{boehmer_jones_zhang_zhang_2021,
author = {Boehmer, Ekkehart and Jones, Charles M. and Zhang, Xiaoyan and Zhang, Xinran},
title = {Tracking Retail Investor Activity},
journal = {The Journal of Finance},
volume = {76},
number = {5},
pages = {2249--2305},
year = {2021},
doi = {10.1111/jofi.13033},
url = {https://doi.org/10.1111/jofi.13033}
}

@misc{CarmonaYang2008PredatoryTrading,
  author = {Ren{\'e} A. Carmona and Z. Joseph Yang},
  title  = {Predatory Trading: A Game on Volatility and Liquidity},
  year   = {2008},
  url    = {https://carmona.scholar.princeton.edu/document/186},
  note   = {Preprint, first version October 2008, revised September 2011}
}

@phdthesis{schoeneborn2008trade,
  author = {Sch{\"o}neborn, Torsten},
  title  = {Trade Execution in Illiquid Markets: Optimal Stochastic Control and Multi-Agent Equilibria},
  school = {Technische Universit{\"a}t Berlin},
  year   = {2008},
  url = {https://d-nb.info/989488934/34}
}

@unpublished{schoeneborn2009liquidation,
  author = {Sch{\"o}neborn, Torsten and Schied, Alexander},
  title  = {Liquidation in the Face of Adversity: Stealth vs. Sunshine Trading},
  year   = {2009},
  month  = mar,
  note   = {EFA 2008 Athens Meetings Paper},
  doi    = {10.2139/ssrn.1007014},
  url    = {https://ssrn.com/abstract=1007014}
}

@article{schied_zhang_2019,
author = {Schied, Alexander and Zhang, Tao},
title = {A Market Impact Game Under Transient Price Impact},
journal = {Mathematics of Operations Research},
volume = {44},
number = {1},
pages = {102--121},
year = {2019},
doi = {10.1287/moor.2017.0916},
URL = {https://doi.org/10.1287/moor.2017.0916},
eprint = {https://doi.org/10.1287/moor.2017.0916}
}

@article{schied_strehle_zhang_2017,
author = {Schied, Alexander and Strehle, Elias and Zhang, Tao},
title = {High-Frequency Limit of {Nash} Equilibria in a Market Impact Game with Transient Price Impact},
journal = {SIAM Journal on Financial Mathematics},
volume = {8},
number = {1},
pages = {589--634},
year = {2017},
doi = {10.1137/16M107030X},
URL = {https://doi.org/10.1137/16M107030X},
eprint = {https://doi.org/10.1137/16M107030X}
}

@article{luo_schied_2019,
author = {Luo, Xiangge and Schied, Alexander},
title = {{Nash} Equilibrium for Risk-Averse Investors in a Market Impact Game with Transient Price Impact},
journal = {Market Microstructure and Liquidity},
volume = {5},
number = {1--4},
pages = {2050001},
year = {2019},
doi = {10.1142/S238262662050001X},
URL = {https://doi.org/10.1142/S238262662050001X},
eprint = {https://doi.org/10.1142/S238262662050001X}
}

@article{strehle_2017,
author = {Strehle, Elias},
title = {Optimal Execution in a Multiplayer Model of Transient Price Impact},
journal = {Market Microstructure and Liquidity},
volume = {3},
number = {3--4},
pages = {1850007},
year = {2017},
doi = {10.1142/S2382626618500077},
URL = {https://doi.org/10.1142/S2382626618500077},
eprint = {https://doi.org/10.1142/S2382626618500077}
}

@article{schied_zhang_2017,
author = {Schied, Alexander and Zhang, Tao},
title = {A State-Constrained Differential Game Arising in Optimal Portfolio Liquidation},
journal = {Mathematical Finance},
volume = {27},
number = {3},
pages = {779--802},
doi = {10.1111/mafi.12108},
url = {https://onlinelibrary.wiley.com/doi/abs/10.1111/mafi.12108},
eprint = {https://onlinelibrary.wiley.com/doi/pdf/10.1111/mafi.12108},
year = {2017}
}

@article{campbell2026optimalexecutionntraders,
  author = {Campbell, Steven and Nutz, Marcel},
  title = {Optimal Execution Among {$N$} Traders With Transient Price Impact},
  journal = {Mathematical Finance},
  pages = {1--43},
  year = {2026},
  doi = {10.1111/mafi.70037},
  url = {https://doi.org/10.1111/mafi.70037},
  eprint = {2501.09638},
  archivePrefix = {arXiv}
}

@article{NeumanVoss2023,
author = {Neuman, Eyal and Vo{\ss}, Moritz},
title = {Trading with the Crowd},
journal = {Mathematical Finance},
volume = {33},
number = {3},
pages = {548--617},
year = {2023},
doi = {10.1111/mafi.12390},
url = {https://doi.org/10.1111/mafi.12390}
}

@unpublished{JiangWuYao2025,
author = {Jiang, Wenxi and Wu, Siyuan and Yao, Chen},
title = {How Index Funds Reshape Intraday Market Dynamics},
year = {2025},
note = {Available at SSRN},
doi = {10.2139/ssrn.3493513},
url = {https://ssrn.com/abstract=3493513}
}

@unpublished{Li2021,
author = {Li, Sida},
title = {Should Passive Investors Actively Manage Their Trades?},
year = {2021},
month = nov,
note = {Available at SSRN},
doi = {10.2139/ssrn.3967799},
url = {https://ssrn.com/abstract=3967799}
}

@article{Ohnishi02102020,
author = {Masamitsu Ohnishi and Makoto Shimoshimizu},
title = {Optimal and Equilibrium Execution Strategies with Generalized Price Impact},
journal = {Quantitative Finance},
volume = {20},
number = {10},
pages = {1625--1644},
year = {2020},
publisher = {Routledge},
doi = {10.1080/14697688.2020.1749294},
URL = {https://doi.org/10.1080/14697688.2020.1749294},
eprint = {https://doi.org/10.1080/14697688.2020.1749294}
}

@article{Ohnishi03072024,
author = {Masamitsu Ohnishi and Makoto Shimoshimizu},
title = {Trade Execution Games in a {Markovian} Environment},
journal = {Applied Mathematical Finance},
volume = {31},
number = {4},
pages = {239--277},
year = {2024},
publisher = {Routledge},
doi = {10.1080/1350486X.2025.2479498},
URL = {https://doi.org/10.1080/1350486X.2025.2479498},
eprint = {https://doi.org/10.1080/1350486X.2025.2479498}
}

@article{Huang04032019,
author = {Xuancheng Huang and Sebastian Jaimungal and Mojtaba Nourian},
title = {Mean-Field Game Strategies for Optimal Execution},
journal = {Applied Mathematical Finance},
volume = {26},
number = {2},
pages = {153--185},
year = {2019},
publisher = {Routledge},
doi = {10.1080/1350486X.2019.1603183},
URL = {https://doi.org/10.1080/1350486X.2019.1603183},
eprint = {https://doi.org/10.1080/1350486X.2019.1603183}
}

@article{cartea_jaimungal_2016,
author = {Cartea, {\'A}lvaro and Jaimungal, Sebastian},
title = {Incorporating Order-Flow into Optimal Execution},
journal = {Mathematics and Financial Economics},
volume = {10},
number = {3},
pages = {339--364},
year = {2016},
doi = {10.1007/s11579-016-0162-z},
URL = {https://doi.org/10.1007/s11579-016-0162-z}
}

@misc{arabneydi2020explicitsequentialequilibrialq,
      title={Explicit Sequential Equilibria in {LQ} Deep Structured Games and Weighted Mean-Field Games},
      author={Jalal Arabneydi and Amir G. Aghdam and Roland P. Malham{\'e}},
      year={2020},
      eprint={1912.03931v4},
      archivePrefix={arXiv},
      primaryClass={math.OC},
      url={https://arxiv.org/abs/1912.03931v4},
  doi = {10.48550/arXiv.1912.03931},
  note = {Version 4, revised October 2020}
}

@article{carlin2007,
author = {Carlin, Bruce Ian and Lobo, Miguel Sousa and Viswanathan, S.},
title = {Episodic Liquidity Crises: Cooperative and Predatory Trading},
journal = {The Journal of Finance},
volume = {62},
number = {5},
pages = {2235--2274},
doi = {10.1111/j.1540-6261.2007.01274.x},
url = {https://onlinelibrary.wiley.com/doi/abs/10.1111/j.1540-6261.2007.01274.x},
eprint = {https://onlinelibrary.wiley.com/doi/pdf/10.1111/j.1540-6261.2007.01274.x},
year = {2007}
}

@article{brunnermeier_2005,
author = {Brunnermeier, Markus K. and Pedersen, Lasse Heje},
title = {Predatory Trading},
journal = {The Journal of Finance},
volume = {60},
number = {4},
pages = {1825--1863},
doi = {10.1111/j.1540-6261.2005.00781.x},
url = {https://onlinelibrary.wiley.com/doi/abs/10.1111/j.1540-6261.2005.00781.x},
eprint = {https://onlinelibrary.wiley.com/doi/pdf/10.1111/j.1540-6261.2005.00781.x},
year = {2005}
}

@article{jaber2026,
author = {{Abi Jaber}, Eduardo and Neuman, Eyal and Tuschmann, Sturmius},
title = {Optimal Portfolio Choice With {Cross-Impact} Propagators},
journal = {Mathematical Finance},
pages = {1--23},
year = {2026},
doi = {10.1111/mafi.70025},
url = {https://onlinelibrary.wiley.com/doi/abs/10.1111/mafi.70025},
eprint = {https://onlinelibrary.wiley.com/doi/pdf/10.1111/mafi.70025}
}

@unpublished{capponi_cont_2020,
author = {Capponi, Francesco and Cont, Rama},
title = {Multi-Asset Market Impact and Order Flow Commonality},
year = {2020},
month = oct,
note = {Available at SSRN 3706390},
doi = {10.2139/ssrn.3706390},
url = {https://ssrn.com/abstract=3706390}
}

@article{cont_cucuringu_zhang_2023,
author = {Cont, Rama and Cucuringu, Mihai and Zhang, Chao},
title = {{Cross-Impact} of Order Flow Imbalance in Equity Markets},
journal = {Quantitative Finance},
volume = {23},
number = {10},
pages = {1373--1393},
year = {2023},
doi = {10.1080/14697688.2023.2236159},
url = {https://www.tandfonline.com/doi/abs/10.1080/14697688.2023.2236159}
}

@article{lecoz_mastromatteo_challet_benzaquen_2024,
author = {{Le Coz}, Victor and Mastromatteo, Iacopo and Challet, Damien and Benzaquen, Michael},
title = {When Is {Cross Impact} Relevant?},
journal = {Quantitative Finance},
volume = {24},
number = {2},
pages = {265--279},
year = {2024},
doi = {10.1080/14697688.2024.2302827},
url = {https://www.tandfonline.com/doi/abs/10.1080/14697688.2024.2302827}
}

@article{alfonsi2016,
author = {Alfonsi, Aur\'{e}lien and Kl\"{o}ck, Florian and Schied, Alexander},
title = {Multivariate Transient Price Impact and Matrix-Valued Positive Definite Functions},
journal = {Mathematics of Operations Research},
volume = {41},
number = {3},
pages = {914--934},
year = {2016},
doi = {10.1287/moor.2015.0761},
URL = {https://doi.org/10.1287/moor.2015.0761},
eprint = {https://doi.org/10.1287/moor.2015.0761}
}

@article{huberman_stanzl,
 ISSN = {00129682, 14680262},
 URL = {http://www.jstor.org/stable/3598784},
 author = {Gur Huberman and Werner Stanzl},
 journal = {Econometrica},
 number = {4},
 pages = {1247--1275},
 publisher = {[Wiley, Econometric Society]},
 title = {Price Manipulation and Quasi-Arbitrage},
 urldate = {2026-07-02},
 volume = {72},
 year = {2004},
  doi = {10.1111/j.1468-0262.2004.00531.x}
}

@article{gatheral_2012,
author = {Gatheral, Jim and Schied, Alexander and Slynko, Alla},
title = {Transient Linear Price Impact and {Fredholm} Integral Equations},
journal = {Mathematical Finance},
volume = {22},
number = {3},
pages = {445--474},
doi = {10.1111/j.1467-9965.2011.00478.x},
url = {https://onlinelibrary.wiley.com/doi/abs/10.1111/j.1467-9965.2011.00478.x},
eprint = {https://onlinelibrary.wiley.com/doi/pdf/10.1111/j.1467-9965.2011.00478.x},
year = {2012}
}

@article{OBIZHAEVA20131,
title = {Optimal Trading Strategy and Supply/Demand Dynamics},
journal = {Journal of Financial Markets},
volume = {16},
number = {1},
pages = {1--32},
year = {2013},
issn = {1386-4181},
doi = {10.1016/j.finmar.2012.09.001},
url = {https://www.sciencedirect.com/science/article/pii/S1386418112000328},
author = {Anna A. Obizhaeva and Jiang Wang}
}

@article{Schneider02012019,
author = {Schneider, Michael and Lillo, Fabrizio},
title = {Cross-Impact and No-Dynamic-Arbitrage},
journal = {Quantitative Finance},
volume = {19},
number = {1},
pages = {137--154},
year = {2019},
publisher = {Routledge},
doi = {10.1080/14697688.2018.1467033},
URL = {https://doi.org/10.1080/14697688.2018.1467033},
eprint = {https://doi.org/10.1080/14697688.2018.1467033}
}

@article{Lacker2020ClosedLoop,
  author = {Lacker, Daniel},
  title = {On the Convergence of Closed-Loop {Nash} Equilibria to the {Mean Field Game} Limit},
  journal = {The Annals of Applied Probability},
  volume = {30},
  number = {4},
  pages = {1693--1761},
  year = {2020},
  doi = {10.1214/19-AAP1541},
  url = {https://doi.org/10.1214/19-AAP1541}
}

@article{Lacker2020,
  title = {Many-Player Games of Optimal Consumption and Investment under Relative Performance Criteria},
  volume = {14},
  ISSN = {1862-9660},
  url = {http://dx.doi.org/10.1007/s11579-019-00255-9},
  DOI = {10.1007/s11579-019-00255-9},
  number = {2},
  journal = {Mathematics and Financial Economics},
  publisher = {Springer Science and Business Media LLC},
  author = {Lacker,  Daniel and Soret,  Agathe},
  year = {2020},
  month = mar,
  pages = {263--281}
}

@article{Lacker2019,
  title = {Mean Field and {$n$}-Agent Games for Optimal Investment under Relative Performance Criteria},
  volume = {29},
  ISSN = {1467-9965},
  url = {http://dx.doi.org/10.1111/mafi.12206},
  DOI = {10.1111/mafi.12206},
  number = {4},
  journal = {Mathematical Finance},
  publisher = {Wiley},
  author = {Lacker,  Daniel and Zariphopoulou,  Thaleia},
  year = {2019},
  month = oct,
  pages = {1003--1038}
}

@article{Casgrain2020,
  title = {Mean-Field Games with Differing Beliefs for Algorithmic Trading},
  volume = {30},
  ISSN = {1467-9965},
  url = {http://dx.doi.org/10.1111/mafi.12237},
  DOI = {10.1111/mafi.12237},
  number = {3},
  journal = {Mathematical Finance},
  publisher = {Wiley},
  author = {Casgrain,  Philippe and Jaimungal,  Sebastian},
  year = {2020},
  month = jul,
  pages = {995--1034}
}

@misc{chen2024periodictradingactivitiesfinancial,
      title={Periodic Trading Activities in Financial Markets: Mean-Field Liquidation Game with Major-Minor Players},
      author={Yufan Chen and Lan Wu and Renyuan Xu and Ruixun Zhang},
      year={2024},
      eprint={2408.09505},
      archivePrefix={arXiv},
      primaryClass={q-fin.MF},
      url={https://arxiv.org/abs/2408.09505},
  doi = {10.48550/arXiv.2408.09505}
}

@article{Shleifer1986,
  author = {Shleifer, Andrei},
  title = {Do Demand Curves for Stocks Slope Down?},
  journal = {The Journal of Finance},
  volume = {41},
  number = {3},
  pages = {579--590},
  year = {1986},
  doi = {10.1111/j.1540-6261.1986.tb04518.x},
  url = {https://doi.org/10.1111/j.1540-6261.1986.tb04518.x}
}

@article{HarrisGurel1986,
  author = {Harris, Lawrence and Gurel, Eitan},
  title = {Price and Volume Effects Associated with Changes in the {S\&P} 500 List: New Evidence for the Existence of Price Pressures},
  journal = {The Journal of Finance},
  volume = {41},
  number = {4},
  pages = {815--829},
  year = {1986},
  doi = {10.1111/j.1540-6261.1986.tb04550.x},
  url = {https://doi.org/10.1111/j.1540-6261.1986.tb04550.x}
}

@article{Petajisto2011,
  author = {Petajisto, Antti},
  title = {The Index Premium and Its Hidden Cost for Index Funds},
  journal = {Journal of Empirical Finance},
  volume = {18},
  number = {2},
  pages = {271--288},
  year = {2011},
  doi = {10.1016/j.jempfin.2010.10.002},
  url = {https://doi.org/10.1016/j.jempfin.2010.10.002}
}

@article{GreenwoodSammon2025,
  author = {Greenwood, Robin and Sammon, Marco},
  title = {The Disappearing Index Effect},
  journal = {The Journal of Finance},
  volume = {80},
  number = {2},
  pages = {657--698},
  year = {2025},
  doi = {10.1111/jofi.13410},
  url = {https://doi.org/10.1111/jofi.13410}
}

@article{ChincoSammon2024,
  author = {Chinco, Alex and Sammon, Marco},
  title = {The Passive Ownership Share Is Double What You Think It Is},
  journal = {Journal of Financial Economics},
  volume = {157},
  pages = {103860},
  year = {2024},
  doi = {10.1016/j.jfineco.2024.103860},
  url = {https://doi.org/10.1016/j.jfineco.2024.103860}
}

@article{MicheliNeuman2022,
  author = {Micheli, Alessandro and Neuman, Eyal},
  title = {Evidence of Crowding on {Russell} 3000 Reconstitution Events},
  journal = {Market Microstructure and Liquidity},
  volume = {6},
  number = {1--4},
  pages = {2050009},
  year = {2020},
  doi = {10.1142/S2382626620500094},
  url = {https://doi.org/10.1142/S2382626620500094}
}

@unpublished{PegoraroSammonShim2026,
  author = {Pegoraro, Stefano and Sammon, Marco and Shim, John J.},
  title = {Optimal Index-Linked Rebalancing with Anticipatory Trading},
  year = {2026},
  month = jun,
  note = {Working paper, SSRN 6772502},
  doi = {10.2139/ssrn.6772502},
  url = {https://ssrn.com/abstract=6772502}
}

@misc{ChengWangXu2024,
  author = {Cheng, Xue and Wang, Meng and Xu, Ziyi},
  title = {Mean Field Game of High-Frequency Anticipatory Trading},
  year = {2024},
  eprint = {2404.18200},
  archivePrefix = {arXiv},
  primaryClass = {q-fin.MF},
  doi = {10.48550/arXiv.2404.18200},
  url = {https://arxiv.org/abs/2404.18200}
}

@article{DrapeauLuoSchiedXiong2021,
  author = {Drapeau, Samuel and Luo, Peng and Schied, Alexander and Xiong, Dewen},
  title = {An {FBSDE} Approach to Market Impact Games with Stochastic Parameters},
  journal = {Probability, Uncertainty and Quantitative Risk},
  volume = {6},
  number = {3},
  pages = {237--260},
  year = {2021},
  doi = {10.3934/puqr.2021012},
  url = {https://doi.org/10.3934/puqr.2021012}
}

@article{ChoiLarsenSeppi2019,
  author = {Choi, Jin Hyuk and Larsen, Kasper and Seppi, Duane J.},
  title = {Information and Trading Targets in a Dynamic Market Equilibrium},
  journal = {Journal of Financial Economics},
  volume = {132},
  number = {3},
  pages = {22--49},
  year = {2019},
  doi = {10.1016/j.jfineco.2018.11.003},
  url = {https://doi.org/10.1016/j.jfineco.2018.11.003}
}

@article{ChoiLarsenSeppi2021Benchmarks,
  author = {Choi, Jin Hyuk and Larsen, Kasper and Seppi, Duane J.},
  title = {Equilibrium Effects of Intraday Order-Splitting Benchmarks},
  journal = {Mathematics and Financial Economics},
  volume = {15},
  number = {2},
  pages = {315--352},
  year = {2021},
  doi = {10.1007/s11579-020-00278-7},
  url = {https://doi.org/10.1007/s11579-020-00278-7}
}

@article{ChenChoiLarsenSeppi2022,
  author = {Chen, Xiao and Choi, Jin Hyuk and Larsen, Kasper and Seppi, Duane J.},
  title = {Learning About Latent Dynamic Trading Demand},
  journal = {Mathematics and Financial Economics},
  volume = {16},
  pages = {615--658},
  year = {2022},
  doi = {10.1007/s11579-022-00317-5},
  url = {https://doi.org/10.1007/s11579-022-00317-5},
  number = {4}
}

@article{ChenChoiLarsenSeppi2023,
  author = {Chen, Xiao and Choi, Jin Hyuk and Larsen, Kasper and Seppi, Duane J.},
  title = {Price Impact in {Nash} Equilibria},
  journal = {Finance and Stochastics},
  volume = {27},
  pages = {305--340},
  year = {2023},
  doi = {10.1007/s00780-023-00499-w},
  url = {https://doi.org/10.1007/s00780-023-00499-w},
  number = {2}
}

@article{LackerLeFlem2023,
  author = {Lacker, Daniel and {Le Flem}, Luc},
  title = {Closed-Loop Convergence for Mean Field Games with Common Noise},
  journal = {The Annals of Applied Probability},
  volume = {33},
  number = {4},
  pages = {2681--2733},
  year = {2023},
  doi = {10.1214/22-AAP1876},
  url = {https://doi.org/10.1214/22-AAP1876}
}

@article{LauriereTangpi2022,
  author = {Lauri{\`e}re, Mathieu and Tangpi, Ludovic},
  title = {Convergence of Large Population Games to Mean Field Games with Interaction Through the Controls},
  journal = {SIAM Journal on Mathematical Analysis},
  volume = {54},
  number = {3},
  pages = {3535--3574},
  year = {2022},
  doi = {10.1137/22M1469328},
  url = {https://doi.org/10.1137/22M1469328}
}

@unpublished{Nathan2026,
  author = {Nathan, Daniel},
  title = {Anticipated Orders and the Measurement of Stock Demand Elasticity: Evidence from Scheduled Index Rebalances},
  year = {2026},
  month = aug,
  note = {Working paper, SSRN 7332418},
  doi = {10.2139/ssrn.7332418},
  url = {https://ssrn.com/abstract=7332418}
}

@misc{Tasitsiomi2025,
  author = {Tasitsiomi, Iro},
  title = {On the Hidden Costs of Passive Investing},
  year = {2025},
  note = {Version 3, revised October 2025},
  eprint = {2506.21775},
  archivePrefix = {arXiv},
  primaryClass = {q-fin.TR},
  doi = {10.48550/arXiv.2506.21775},
  url = {https://arxiv.org/abs/2506.21775}
}

@article{BessembinderCarrionTuttleVenkataraman2016,
  author = {Bessembinder, Hendrik and Carrion, Allen and Tuttle, Laura and Venkataraman, Kumar},
  title = {Liquidity, Resiliency and Market Quality around Predictable Trades: Theory and Evidence},
  journal = {Journal of Financial Economics},
  volume = {121},
  number = {1},
  pages = {142--166},
  year = {2016},
  doi = {10.1016/j.jfineco.2016.02.011},
  url = {https://doi.org/10.1016/j.jfineco.2016.02.011}
}

@article{CordoniLillo2024,
  author = {Cordoni, Francesco and Lillo, Fabrizio},
  title = {Instabilities in Multi-Asset and Multi-Agent Market Impact Games},
  journal = {Annals of Operations Research},
  volume = {336},
  number = {1--2},
  pages = {505--539},
  year = {2024},
  doi = {10.1007/s10479-022-05066-8},
  url = {https://doi.org/10.1007/s10479-022-05066-8}
}

@article{CarteaJaimungalSanchezBetancourt2026,
  author = {Cartea, {\'A}lvaro and Jaimungal, Sebastian and S{\'a}nchez-Betancourt, Leandro},
  title = {{Nash} Equilibrium between Brokers and Traders},
  journal = {Finance and Stochastics},
  volume = {30},
  number = {3},
  pages = {951--982},
  year = {2026},
  doi = {10.1007/s00780-026-00595-7},
  url = {https://doi.org/10.1007/s00780-026-00595-7}
}

@misc{AbiJaberNeumanVoss2023,
  author = {{Abi Jaber}, Eduardo and Neuman, Eyal and Vo{\ss}, Moritz},
  title = {Equilibrium in Functional Stochastic Games with Mean-Field Interaction},
  year = {2023},
  note = {arXiv:2306.05433, version 3, revised July 2026},
  eprint = {2306.05433},
  archivePrefix = {arXiv},
  primaryClass = {math.OC},
  doi = {10.48550/arXiv.2306.05433},
  url = {https://arxiv.org/abs/2306.05433}
}

@misc{LeclereRosenbaum2026,
  author = {Lecl{\`e}re, Joseph and Rosenbaum, Mathieu},
  title = {Mean-Field Equilibrium of Heterogeneous Agents under Market Impact},
  year = {2026},
  eprint = {2609.03115},
  archivePrefix = {arXiv},
  primaryClass = {q-fin.TR},
  doi = {10.48550/arXiv.2609.03115},
  url = {https://arxiv.org/abs/2609.03115}
}

@article{Herzing2026,
  author = {Herzing, Tobias J.},
  title = {Dealing with Information-Uncertainty: Insights from a Predatory Trading Game},
  journal = {Finance Research Letters},
  volume = {91},
  pages = {109499},
  year = {2026},
  doi = {10.1016/j.frl.2026.109499},
  url = {https://doi.org/10.1016/j.frl.2026.109499}
}

@article{CampbellNutzRandomization,
  author = {Campbell, Steven and Nutz, Marcel},
  title = {Randomization in Optimal Execution Games},
  journal = {SIAM Journal on Financial Mathematics},
  volume = {17},
  number = {3},
  pages = {1419--1449},
  year = {2026},
  doi = {10.1137/25M1741108},
  url = {https://doi.org/10.1137/25M1741108}
}

@article{CardaliaguetLehalle2018,
  author = {Cardaliaguet, Pierre and Lehalle, Charles-Albert},
  title = {Mean Field Game of Controls and an Application to Trade Crowding},
  journal = {Mathematics and Financial Economics},
  volume = {12},
  number = {3},
  pages = {335--363},
  year = {2018},
  doi = {10.1007/s11579-017-0206-z},
  url = {https://doi.org/10.1007/s11579-017-0206-z}
}

@misc{LehalleMouzouni2019,
  author = {Lehalle, Charles-Albert and Mouzouni, Charafeddine},
  title = {A Mean Field Game of Portfolio Trading and Its Consequences on Perceived Correlations},
  year = {2019},
  eprint = {1902.09606},
  archivePrefix = {arXiv},
  primaryClass = {q-fin.TR},
  doi = {10.48550/arXiv.1902.09606},
  url = {https://arxiv.org/abs/1902.09606}
}

@misc{NutzProsperi2025,
  author = {Nutz, Marcel and Prosperi, Alessandro},
  title = {High-Frequency Analysis of a Trading Game with Transient Price Impact},
  year = {2025},
  note = {arXiv:2512.11765},
  eprint = {2512.11765},
  archivePrefix = {arXiv},
  primaryClass = {q-fin.TR},
  doi = {10.48550/arXiv.2512.11765},
  url = {https://arxiv.org/abs/2512.11765}
}

\clearpage
\appendix
\section{Proofs of Section~\ref{sec:solution}}\label{sec:proofs_solution}

\subsection{Proof of Theorem~\ref{thm:stitched_hjb_verification}}
\label{subsec:proof_stitched_hjb_verification}

\begin{proof}
Fix a continuation time \(s\) and prescribed histories
\((X_n,I,A,S^0)_{[0,s]}\), with \(X_n(s)=x_n\), \(I(s)=\iota\), and
\(A(s)=a\).
We first verify the best-response inequality for a continuation game starting
at or after the announcement, so \(s\geq T_{\mathrm{ann}}\) and \(a=\xi^k\).
Let trader \(i\) choose an admissible unilateral continuation deviation
\(\Psi_n^i\), which induces a potentially stochastic rate \(u_n^i(t)\) through
its dependence on the observed history of \(S^0\).
All trading rates and controlled states are evaluated under
\((\Psi_n^i,\phi_n^{-i})\) in scenario \(k\).
By admissibility, for almost every price path on \([0,T]\) agreeing with
the prescribed price history on \([0,s]\) under trader \(i\)'s chosen
continuation law, the integral equations defining inventory and impact
hold for every \(t\in[s,T]\) and all realized trading rates
belong to \(L^2([s,T];\mathbb R^d)\).
Fix one such path \(\mathbf{s}^0\), suppressing dependence on
\((\mathbf{s}^0,k)\) in the notation for the rates and controlled states.
Then \(Z_n^i\) is absolutely continuous.
The chain rule and equation \eqref{eq:stitched_hjb}
give for almost every \(t\) outside the switching times of \(g^k\)
\[
\begin{aligned}
&\mathcal L\left(
X_n^i(t),u_n^i(t),
I(t),\varpi_nU_n^{-i}(t),g^k(t)
\right)
+\frac{d}{dt}V_n^{i,k}\bigl(t,Z_n^i(t)\bigr)
\\
&\quad=
\partial_tV_n^{i,k}\bigl(t,Z_n^i(t)\bigr)
+
\mathscr H_n^{i,k}\left(
t,X_n(t),I(t),u_n^i(t),
\nabla V_n^{i,k}\bigl(t,Z_n^i(t)\bigr);
\phi_n^{-i,k}
\right)
\\
&\quad=
\mathscr H_n^{i,k}\left(
t,X_n(t),I(t),u_n^i(t),
\nabla V_n^{i,k}\bigl(t,Z_n^i(t)\bigr);
\phi_n^{-i,k}
\right)
\\
&\qquad-
\min_{v\in\mathbb R^d}
\mathscr H_n^{i,k}\left(
t,X_n(t),I(t),v,
\nabla V_n^{i,k}\bigl(t,Z_n^i(t)\bigr);
\phi_n^{-i,k}
\right)
\geq 0.
\end{aligned}
\]
Integrating between the switching times of \(g^k\), passing to their
endpoints by continuity, and using the terminal condition
\eqref{eq:stitched_hjb_terminal} gives
\begin{equation}\label{eq:post_announcement_pathwise_bound}
\begin{aligned}
V_n^{i,k}\left(s,z^i(x_n^i,\bar x_n,\iota)\right)
\leq{}&
\int_s^T
\mathcal L\left(
X_n^i(t),u_n^i(t),
I(t),\varpi_nU_n^{-i}(t),g^k(t)
\right)\,dt
\\
&+
\frac{1}{2}X_n^i(T)^\top Q_TX_n^i(T).
\end{aligned}
\end{equation}
The total reduced cost is integrable by admissibility. Taking expectation gives
\[
J_n^i\left(
\Psi_n^i;\phi_n^{-i}
\mid s,(X_n,I,A,S^0)_{[0,s]}
\right)
\geq
V_n^{i,k}\left(s,z^i(x_n^i,\bar x_n,\iota)\right).
\]

For \(s<T_{\mathrm{ann}}\), admissibility and independence of \(\Xi\) and
\(S^0\) ensure that, for almost every price path on \([0,T]\) agreeing with
the prescribed price history on \([0,s]\) under trader \(i\)'s chosen
continuation law, the integral equations defining inventory and impact
hold for every \(t\in[s,T]\) and all realized trading rates
belong to \(L^2([s,T];\mathbb R^d)\) in every scenario with
\(\pi_n^{i,k}>0\).
Fix one such path \(\mathbf{s}^0\), again suppressing dependence on
\(\mathbf{s}^0\) in the notation for the rates and controlled states.

The pre-announcement HJB inequality and
\eqref{eq:stitched_hjb_value_matching} give
\begin{subequations}\label{eq:pre_announcement_pathwise_bound}
\begin{equation}\label{eq:pre_announcement_hjb_bound}
\begin{aligned}
V_n^{i,0}\left(s,z^i(x_n^i,\bar x_n,\iota)\right)
&\leq
\int_s^{T_{\mathrm{ann}}}
\mathcal L\left(
X_n^i(t),u_n^i(t),
I(t),\varpi_nU_n^{-i}(t),0
\right)\,dt
\\
&\quad+
\sum_{k=1}^m\pi_n^{i,k}
V_n^{i,k}\bigl(
T_{\mathrm{ann}},Z_n^i(T_{\mathrm{ann}})
\bigr).
\end{aligned}
\end{equation}
\begin{samepage}
Apply \eqref{eq:post_announcement_pathwise_bound} from
\(T_{\mathrm{ann}}\) in each scenario with \(\pi_n^{i,k}>0\), with
post-announcement processes evaluated under \(\Xi=\xi^k\).
With the entire price path \(S^0|_{[0,T]}=\mathbf{s}^0\) fixed, independence
of \(\Xi\) and \(S^0\) allows us to average only over \(\Xi\), with weights
\(\pi_n^{i,k}\), giving
\begin{equation}\label{eq:pre_announcement_scenario_bound}
\begin{aligned}
&\sum_{k=1}^m\pi_n^{i,k}
V_n^{i,k}\bigl(T_{\mathrm{ann}},Z_n^i(T_{\mathrm{ann}})\bigr)
\\
&\quad\leq
\mathbb E_n^i\Bigg[
\int_{T_{\mathrm{ann}}}^T
\mathcal L\left(
X_n^i(t),u_n^i(t),
I(t),\varpi_nU_n^{-i}(t),g(t)
\right)\,dt
\\
&\qquad\phantom{\mathbb E_n^i\Bigg[}
+\frac{1}{2}X_n^i(T)^\top Q_TX_n^i(T)
\,\Bigm|\,
S^0|_{[0,T]}=\mathbf{s}^0,\ A(s)=0
\Bigg].
\end{aligned}
\end{equation}
\end{samepage}
\end{subequations}
The total reduced cost is integrable by admissibility.
Combining \eqref{eq:pre_announcement_hjb_bound} and
\eqref{eq:pre_announcement_scenario_bound} and averaging over the future
price path \(S^0|_{[s,T]}\) under trader \(i\)'s chosen continuation law,
with the prescribed \(S^0|_{[0,s]}\) and \(A(s)=0\) held fixed, gives
\[
J_n^i\left(
\Psi_n^i;\phi_n^{-i}
\mid s,(X_n,I,A,S^0)_{[0,s]}
\right)
\geq
V_n^{i,0}\left(s,z^i(x_n^i,\bar x_n,\iota)\right).
\]

For \(\Psi_n^i=\phi_n^i\), \eqref{eq:stitched_hjb_argmin} makes all inequalities
equalities. Thus the candidate attains the lower bounds above, whereas every
admissible deviation has weakly larger cost. Since the trader and
continuation game were arbitrary, \eqref{eq:subgame_perfect_ne} holds, and the
asserted continuation-value identities follow.
\end{proof}

\subsection{Proof of Theorem~\ref{thm:full_horizon_closed_loop_equilibrium}}
\label{subsec:proof_full_horizon_closed_loop_equilibrium}

\begin{proof}
We show that the feedbacks and value functions stated in
Theorem~\ref{thm:full_horizon_closed_loop_equilibrium} satisfy the conditions
of the verification Theorem~\ref{thm:stitched_hjb_verification}.
Fix either the pre-announcement interval, denoted by \(k=0\), or the
post-announcement interval together with a scenario
\(k\in\{1,\ldots,m\}\). Suppress time and the fixed index \(k\), and write
\[
V_n^i(z^i)
=\frac12(z^i)^\top H_nz^i+(r_n^i)^\top z^i+c_n^i.
\]
For \(k=0\), \(r_n^i=r_n^{i,0}\) and \(c_n^i=c_n^{i,0}\). For \(k\geq1\),
\(r_n^i=r_n^k\) and \(c_n^i=c_n^k\), so these coefficients are independent
of \(i\). Also write \(g:=g^k\), where \(g^0\equiv0\), and
\(U_n^{-i}:=U_n^{-i,k}\). We first perform the HJB calculation for generic
\(r_n^i,c_n^i,g\) and specialize afterward. For \(k\geq1\), the calculation
applies on each interval between jump times of \(g^k\). No derivative of
\(g^k\) appears, so the coefficient solutions can be pasted continuously
across these times.

Substituting
\(\nabla V_n^i=H_nz^i+r_n^i\) in \eqref{eq:stitched_hamiltonian} yields:
\begin{align}
&\mathscr H_n^{i,k}
\left(t,\!x_n,\!\iota,\!v,\!H_nz^i+r_n^i;\phi_n^{-i,k}\right)
={}\frac12(z^i)^\top Q_zz^i
+\left(A_zz^i+B_n^{-}U_n^{-i}+B_gg\right)^\top(H_nz^i+r_n^i)\notag
\\
&\quad+v^\top\Lambda v
+\left(\Theta_n(H_n)z^i+B_n^\top r_n^i+\chi_g\Lambda g
+\vartheta_n\Lambda U_n^{-i}\right)^\top v.
\label{eq:vectorized_quadratic_hamiltonian}
\end{align}
The minimizing feedbacks satisfy the following simultaneous first-order conditions,
which are sufficient because the Hamiltonian has Hessian
\(2\Lambda\succ0\) in \(v\):
\begin{equation}\label{eq:vectorized_quadratic_foc}
2\Lambda\phi_n^i+\Theta_n(H_n)z^i+B_n^\top r_n^i+\chi_g\Lambda g
+\vartheta_n\Lambda U_n^{-i}=0,
\qquad i=1,\ldots,n.
\end{equation}
Let \(U_n:=\sum_{j=1}^n\phi_n^j\). Since
\(U_n^{-i}=U_n-\phi_n^i\), solving
\eqref{eq:vectorized_quadratic_foc} for \(\phi_n^i\) gives
\begin{equation}\label{eq:vectorized_quadratic_rate_given_aggregate}
\phi_n^i
=-\frac1{a_n}\Lambda^{-1}
\left(\Theta_n(H_n)z^i+B_n^\top r_n^i+\chi_g\Lambda g
+\vartheta_n\Lambda U_n\right).
\end{equation}
Summing this equation over \(i\) gives
\[
U_n
=-\frac1{a_n}\Lambda^{-1}
\left(
\Theta_n(H_n)\sum_{j=1}^nz^j+B_n^\top\sum_{j=1}^nr_n^j+n\chi_g\Lambda g
\right)
-\frac{n\vartheta_n}{a_n}U_n.
\]
Using \(d_n=a_n+n\vartheta_n\) and solving for \(U_n\) gives
\begin{equation}\label{eq:vectorized_quadratic_aggregate_rate}
U_n
=-\frac1{d_n}\Lambda^{-1}
\left(
\Theta_n(H_n)\sum_{j=1}^nz^j+
B_n^\top\sum_{j=1}^nr_n^j+n\chi_g\Lambda g
\right).
\end{equation}
Since \(\sum_{j=1}^n\delta^j=0\), the definitions of \(D_n\) and \(F_n\)
give
\begin{equation}\label{eq:vectorized_quadratic_state_identities}
\frac1{a_n}z^i-\frac{\vartheta_n}{a_nd_n}\sum_{j=1}^nz^j=D_nz^i,
\qquad
\frac1{a_n}z^i-\frac2{a_nd_n}\sum_{j=1}^nz^j=F_nz^i.
\end{equation}
Substituting \eqref{eq:vectorized_quadratic_aggregate_rate} into
\eqref{eq:vectorized_quadratic_rate_given_aggregate}, applying the first identity in
\eqref{eq:vectorized_quadratic_state_identities}, and then using
\(U_n^{-i}=U_n-\phi_n^i\) with the second identity in
\eqref{eq:vectorized_quadratic_state_identities} gives
\begin{subequations}\label{eq:vectorized_quadratic_feedbacks}
\begin{align}
\phi_n^i
&=-\Lambda^{-1}\left[
\Theta_n(H_n)D_nz^i+\frac1{a_n}\bigl(B_n^\top r_n^i+\chi_g\Lambda g\bigr)
-\frac{\vartheta_n}{a_nd_n}\left(B_n^\top\sum_{j=1}^nr_n^j+n\chi_g\Lambda g\right)
\right],
\label{eq:vectorized_quadratic_individual_feedback}
\\
U_n^{-i}
&=\Lambda^{-1}\left[
\Theta_n(H_n)F_nz^i+\frac1{a_n}\bigl(B_n^\top r_n^i+\chi_g\Lambda g\bigr)
-\frac2{a_nd_n}\left(B_n^\top\sum_{j=1}^nr_n^j+n\chi_g\Lambda g\right)
\right].
\label{eq:vectorized_quadratic_opponent_feedback}
\end{align}
\end{subequations}
Substituting \eqref{eq:vectorized_quadratic_feedbacks} into the HJB and using
\eqref{eq:vectorized_quadratic_foc} gives
\begin{equation}\label{eq:vectorized_quadratic_hjb_identity}
\begin{aligned}
0={}&\frac12(z^i)^\top\dot H_nz^i+(z^i)^\top\dot r_n^i+\dot c_n^i
+\frac12(z^i)^\top Q_zz^i
\\
&+\left(A_zz^i+B_n^{-}U_n^{-i}+B_gg\right)^\top(H_nz^i+r_n^i)
-(\phi_n^i)^\top\Lambda\phi_n^i,
\end{aligned}
\end{equation}
Grouping the terms quadratic in \(z^i\) gives the same Riccati equation
\eqref{eq:fixed_interval_riccati_P} for every \(k=0,\ldots,m\), hence over the
full horizon \([0,T]\). The linear and constant coefficients are determined
separately before and after the announcement.

After the announcement \(r_n^i=r_n^k\), \(c_n^i=c_n^k\), and \(g=g^k\)
for every trader. With this the individual formula in
\eqref{eq:vectorized_quadratic_individual_feedback} becomes \(\phi_n^{i,k}\) in
\eqref{eq:full_horizon_markov_feedback}, while
\begin{equation}\label{eq:vectorized_quadratic_post_opponent_feedback}
U_n^{-i,k}
=\Lambda^{-1}\left[
\Theta_n(H_n)F_nz^i-\frac{n-1}{d_n}
\bigl(B_n^\top r_n^k+\chi_g\Lambda g^k\bigr)\right].
\end{equation}
Substituting \(\phi_n^{i,k}\) from \eqref{eq:full_horizon_markov_feedback_post}
and \(U_n^{-i,k}\) from \eqref{eq:vectorized_quadratic_post_opponent_feedback}
into \eqref{eq:vectorized_quadratic_hjb_identity}, grouping the terms linear
in \(z^i\) and the constant terms, and using
\eqref{eq:pre_ann_matrix_maps_def} gives
\eqref{eq:post_ann_common_r} and \eqref{eq:post_ann_common_alpha}.

Before the announcement, \(r_n^i=r_n^{i,0}\), \(c_n^i=c_n^{i,0}\), and
\(g=0\). Using the definition of \(r_n^0\),
\eqref{eq:vectorized_quadratic_individual_feedback} becomes the
pre-announcement feedback
\eqref{eq:full_horizon_markov_feedback_pre}.
Substituting \(\phi_n^{i,0}\) from \eqref{eq:full_horizon_markov_feedback_pre}
and \(U_n^{-i,0}\) from \eqref{eq:vectorized_quadratic_opponent_feedback}
into \eqref{eq:vectorized_quadratic_hjb_identity}, grouping the terms linear
in \(z^i\) and the constant terms, and using
\eqref{eq:pre_ann_matrix_maps_def} gives the differential equations in
\eqref{eq:pre_ann_individual_r_ode} and \eqref{eq:pre_ann_constant_ode}.
Matching the linear and constant coefficients in
\eqref{eq:stitched_hjb_value_matching} gives the boundary conditions for
\(r_n^{i,0}\) and \(c_n^{i,0}\).
Averaging \eqref{eq:pre_ann_individual_r_ode} gives
\eqref{eq:pre_ann_common_r_ode}.

The coefficient boundary conditions establish
\eqref{eq:stitched_hjb_terminal} and
\eqref{eq:stitched_hjb_value_matching}. Since \(H_n\in C^1([0,T])\), the
linear coefficient equations have bounded coefficients and piecewise
continuous forcing, so they admit unique finite solutions. The constant
coefficient equations then have piecewise continuous right-hand sides and
unique finite solutions given by integration. All coefficient solutions are
continuous on their respective intervals and \(C^1\) between switching times.
In every continuation game, the affine
feedbacks induce a unique linear closed-loop system with square-integrable
rates and are therefore admissible. The claim follows from
Theorem~\ref{thm:stitched_hjb_verification}.
\end{proof}

\subsection{Proof of Theorem~\ref{thm:global_riccati_solvability}}
\label{subsec:proof_global_riccati_solvability}

\begin{proof}
Fix \(n\geq2\) satisfying \(\chi_u\varpi_n\leq1\).  The right-hand side of
\eqref{eq:fixed_interval_riccati_P} is polynomial in the entries of \(H_n\),
so there is a unique local symmetric solution backward from \(H_T\).  Let
\((\underline t_n,T]\) be its maximal backward interval.  If
\(\underline t_n>0\), the continuation criterion gives
\(\|H_n(t)\|\to\infty\) as
\(t\downarrow\underline t_n\). Hence it suffices to prove the finite bound
\eqref{eq:global_riccati_two_sided_bound} on \((\underline t_n,T]\), which
rules out such blow-up.

The calculations in the proofs of
Theorems~\ref{thm:stitched_hjb_verification} and
\ref{thm:full_horizon_closed_loop_equilibrium} show that, whenever \(H_n\)
exists on \([t,T]\), \(\frac12z^\top H_nz\) is the equilibrium value
function of the \(n\)-player game without indexer trading, obtained by setting
\(g^k\equiv0\) for every \(k\). In this game the announcement scenarios are
identical, so no stitching is required and the affine and constant
coefficients vanish. Since \(H_n\) is bounded on \([t,T]\), the induced
feedbacks are admissible. Theorem~\ref{thm:stitched_hjb_verification} then
yields \(V_n^i(t,x_n,\iota)=\frac12z^\top H_n(t)z\),
\(z=(\delta,\bar x,\iota)^\top\).

Fix \(t\in(\underline t_n,T]\). Trader \(i\) may
deviate by choosing \(u_n^i\equiv0\) on \([t,T]\). Its inventory stays at
\(\bar x+\delta\). Because \(u_n^i\equiv0\), the execution-price term
vanishes, irrespective of the opponents' controls and the impact state. Hence
\(V_n^i(t,x_n,\iota)
\leq\frac12(\bar x+\delta)^\top G_Q(t)(\bar x+\delta)\)
for every \(z\), which implies
\begin{equation}
\label{eq:uniform_solvability_upper_bound_proof}
H_n(t)\preceq\overline H(t).
\end{equation}

For the lower bound, fix \(\delta,\bar x,\iota\in\mathbb R^d\) and choose
\(X_n^i(t)=\bar x+\delta\),
\(X_n^j(t)=\bar x-\frac{\delta}{n-1}\) for \(j\ne i\)
so that the average of the \(X^j_n(t)\) is \(\bar x\).  Let \(V_n^j(t)\) be trader \(j\)'s equilibrium
value from this profile and impact state \(\iota\).  Summing the objectives,
dropping the nonnegative inventory penalties, and using
\(\sum_{i\ne j}u_n^i=U_n-u_n^j\) gives
\[
\begin{aligned}
\sum_{j=1}^nV_n^j(t)
\geq{}&\int_t^T\Bigg[
I(s)^\top U_n(s)
+\sum_{j=1}^n(u_n^j(s))^\top\Lambda u_n^j(s)
\\
&\qquad
+\vartheta_n\sum_{j=1}^n
(u_n^j(s))^\top\Lambda(U_n(s)-u_n^j(s))
\Bigg]ds.
\end{aligned}
\]
The execution-cost terms satisfy
\[
\begin{aligned}
\sum_j(u_n^j)^\top\Lambda
\left(u_n^j+\vartheta_n(U_n-u_n^j)\right)
&=\vartheta_nU_n^\top\Lambda U_n
+(1-\vartheta_n)\sum_j(u_n^j)^\top\Lambda u_n^j
\\
&\geq
\left(\vartheta_n+\frac{1-\vartheta_n}{n}\right)
U_n^\top\Lambda U_n.
\end{aligned}
\]
The inequality uses \(0\leq\vartheta_n\leq1\) from
Assumption~\ref{ass:standing_model} and
\(\sum_j(u_n^j)^\top\Lambda u_n^j\geq n^{-1}U_n^\top\Lambda U_n\).  Set
\(\alpha_n:=\vartheta_n+(1-\vartheta_n)/n\).  Then
\begin{equation}
\label{eq:uniform_solvability_aggregate_execution_bound}
\sum_{j=1}^nV_n^j(t)
\geq
\int_t^T\left(
I(s)^\top U_n(s)+\alpha_nU_n(s)^\top\Lambda U_n(s)
\right)ds.
\end{equation}

Because \(g\equiv0\),
\(I(s)=e^{-R(s-t)}\iota
+\varpi_n\int_t^se^{-R(s-t')}\Gamma U_n(t')dt'\).
Therefore
\begin{equation}\label{eq:uniform_solvability_impact_cost_decomposition}
\begin{aligned}
\int_t^TI(s)^\top U_n(s)ds
={}&\int_t^TU_n(s)^\top e^{-R(s-t)}\iota ds
\\
&+\varpi_n\int_t^TU_n(s)^\top
\left(\int_t^se^{-R(s-t')}\Gamma U_n(t')dt'\right)ds.
\end{aligned}
\end{equation}
The second term on the right-hand side of
\eqref{eq:uniform_solvability_impact_cost_decomposition} is nonnegative because
\(\varpi_n>0\) and \(\mathcal C(v)\geq0\) under
Assumption~\ref{ass:standing_model}. Substituting
\eqref{eq:uniform_solvability_impact_cost_decomposition} into
\eqref{eq:uniform_solvability_aggregate_execution_bound}, dropping the
nonnegative second term, and completing the square pointwise in \(U_n(s)\)
gives
\begin{equation}\label{eq:uniform_solvability_aggregate_value_bound}
\begin{aligned}
\sum_{j=1}^nV_n^j(t)
&\geq
\int_t^T\left[
U_n(s)^\top e^{-R(s-t)}\iota
+\alpha_nU_n(s)^\top\Lambda U_n(s)
\right]ds
\\
&\geq
-\frac1{4\alpha_n}\int_t^T
\iota^\top e^{-R^\top(s-t)}\Lambda^{-1}e^{-R(s-t)}\iota ds
\\
&=-\frac1{4\alpha_n}\iota^\top G_R(t)\iota.
\end{aligned}
\end{equation}

Let \(P_n(t):=[H_n(t)]_{\delta\delta}\), and partition \(H_n(t)\) according
to the deviation coordinate \(\delta\) and the common-state coordinates
\((\bar x,\iota)\):
\[
H_n(t)=
\begin{bmatrix}
P_n(t)&[H_n(t)]_{\delta,(\bar x,\iota)}\\
[H_n(t)]_{(\bar x,\iota),\delta}
&[H_n(t)]_{(\bar x,\iota),(\bar x,\iota)}
\end{bmatrix}.
\]
For the inventory profile above,
\(z^i=(\delta,\bar x,\iota)^\top\) and
\(z^j=(-\delta/(n-1),\bar x,\iota)^\top\) for \(j\ne i\).  Since the
deviations sum to zero, the cross terms cancel:
\begin{equation}
\label{eq:uniform_solvability_aggregate_block_identity}
\begin{aligned}
\sum_{j=1}^nV_n^j(t)
={}&\frac{n}{2(n-1)}\delta^\top P_n(t)\delta
+\frac n2
\begin{bmatrix}\bar x\\\iota\end{bmatrix}^{\!\top}
[H_n(t)]_{(\bar x,\iota),(\bar x,\iota)}
\begin{bmatrix}\bar x\\\iota\end{bmatrix}.
\end{aligned}
\end{equation}
Setting \(\bar x=\iota=0\) in
\eqref{eq:uniform_solvability_aggregate_value_bound} and using
\eqref{eq:uniform_solvability_aggregate_block_identity} gives, for every
\(\delta\in\mathbb R^d\),
\[
0\leq\sum_{j=1}^nV_n^j(t)
=\frac{n}{2(n-1)}\delta^\top P_n(t)\delta.
\]
Hence
\begin{equation}
\label{eq:uniform_solvability_deviation_block_psd}
P_n(t)\succeq0.
\end{equation}
Setting \(\delta=0\) in
\eqref{eq:uniform_solvability_aggregate_value_bound} and using
\eqref{eq:uniform_solvability_aggregate_block_identity} gives, for every
\((\bar x,\iota)\in\mathbb R^{2d}\),
\[
\frac n2
\begin{bmatrix}\bar x\\\iota\end{bmatrix}^{\!\top}
[H_n(t)]_{(\bar x,\iota),(\bar x,\iota)}
\begin{bmatrix}\bar x\\\iota\end{bmatrix}
\geq-\frac1{4\alpha_n}\iota^\top G_R(t)\iota.
\]
This gives the first matrix inequality below, while
\(n\alpha_n=1+(n-1)\vartheta_n\geq1\) and \(G_R(t)\succeq0\) give the second:
\begin{equation}\label{eq:uniform_solvability_common_block_lower_bound}
\bigl[H_n(t)\bigr]_{(\bar x,\iota),(\bar x,\iota)}
\succeq
-\frac1{2n\alpha_n}
\begin{bmatrix}0&0\\0&G_R(t)\end{bmatrix}
\succeq
-\frac12\begin{bmatrix}0&0\\0&G_R(t)\end{bmatrix}.
\end{equation}

Taking traces in \eqref{eq:uniform_solvability_deviation_block_psd} and
\eqref{eq:uniform_solvability_common_block_lower_bound} gives
\begin{equation}\label{eq:uniform_solvability_trace_lower_bound}
\operatorname{tr}(H_n(t))
=\operatorname{tr}(P_n(t))
+\operatorname{tr}\!\left(
\bigl[H_n(t)\bigr]_{(\bar x,\iota),(\bar x,\iota)}
\right)
\geq-\frac12\operatorname{tr}(G_R(t)).
\end{equation}
Set \(W_n(t):=\overline H(t)-H_n(t)\). By
\eqref{eq:uniform_solvability_upper_bound_proof}, \(W_n(t)\succeq0\). Since
\(\operatorname{tr}(\overline H(t))=2\operatorname{tr}(G_Q(t))\),
\eqref{eq:uniform_solvability_trace_lower_bound} gives
\begin{equation}\label{eq:uniform_solvability_w_trace_upper_bound}
\operatorname{tr}(W_n(t))
=\operatorname{tr}(\overline H(t))-\operatorname{tr}(H_n(t))
\leq
2\operatorname{tr}(G_Q(t))
+\frac12\operatorname{tr}(G_R(t)).
\end{equation}
Positive semidefiniteness of \(W_n(t)\) and
\eqref{eq:uniform_solvability_w_trace_upper_bound} imply
\[
W_n(t)
\preceq\operatorname{tr}(W_n(t))\operatorname{Id}_{3d}
\preceq
\left(
2\operatorname{tr}(G_Q(t))
+\frac12\operatorname{tr}(G_R(t))
\right)\operatorname{Id}_{3d}.
\]
Since \(\overline H(t)\succeq0\),
\[
\begin{aligned}
H_n(t)
&=\overline H(t)-W_n(t)
\succeq-W_n(t)
\\
&\succeq
-\left(
2\operatorname{tr}(G_Q(t))
+\frac12\operatorname{tr}(G_R(t))
\right)\operatorname{Id}_{3d}
=-\underline H(t).
\end{aligned}
\]
Together with \eqref{eq:uniform_solvability_upper_bound_proof}, this proves
\eqref{eq:global_riccati_two_sided_bound} on the maximal interval.

The functions \(G_Q\) and \(G_R\) are continuous on \([0,T]\), so the bound
is finite there and uniform over \(n\geq2\) satisfying
\(\chi_u\varpi_n\leq1\).  A positive \(\underline t_n\)
would contradict the continuation criterion.  Hence \(\underline t_n=0\),
and \(H_n\) exists uniquely on \([0,T]\).

\end{proof}

\section{Proofs of Section~\ref{sec:mean_field_limit}}
\label{sec:proofs_mean_field_limit}

\subsection{Proof of Theorem~\ref{thm:mf_solution}}
\label{subsec:proof_mf_solution}

\begin{proof}
Fix a candidate mean flow and its impact components from
\eqref{eq:mf_component_impact_dynamics}. We first solve the type-\(\pi\)
control problem \eqref{eq:mf_type_objective} and then verify the consistency
condition \eqref{eq:mf_average_consistency}.

In a post-announcement scenario \(k\ge1\), the player's HJB equation on
\([T_{\mathrm{ann}},T]\) is
\[
\begin{aligned}
0
={}&\frac{\partial}{\partial t}V^k(t,x)
+\frac12 x^\top Q x
\\
&+\inf_{u\in\mathbb R^d}
\Bigl\{
u^\top\Lambda u
+\bigl(I^k(t)+\Lambda(\varpi\chi_u\bar u^k(t)+\chi_g g^k(t))
+V_x^k(t,x)\bigr)^\top u
\Bigr\},
\end{aligned}
\]
with terminal condition \(V^k(T,x)=\frac12x^\top Q_T x\). Insert the
post-announcement ansatz in \eqref{eq:mf_quadratic_value_ansatz}.
The first-order condition is
\[
2\Lambda u+I^k(t)+\Lambda(\varpi\chi_u\bar u^k(t)+\chi_g g^k(t))
+P(t)x+\ell^k(t)=0.
\]
Solving this for \(u\) gives the optimal feedback
\begin{equation}\label{eq:mf_post_feedback_derived}
\phi^{k}(t,x)
=
-\frac12\Lambda^{-1}
\left(P(t)x+\ell^k(t)+I^k(t)+
\Lambda(\varpi\chi_u\bar u^k(t)+\chi_g g^k(t))\right).
\end{equation}
Substituting the minimizer back into the HJB and collecting the quadratic,
linear, and constant terms gives
\begin{subequations}\label{eq:mf_post_value_system_derived}
\begin{alignat}{2}
\dot P(t)
&=
- Q+\frac12P(t)\Lambda^{-1}P(t),
\qquad &P(T)&=Q_T,
\label{eq:mf_post_riccati_derived}\\
\dot \ell^k(t)
&=
\frac12P(t)\Lambda^{-1}
\left(I^k(t)+\Lambda(\varpi\chi_u\bar u^k(t)+\chi_g g^k(t))+\ell^k(t)\right),
\qquad &\ell^k(T)&=0,
\label{eq:mf_post_ell_derived}\\
\dot c^k(t)
&=
\frac14
\left(I^k(t)+\Lambda(\varpi\chi_u\bar u^k(t)+\chi_g g^k(t))+\ell^k(t)\right)^\top
\qquad &c^k(T)&=0
\label{eq:mf_post_alpha_derived}\\
&\qquad
{}\Lambda^{-1}
\left(I^k(t)+\Lambda(\varpi\chi_u\bar u^k(t)+\chi_g g^k(t))+\ell^k(t)\right).
&&\notag
\end{alignat}
\end{subequations}

Before the announcement,
\(V^{\pi,0}(T_{\mathrm{ann}},x)=\sum_{k=1}^m\pi^k V^k(T_{\mathrm{ann}},x)\), hence
\[
V^{\pi,0}(T_{\mathrm{ann}},x)
=
\frac12x^\top P(T_{\mathrm{ann}})x
+\left(\sum_{k=1}^m\pi^k\ell^k(T_{\mathrm{ann}})\right)^\top x
+\sum_{k=1}^m\pi^kc^k(T_{\mathrm{ann}}),
\]
so the pre-announcement boundary conditions are
\begin{equation}\label{eq:mf_pre_boundary_derived}
\ell^{\pi,0}(T_{\mathrm{ann}})
=
\sum_{k=1}^m\pi^k\ell^k(T_{\mathrm{ann}}),
\qquad
c^{\pi,0}(T_{\mathrm{ann}})
=
\sum_{k=1}^m\pi^kc^k(T_{\mathrm{ann}}).
\end{equation}
The same HJB calculation applies before the announcement. Its quadratic
equation is again \eqref{eq:mf_post_riccati_derived}, so the ODE for \(P\) is independent of
\(\pi\) and common to the pre- and post-announcement problems. Thus it can be
solved once on \([0,T]\). The type-dependent affine and scalar equations on the
pre-announcement interval are
\begin{subequations}\label{eq:mf_pre_value_system_derived}
\begin{align}
\dot\ell^{\pi,0}(t)
&=
\frac12P(t)\Lambda^{-1}
\left(I^0(t)+\Lambda\varpi\chi_u\bar u^0(t)+\ell^{\pi,0}(t)\right),
\label{eq:mf_pre_ell_derived}\\
\dot c^{\pi,0}(t)
&=
\frac14
\left(I^0(t)+\Lambda\varpi\chi_u\bar u^0(t)+\ell^{\pi,0}(t)\right)^\top
\Lambda^{-1}
\left(I^0(t)+\Lambda\varpi\chi_u\bar u^0(t)+\ell^{\pi,0}(t)\right)
\label{eq:mf_pre_alpha_derived}
\end{align}
\end{subequations}
with the boundary conditions in \eqref{eq:mf_pre_boundary_derived}. The
first-order condition gives the pre-announcement feedback
\begin{equation}\label{eq:mf_pre_feedback_derived}
\phi^{\pi,0}(t,x)
=
-\frac12\Lambda^{-1}
\left(P(t)x+\ell^{\pi,0}(t)+I^0(t)+\Lambda\varpi\chi_u\bar u^0(t)\right).
\end{equation}

The affine feedbacks induce unique trajectories with square-integrable rates,
and standard HJB verification establishes their optimality from every starting
time and inventory, including in scenarios with \(\pi^k=0\).

Under the optimal feedbacks, we impose the mean-field consistency condition
\eqref{eq:mf_average_consistency}. Averaging
\eqref{eq:mf_pre_feedback_derived} and
\eqref{eq:mf_post_feedback_derived} over types and solving for the mean flows
gives
\begin{subequations}\label{eq:mf_mu_closed_form_derived}
\begin{align}
\bar u^0
&=-\frac{1}{2+\varpi\chi_u}\Lambda^{-1}
\left(P\bar X^0+I^0+\ell^0\right),
\label{eq:mf_pre_mu_closed_form_derived}\\
\bar u^k
&=-\frac{1}{2+\varpi\chi_u}\Lambda^{-1}
\left(P\bar X^k+I^k+\ell^k+\Lambda\chi_g g^k\right),
\qquad k=1,\ldots,m.
\label{eq:mf_post_mu_closed_form_derived}
\end{align}
\end{subequations}
Averaging the inventory dynamics gives
\begin{subequations}\label{eq:mf_mean_inventory_dynamics_derived}
\begin{align}
\dot{\bar X}^0(t)&=\bar u^0(t),
&\bar X^0(0)&=0,
\label{eq:mf_pre_mean_inventory_dynamics_derived}\\
\dot{\bar X}^k(t)&=\bar u^k(t),
&\bar X^k(T_{\mathrm{ann}})&=\bar X^0(T_{\mathrm{ann}}),
\quad k=1,\ldots,m.
\label{eq:mf_post_mean_inventory_dynamics_derived}
\end{align}
\end{subequations}
Averaging \eqref{eq:mf_pre_ell_derived} and
the affine boundary condition in \eqref{eq:mf_pre_boundary_derived}, using
\(\bar\pi=\int_{\Delta_m}\pi\,\Pi(d\pi)\), gives
\begin{equation}\label{eq:mf_pre_common_affine_system_derived}
\dot \ell^0(t)
=
\frac12P(t)\Lambda^{-1}
\left(I^0(t)+\Lambda\varpi\chi_u\bar u^0(t)+\ell^0(t)\right),
\quad
\ell^0(T_{\mathrm{ann}})
=
\sum_{k=1}^m\bar\pi^k\ell^k(T_{\mathrm{ann}}).
\end{equation}
Substituting
\eqref{eq:mf_mu_closed_form_derived} into
\eqref{eq:mf_mean_inventory_dynamics_derived},
\eqref{eq:mf_component_impact_dynamics},
\eqref{eq:mf_pre_common_affine_system_derived}, and
\eqref{eq:mf_post_ell_derived}, and collecting terms,
gives the linear ODE in \eqref{eq:mf_y_system} with \(\mathcal A\) and \(b^k\) defined in
\eqref{eq:mf_A_matrix} and \eqref{eq:mf_b_vector}.
The boundary conditions in
\eqref{eq:mf_mean_inventory_dynamics_derived},
\eqref{eq:mf_component_impact_dynamics},
\eqref{eq:mf_pre_common_affine_system_derived}, and
\eqref{eq:mf_post_ell_derived} establish the boundary conditions of the ODE in
\eqref{eq:mf_consistency_system}. Conversely, given a solution of \eqref{eq:mf_consistency_system}, the preceding
construction yields admissible trading schedules \(u^\pi\) that are optimal responses
to \(\bar u\) and satisfy the consistency condition
\eqref{eq:mf_average_consistency}.
The equilibrium feedbacks induce rates that are square-integrable uniformly
in \(\pi\), and their type averages \(\bar u^0,\ldots,\bar u^m\) are therefore
square-integrable. Thus \((u,\bar u)\) satisfies
Definition~\ref{def:mf_equilibrium}.
\end{proof}

\subsection{Proof of Lemma~\ref{lem:mf_state_transition_representation}}
\label{subsec:proof_mf_state_transition_representation}

\begin{proof}
For the linear ODE \(\dot Y(t)=\mathcal A(t)Y(t)+b(t)\), the solution with
initial value \(Y(s)\) is
\begin{equation}\label{eq:mf_linear_solution_formula}
Y(t)
=
\mathcal E(t,s)Y(s)
+
\int_s^t\mathcal E(t,s')b(s')\,ds'.
\end{equation}
On the pre-announcement interval, \(b^0=0\) and
\(\bar X^0(0)=I^0(0)=0\). Hence the whole pre-announcement path is determined by
the single unknown \(\ell^0(0)\). Applying
\eqref{eq:mf_linear_solution_formula} with \(s=0\) gives
\begin{equation}\label{eq:mf_pre_announcement_values}
\begin{aligned}
\bar X^0(T_{\mathrm{ann}})
&=\mathcal E_{X\ell}(T_{\mathrm{ann}},0)\ell^0(0),\\
I^0(T_{\mathrm{ann}})
&=\mathcal E_{I\ell}(T_{\mathrm{ann}},0)\ell^0(0),\\
\ell^0(T_{\mathrm{ann}})
&=\mathcal E_{\ell\ell}(T_{\mathrm{ann}},0)\ell^0(0).
\end{aligned}
\end{equation}
For each post-announcement scenario \(k\), applying
\eqref{eq:mf_linear_solution_formula} from \(T_{\mathrm{ann}}\) and using
\(\ell^k(T)=0\) gives
\begin{equation}\label{eq:simplification_proof_zero_cond}
\begin{aligned}
0=\ell^k(T)
={}&
\mathcal E_{\ell X}(T,T_{\mathrm{ann}})\bar X^0(T_{\mathrm{ann}})
+\mathcal E_{\ell I}(T,T_{\mathrm{ann}})I^0(T_{\mathrm{ann}})
\\
&+\mathcal E_{\ell\ell}(T,T_{\mathrm{ann}})\ell^k(T_{\mathrm{ann}})
+\int_{T_{\mathrm{ann}}}^T\mathcal E_{\ell:}(T,s)b^k(s)\,ds.
\end{aligned}
\end{equation}
Averaging these identities with weights \(\bar\pi^k\) and using the announcement matching condition
\(\ell^0(T_{\mathrm{ann}})=\sum_k\bar\pi^k\ell^k(T_{\mathrm{ann}})\) gives
\[
\begin{aligned}
0
={}&
\mathcal E_{\ell X}(T,T_{\mathrm{ann}})\bar X^0(T_{\mathrm{ann}})
+\mathcal E_{\ell I}(T,T_{\mathrm{ann}})I^0(T_{\mathrm{ann}})
\\
&+\mathcal E_{\ell\ell}(T,T_{\mathrm{ann}})\ell^0(T_{\mathrm{ann}})
+\sum_{k=1}^m\bar\pi^k\int_{T_{\mathrm{ann}}}^T\mathcal E_{\ell:}(T,s)b^k(s)\,ds.
\end{aligned}
\]
Substituting \eqref{eq:mf_pre_announcement_values} and using
\(\mathcal E(T,T_{\mathrm{ann}})\mathcal E(T_{\mathrm{ann}},0)
=\mathcal E(T,0)\), we obtain
\begin{equation}\label{eq:mf_matching_equation}
\mathcal E_{\ell\ell}(T,0)\ell^0(0)
=
-\sum_{k=1}^m\bar\pi^k\int_{T_{\mathrm{ann}}}^T\mathcal E_{\ell:}(T,s)b^k(s)\,ds.
\end{equation}
Since \(\mathcal E_{\ell\ell}(T,0)\) is invertible,
\eqref{eq:mf_matching_equation} yields \eqref{eq:mf_pre_matching} and
determines the pre-announcement path. If \(\mathcal E_{\ell\ell}(T,T_{\mathrm{ann}})\) is invertible, each
scenario-specific terminal condition \eqref{eq:simplification_proof_zero_cond} determines
\[
\begin{aligned}
\ell^k(T_{\mathrm{ann}})
={}&
-(\mathcal E_{\ell\ell}(T,T_{\mathrm{ann}}))^{-1}
\Bigl(\mathcal E_{\ell X}(T,T_{\mathrm{ann}})\bar X^0(T_{\mathrm{ann}})+
\mathcal E_{\ell I}(T,T_{\mathrm{ann}})I^0(T_{\mathrm{ann}})
\\
&\qquad+\int_{T_{\mathrm{ann}}}^T\mathcal E_{\ell:}(T,s)b^k(s)\,ds\Bigr).
\end{aligned}
\]
The solution formula \eqref{eq:mf_linear_solution_formula} then gives all
post-announcement paths. These paths satisfy the ODEs and all boundary
conditions. Uniqueness follows because the invertibility of \(\mathcal E_{\ell\ell}(T,0)\)
fixes \(\ell^0(0)\), the invertibility of
\(\mathcal E_{\ell\ell}(T,T_{\mathrm{ann}})\) fixes each \(\ell^k(T_{\mathrm{ann}})\), and the resulting linear
initial-value problems have unique solutions.
\end{proof}

\subsection{Proof of Theorem~\ref{thm:mf_implementability}}
\label{subsec:proof_mf_implementability}

\begin{proof}
We first prove that the post-announcement block
\(\mathcal E_{\ell\ell}(T,T_{\mathrm{ann}})\) is invertible. The matrix
\(\mathcal E(t,s)\) is the fundamental solution of the homogeneous ODE
\begin{equation}\label{eq:mf_homogeneous_transition_ode}
\dot{Y}(t)=\mathcal A(t)Y(t),
\qquad
Y(t)=(\bar X(t),I(t),\ell(t))^\top
\end{equation}
with \(\mathcal{A}(t)\) defined in \eqref{eq:mf_A_matrix}.
Let \(\mathfrak q\in\ker(\mathcal E_{\ell\ell}(T,T_{\mathrm{ann}}))\). Solve
\eqref{eq:mf_homogeneous_transition_ode} on
\([T_{\mathrm{ann}},T]\) with \(\bar X(T_{\mathrm{ann}})=0\),
\(I(T_{\mathrm{ann}})=0\), and \(\ell(T_{\mathrm{ann}})=\mathfrak q\).
By \eqref{eq:mf_linear_solution_formula} and the assumption
\(\mathfrak q\in\ker(\mathcal E_{\ell\ell}(T,T_{\mathrm{ann}}))\), the terminal affine coefficient is
\[
\begin{aligned}
\ell(T)
&={}
\mathcal E_{\ell X}(T,T_{\mathrm{ann}})\bar X(T_{\mathrm{ann}})
+
\mathcal E_{\ell I}(T,T_{\mathrm{ann}})I(T_{\mathrm{ann}})+\mathcal E_{\ell\ell}(T,T_{\mathrm{ann}})\ell(T_{\mathrm{ann}})
\\
&=\mathcal E_{\ell\ell}(T,T_{\mathrm{ann}})\mathfrak q
=0.
\end{aligned}
\]
The first row of \eqref{eq:mf_homogeneous_transition_ode} is
\begin{equation}\label{eq:mf_homogeneous_mean_inventory_ode}
\dot{\bar X}(t)
=
-\frac{1}{2+\varpi\chi_u}\Lambda^{-1}
\left(P(t)\bar X(t)+I(t)+\ell(t)\right).
\end{equation}
Define \(\bar u(t):=-\frac{1}{2+\varpi\chi_u}\Lambda^{-1}
\left(P(t)\bar X(t)+I(t)+\ell(t)\right)\) and
\(p(t):=P(t)\bar X(t)+\ell(t)\). Then
\eqref{eq:mf_homogeneous_mean_inventory_ode} can be rewritten as
\begin{equation}\label{eq:mf_homogeneous_first_row_identity}
p(t)+I(t)+(2+\varpi\chi_u)\Lambda\bar u(t)=0.
\end{equation}
Multiplying \eqref{eq:mf_homogeneous_first_row_identity} by
\(\bar u(t)^\top\) and
integrating over \([T_{\mathrm{ann}},T]\) gives
\begin{equation}\label{eq:mf_homogeneous_energy_identity}
0
=
\int_{T_{\mathrm{ann}}}^T\bar u(t)^\top p(t)\,dt
+
\int_{T_{\mathrm{ann}}}^T\bar u(t)^\top I(t)\,dt
+
(2+\varpi\chi_u)\int_{T_{\mathrm{ann}}}^T
\bar u(t)^\top\Lambda\bar u(t)\,dt.
\end{equation}
We now show that all three integrals in
\eqref{eq:mf_homogeneous_energy_identity} are nonnegative. The last term is
clearly nonnegative because \(2+\varpi\chi_u>0\) and \(\Lambda\succ0\).
For the first integral, differentiating \(p=P\bar X+\ell\) and using \eqref{eq:mf_riccati} together
with the first and third rows of \eqref{eq:mf_homogeneous_transition_ode}
gives \(\dot p=-Q\bar X\). Since \(P(T)=Q_T\) and \(\ell(T)=0\),
\(p(T)=Q_T\bar X(T)\). Integration by parts of the first integral in \eqref{eq:mf_homogeneous_energy_identity} therefore gives
\begin{equation}\label{eq:mf_homogeneous_inventory_cost}
\begin{aligned}
\int_{T_{\mathrm{ann}}}^T\bar u(t)^\top p(t)\,dt
={}&\bar X(T)^\top p(T)
-\bar X(T_{\mathrm{ann}})^\top p(T_{\mathrm{ann}})-\int_{T_{\mathrm{ann}}}^T\bar X(t)^\top\dot p(t)\,dt\\
={}&\bar X(T)^\top Q_T\bar X(T)
+\int_{T_{\mathrm{ann}}}^T\bar X(t)^\top Q\bar X(t)\,dt
\end{aligned}
\end{equation}
which is non-negative since \(Q,Q_T\succeq0\).
For the second integral, use the definition of \(\bar u\) to write the
impact row of \eqref{eq:mf_homogeneous_transition_ode} as
\begin{equation}\label{eq:mf_homogeneous_impact_row}
\dot I(t)=-RI(t)+\varpi\Gamma\bar u(t).
\end{equation}
By \eqref{eq:mf_homogeneous_impact_row} and
\(I(T_{\mathrm{ann}})=0\),
\[
\int_{T_{\mathrm{ann}}}^T\bar u(t)^\top I(t)\,dt
=\varpi\int_{T_{\mathrm{ann}}}^T\bar u(t)^\top
\left(\int_{T_{\mathrm{ann}}}^t e^{-R(t-s)}\Gamma\bar u(s)\,ds\right)dt
\geq0,
\]
where the inequality follows from \(\varpi>0\) and
Assumption~\ref{ass:standing_model}. Hence
\eqref{eq:mf_homogeneous_energy_identity} implies \(\bar u\equiv0\). Since
\(\bar X(T_{\mathrm{ann}})=0\) and
\(\dot{\bar X}(t)=\bar u(t)\), this implies \(\bar X\equiv0\). The impact row
\eqref{eq:mf_homogeneous_impact_row} with \(I(T_{\mathrm{ann}})=0\) then
gives \(I\equiv0\). The identity \(p+I+(2+\varpi\chi_u)\Lambda\bar u=0\) from
\eqref{eq:mf_homogeneous_first_row_identity} gives
\(p\equiv0\). Evaluating \(p=P\bar X+\ell\) at \(T_{\mathrm{ann}}\) therefore
gives \(\mathfrak q=\ell(T_{\mathrm{ann}})=p(T_{\mathrm{ann}})=0\). Hence
\(\ker(\mathcal E_{\ell\ell}(T,T_{\mathrm{ann}}))=\{0\}\) and
\(\mathcal E_{\ell\ell}(T,T_{\mathrm{ann}})\) is invertible.

It remains to prove that \(\mathcal E_{\ell\ell}(T,0)\) is invertible.
Let \(\mathfrak q\in\ker(\mathcal E_{\ell\ell}(T,0))\). Run
\eqref{eq:mf_homogeneous_transition_ode} on \([0,T_{\mathrm{ann}}]\) from
\(\bar X(0)=0\), \(I(0)=0\), and \(\ell(0)=\mathfrak q\).
The state at \(T_{\mathrm{ann}}\) satisfies
\(\bar X(T_{\mathrm{ann}})=\mathcal E_{X\ell}(T_{\mathrm{ann}},0)\mathfrak q\),
\(I(T_{\mathrm{ann}})=\mathcal E_{I\ell}(T_{\mathrm{ann}},0)\mathfrak q\), and
\(\ell(T_{\mathrm{ann}})=\mathcal E_{\ell\ell}(T_{\mathrm{ann}},0)\mathfrak q\).
Substituting the values at \(T_{\mathrm{ann}}\) and using
\(\mathcal E(T,0)=\mathcal E(T,T_{\mathrm{ann}})
\mathcal E(T_{\mathrm{ann}},0)\) together with
\(\mathfrak q\in\ker(\mathcal E_{\ell\ell}(T,0))\), we obtain
\[
\mathcal E_{\ell X}(T,T_{\mathrm{ann}})\bar X(T_{\mathrm{ann}})
+
\mathcal E_{\ell I}(T,T_{\mathrm{ann}})I(T_{\mathrm{ann}})
+
\mathcal E_{\ell\ell}(T,T_{\mathrm{ann}})\ell(T_{\mathrm{ann}})
=\mathcal E_{\ell\ell}(T,0)\mathfrak q
=0.
\]
This implies that the solution to
\eqref{eq:mf_homogeneous_transition_ode} on
\([T_{\mathrm{ann}},T]\), starting from
\((\bar X(T_{\mathrm{ann}}),\allowbreak I(T_{\mathrm{ann}}),\allowbreak\ell(T_{\mathrm{ann}}))\), has terminal
affine coefficient \(\ell(T)=0\). Equivalently, running the homogeneous ODE over
the combined interval \([0,T]\) with \(\bar X(0)=0\), \(I(0)=0\), and
\(\ell(0)=\mathfrak q\)
produces a solution satisfying \(\ell(T)=0\). The same energy argument used for
\(\mathcal E_{\ell\ell}(T,T_{\mathrm{ann}})\), now on \([0,T]\), gives
\(\bar u\equiv0\), hence \(\bar X\equiv0\), \(I\equiv0\), and
\(p\equiv0\).
Since \(p\equiv0\), \(\bar X(0)=0\), and \(\ell(0)=\mathfrak q\), evaluating
\(p=P\bar X+\ell\) at \(t=0\) gives
\(0=p(0)=\mathfrak q\). Thus
\(\ker(\mathcal E_{\ell\ell}(T,0))=\{0\}\), and \(\mathcal E_{\ell\ell}(T,0)\) is invertible.
\end{proof}

\subsection{Proof of Theorem~\ref{thm:finite_to_mf_convergence}}
\label{subsec:proof_finite_to_mf_convergence}

\begin{proof}
The proof has four steps. First, we prove convergence of the finite-player
Riccati matrices to the solution of the Riccati equation obtained from the
large-population limit of the finite-player vector fields and identify the
limiting deviation block with the mean-field Riccati coefficient. Second, we
augment the finite-player common state with an auxiliary affine coordinate and
show that the resulting ODE vector fields and boundary and matching conditions
differ from their mean-field counterparts by \(O(n^{-1})\) and
\(O(\|\bar\pi_n-\bar\pi\|)\), respectively. Third, we use the boundary and
announcement matching conditions to deduce bounds on the differences between
\((\bar X_n,I_n)\) and \((\bar X,I)\) from the vector-field bounds.
The common-state equations form a boundary-value problem, so Gronwall's
inequality does not apply directly. Finally, we bound the differences between
the finite-player and mean-field inventory deviations from their respective
population averages, uniformly over players. Combining the common-state and
deviation bounds gives the player-level inventory and trading-rate estimates.

\emph{Step 1: Finite Riccati Coefficients.}
Recall that \(H_n:[0,T]\to\mathbb S^{3d}\) is the \(3d\times3d\) finite-player
Riccati coefficient from \eqref{eq:fixed_interval_riccati_P}. In contrast, the
mean-field Riccati coefficient \(P:[0,T]\to\mathbb S^d\) from
\eqref{eq:mf_riccati} is \(d\times d\). Thus \(H_n\) and \(P\) cannot be
compared as whole matrices. The goal of this step is to show that \(H_n\)
converges to the solution \(H:[0,T]\to\mathbb S^{3d}\) of the Riccati equation
obtained from the large-population limit of the finite-player vector fields,
and then to identify
\([H]_{\delta\delta}\) with \(P\).

The Riccati equation for \(H_n\) can be written as
\(0=\dot H_n(t)+\mathscr F_n(H_n(t))\), with \(H_n(T)=H_T\),
where the vector field \(\mathscr F_n(H_0)\) is given by
\(\mathscr F_n(H_0)
:=
Q_z
+M_n(H_0)^\top H_0
+H_0 M_n(H_0)
-2D_n^\top
\Theta_n(H_0)^\top\Lambda^{-1}\Theta_n(H_0)D_n\).
We now take the limit of \(\mathscr F_n\) term by term. All estimates below are uniform
for \(H_0\) in bounded subsets of \(\mathbb S^{3d}\). First, \(D_n\) from
\eqref{eq:fixed_interval_D_def} satisfies
\(D_n=D+O\!\left(\frac1n\right)\), where
\(D:=
\operatorname{diag}
\left(\frac12\operatorname{Id}_d,\frac{1}{2+\varpi\chi_u}\operatorname{Id}_d,
\frac{1}{2+\varpi\chi_u}\operatorname{Id}_d\right)\).
Next, from \eqref{eq:fixed_interval_Theta_def},
\(\Theta_n(H_0)=\Theta(H_0)+O\!\left(\frac1n\right)\), where
\(\Theta(H_0):=E_\iota+B_\delta^\top H_0\) and
\(B_\delta:=[\operatorname{Id}_d\;0\;0]^\top\).
Finally, from \eqref{eq:fixed_interval_M_def},
\(M_n(H_0)=M(H_0)+O\!\left(\frac1n\right)\), where
\[
\begin{aligned}
M(H_0)&:=
A_z+\frac{1}{2+\varpi\chi_u} B_{\mathrm{com}}\Lambda^{-1}\Theta(H_0)E_{\mathrm{com}},
\\
B_{\mathrm{com}}&:=
\begin{bmatrix}
\operatorname{Id}_d\\-\operatorname{Id}_d\\-\varpi\Gamma
\end{bmatrix},
\\
E_{\mathrm{com}}&:=
\operatorname{diag}(0,\operatorname{Id}_d,\operatorname{Id}_d).
\end{aligned}
\]
Indeed, \(B_n^{-}=-\frac1n B_{\mathrm{com}}\) and
\(\frac1n F_n=-\frac{1}{2+\varpi\chi_u}E_{\mathrm{com}}
+O\!\left(\frac1n\right)\),
so the order-\(n\) \(\bar x\)- and \(\iota\)-blocks of \(F_n\) offset the
\(1/n\)-scale of \(B_n^{-}\).
Combining these three estimates gives, for every compact
\(\mathfrak C\subset\mathbb S^{3d}\), a constant
\(C_{\mathfrak C}<\infty\), independent of \(n\), such that
\begin{equation}\label{eq:finite_riccati_vector_field_convergence}
\sup_{H_0\in\mathfrak C}
\|\mathscr F_n(H_0)-\mathscr F(H_0)\|
\le
\frac{C_{\mathfrak C}}{n},
\end{equation}
where \(\mathscr F(H_0)
:=
Q_z
+M(H_0)^\top H_0
+H_0 M(H_0)
-2D^\top
\Theta(H_0)^\top\Lambda^{-1}
\Theta(H_0)D\) and \(\|\cdot\|\) is any fixed matrix norm.

Since \(\mathscr F\) is locally Lipschitz, let \(H\) be the unique
maximal backward solution of
\begin{equation}\label{eq:limiting_three_block_riccati}
0=\dot{H}(t)+\mathscr F(H(t)),
\qquad
H(T)=H_T.
\end{equation}
Write its maximal interval as \((\underline t,T]\).
Since \(n\ge n_0=\max\{2,\lceil\varpi\chi_u\rceil\}\) implies
\(\chi_u\varpi_n=\varpi\chi_u/n\le1\),
Theorem~\ref{thm:global_riccati_solvability} gives a constant \(C_H\), independent
of \(n\), such that
\begin{equation}\label{eq:finite_riccati_uniform_bound}
\sup_{n\ge n_0}\sup_{t\in[0,T]}\lVert H_n(t)\rVert\le C_H.
\end{equation}
For \(s\in(\underline t,T)\), define
\(\varrho_s:=1+\max\{C_H,\sup_{t\in[s,T]}\lVert H(t)\rVert\}\) and
\(\mathbb B_s:=\{H_0\in\mathbb S^{3d}:\lVert H_0\rVert\le\varrho_s\}\).
The compact set \(\mathbb B_s\) contains \(H(t)\) and \(H_n(t)\)
for every \(n\ge n_0\) and \(t\in[s,T]\). Choose \(L_s<\infty\) such that
\(\lVert\mathscr F(H_0)-\mathscr F(H_1)\rVert
\le L_s\lVert H_0-H_1\rVert\) for \(H_0,H_1\in\mathbb B_s\),
and, by \eqref{eq:finite_riccati_vector_field_convergence}, choose
\(c_s<\infty\) such that
\(\sup_{H_0\in\mathbb B_s}\lVert\mathscr F_n(H_0)-\mathscr F(H_0)\rVert
\le\frac{c_s}{n}\) for \(n\ge n_0\).
Since
\(H_n(T)=H(T)\), for \(t\in[s,T]\),

{\small
\[
\begin{aligned}
\lVert H_n(t)-H(t)\rVert
&\le
\int_t^T
\left\lVert
\mathscr F_n(H_n(t'))-\mathscr F(H_n(t'))
\right\rVert\,dt'
+
\int_t^T
\left\lVert
\mathscr F(H_n(t'))-\mathscr F(H(t'))
\right\rVert\,dt'
\\
&\le
\frac{c_s}{n}(T-t)
+L_s\int_t^T\lVert H_n(t')-H(t')\rVert\,dt'.
\end{aligned}
\]
}
Gronwall therefore gives a constant \(C_s<\infty\), independent of
\(n\), such that
\begin{equation}\label{eq:finite_riccati_local_convergence}
\sup_{t\in[s,T]}\lVert H_n(t)-H(t)\rVert
\le\frac{C_s}{n}.
\end{equation}
Combining \eqref{eq:finite_riccati_uniform_bound} and
\eqref{eq:finite_riccati_local_convergence}, and taking \(n\) sufficiently
large, shows that
\begin{equation}\label{eq:limiting_riccati_uniform_bound}
\sup_{t\in[s,T]}\lVert H(t)\rVert\le C_H+1.
\end{equation}
The bound in \eqref{eq:limiting_riccati_uniform_bound} does not depend on
\(s\), so \(H\) cannot blow up as \(t\downarrow\underline t\).
The continuation criterion therefore shows that \(H\)
can be extended to the full interval \([0,T]\) and, together with all \(H_n\), remains in one compact
set. On this compact set, \(\mathscr F\) is Lipschitz, and the estimate in
\eqref{eq:finite_riccati_vector_field_convergence} applies with a constant
independent of \(s\). The preceding
Gronwall argument with \(s=0\) yields
\begin{equation}\label{eq:finite_riccati_limit_to_three_block}
\sup_{t\in[0,T]}\|H_n(t)-H(t)\|
\le
\frac{C}{n}.
\end{equation}

Reading off the \((\delta,\delta)\)-block of
\eqref{eq:limiting_three_block_riccati} gives
\(\frac{d}{dt}[H]_{\delta\delta}
=-Q+\frac12[H]_{\delta\delta}\Lambda^{-1}[H]_{\delta\delta}\),
with \([H(T)]_{\delta\delta}=Q_T\).
This is the mean-field Riccati equation \eqref{eq:mf_riccati}, so uniqueness
gives \([H]_{\delta\delta}=P\).

Set \(P_n(t):=[H_n(t)]_{\delta\delta}\). Taking the
\((\delta,\delta)\)-block in
\eqref{eq:finite_riccati_limit_to_three_block} gives
\begin{equation}\label{eq:finite_riccati_deviation_block_convergence}
\sup_{t\in[0,T]}
\lVert P_n(t)-P(t)\rVert
\le\frac{C}{n}.
\end{equation}
For the generic coordinates \((\delta,\bar x,\iota)\), write the block columns
of the feedback matrix as
\(\Theta_n(H_n(t))=
\begin{bmatrix}
[\Theta_n(H_n(t))]_{\delta}&
[\Theta_n(H_n(t))]_{\bar x}&
[\Theta_n(H_n(t))]_{\iota}
\end{bmatrix}\).
We also have the following bound, which will be needed in Step~4:
\begin{equation}\label{eq:finite_theta_delta_convergence}
\sup_{t\in[0,T]}
\left\lVert
\bigl[\Theta_n(H_n(t))\bigr]_\delta-P(t)
\right\rVert
\le
\frac{C}{n}.
\end{equation}
Indeed, \(\bigl[\Theta_n(H_n)\bigr]_\delta
=P_n+\frac1n\left([H_n]_{\bar x\delta}-P_n\right)
+\frac{\varpi}{n}\Gamma^\top[H_n]_{\iota\delta}\).
The correction terms are uniformly \(O(n^{-1})\) by
\eqref{eq:finite_riccati_uniform_bound}, and
\eqref{eq:finite_riccati_deviation_block_convergence} gives
\(P_n-P=O(n^{-1})\) uniformly in time.

\emph{Step 2: Approximation of the Mean-Field Common-State Dynamics.}
In this step, we augment the finite-player common state
\((\bar X_n^k,I_n^k)\) with an auxiliary affine coordinate \(\ell_n^k\)
and set \(Y_n^k:=(\bar X_n^k,I_n^k,\ell_n^k)^\top\) to prepare the proof
of convergence to the mean-field common-state solution.
We show that its vector field differs by \(O(n^{-1})\)
from the mean-field vector field in \eqref{eq:mf_y_system} on bounded sets.

Under the scaling \(\varpi_n=\varpi/n\), the finite common-state
system \eqref{eq:common_state_dynamics} can be written on
\([0,T_{\mathrm{ann}}]\) for \(k=0\) and on
\([T_{\mathrm{ann}},T]\) for \(k=1,\ldots,m\), as
\begin{subequations}\label{eq:finite_common_original_system}
\begin{align}
\dot{\bar X}_n^k
&=
-\frac1{d_n}\Lambda^{-1}
\left(
[\Theta_n(H_n)]_{\bar x}\bar X_n^k
+[\Theta_n(H_n)]_{\iota}I_n^k
+B_n^\top r_n^k
+\chi_g\Lambda g^k
\right),
\label{eq:finite_common_original_X}\\
\dot I_n^k
&=
-RI_n^k
-\frac{\varpi}{d_n}\Gamma\Lambda^{-1}
\left(
[\Theta_n(H_n)]_{\bar x}\bar X_n^k
+[\Theta_n(H_n)]_{\iota}I_n^k
+B_n^\top r_n^k
+\chi_g\Lambda g^k
\right)
+\Gamma g^k.
\label{eq:finite_common_original_I}
\end{align}
\end{subequations}
To put the finite equations in the same form as their mean-field counterparts in
\eqref{eq:mf_y_system}, introduce the
\(d\)-dimensional auxiliary affine coefficient
\begin{equation}\label{eq:finite_common_auxiliary_ell_definition}
\ell_n^k
:=
\left(
[\Theta_n(H_n)]_{\bar x}-P
\right)\bar X_n^k
+
\left(
[\Theta_n(H_n)]_{\iota}-\operatorname{Id}_d
\right)I_n^k
+B_n^\top r_n^k,
\qquad k=0,\ldots,m.
\end{equation}
Substituting
\eqref{eq:finite_common_auxiliary_ell_definition} into
\eqref{eq:finite_common_original_system} gives
\begin{subequations}\label{eq:finite_common_auxiliary_ell_state_system}
\begin{align}
\dot{\bar X}_n^k
&=
-\frac1{d_n}\Lambda^{-1}
\left(
P\bar X_n^k+I_n^k+\ell_n^k+\chi_g\Lambda g^k
\right),
\label{eq:finite_common_auxiliary_ell_X}\\
\dot I_n^k
&=
-RI_n^k
-\frac{\varpi}{d_n}\Gamma\Lambda^{-1}
\left(
P\bar X_n^k+I_n^k+\ell_n^k+\chi_g\Lambda g^k
\right)
+\Gamma g^k.
\label{eq:finite_common_auxiliary_ell_I}
\end{align}
\end{subequations}
These equations coincide with the \((\bar X,I)\)-components of
\eqref{eq:mf_y_system}, except that \((2+\varpi\chi_u)\) is replaced by \(d_n\).
To close the finite system, it remains to derive the equation for
\(\ell_n^k\).

Differentiating \eqref{eq:finite_common_auxiliary_ell_definition} and using
\eqref{eq:finite_common_auxiliary_ell_state_system},
\eqref{eq:mf_riccati}, \eqref{eq:fixed_interval_riccati_P},
\eqref{eq:post_ann_common_r}, and
\eqref{eq:pre_ann_common_r_ode} gives
\begin{equation}\label{eq:finite_common_auxiliary_ell_equation}
\begin{aligned}
\dot\ell_n^k(t)
={}&
-\frac{\varpi\chi_u}{2(2+\varpi\chi_u)}
P(t)\Lambda^{-1}P(t)\bar X_n^k(t)
\\
&\quad
+\frac{1}{2+\varpi\chi_u} P(t)\Lambda^{-1}
\left(
I_n^k(t)+\ell_n^k(t)+\chi_g\Lambda g^k(t)
\right)
\\
&\quad
+\upsilon_n^k
\left(
t,\bar X_n^k(t),I_n^k(t),\ell_n^k(t)
\right).
\end{aligned}
\end{equation}
For \(t\in[0,T_{\mathrm{ann}}]\) when \(k=0\) and
\(t\in[T_{\mathrm{ann}},T]\) when \(k=1,\ldots,m\), and for
\(\bar x,\iota,\lambda\in\mathbb R^d\), the residual \(\upsilon_n^k\) is given by
\begin{equation}\label{eq:finite_common_auxiliary_ell_residual_identity}
\begin{aligned}
\upsilon_n^k(t,\bar x,\iota,\lambda)
:={}&
-\left(
M_n(H_n(t))B_n
\right)^\top
\left[
H_n(t)
\begin{pmatrix}
0\\
\bar x\\
\iota
\end{pmatrix}
+r_n^k(t)
\right]
\\
&\quad
+\frac1{d_n}
\left[
2\left(
\Theta_n(H_n(t))D_nB_n
\right)^\top
-B_n^\top H_n(t)B_n
\right]
\\
&\qquad {}\times
\Lambda^{-1}
\left(
P(t)\bar x+\iota+\lambda+\chi_g\Lambda g^k(t)
\right)
\\
&\quad
+\left(
\frac1{d_n}-\frac{1}{2+\varpi\chi_u}
\right)
P(t)\Lambda^{-1}
\left(
P(t)\bar x+\iota+\lambda+\chi_g\Lambda g^k(t)
\right).
\end{aligned}
\end{equation}

With this definition,
\eqref{eq:finite_common_auxiliary_ell_equation} mirrors the equation for
\(\ell^k\) in the mean-field system \eqref{eq:mf_y_system}, apart from the
remainder \(\upsilon_n^k\), which we now show is \(O(n^{-1})\), uniformly on
bounded sets of \((\bar x,\iota,\lambda)\). First,
\begin{equation}\label{eq:finite_common_affine_uniform_bound}
\max\left\{
\sup_{t\in[0,T_{\mathrm{ann}}]}\lVert r_n^0(t)\rVert,
\max_{k=1,\ldots,m}\sup_{t\in[T_{\mathrm{ann}},T]}
\lVert r_n^k(t)\rVert
\right\}
\leq
C.
\end{equation}
Indeed, by \eqref{eq:finite_riccati_uniform_bound}, the coefficients and
forcing in the post-announcement ODE \eqref{eq:post_ann_common_r} are
uniformly bounded, and its terminal data are zero.
The pre-announcement ODE \eqref{eq:pre_ann_common_r_ode} also has uniformly
bounded coefficients, and its terminal condition is a convex combination
of the bounded values \(r_n^k(T_{\mathrm{ann}})\).
To bound the matrix coefficients in
\eqref{eq:finite_common_auxiliary_ell_residual_identity}, we use the estimates
from Step~1 and \(A_zB_\delta=E_{\mathrm{com}}B_\delta=0\) to obtain,
uniformly in \(t\in[0,T]\),
\begin{equation}\label{eq:finite_common_m_product_bound}
\begin{aligned}
M_n(H_n(t))B_n
&=M(H(t))B_\delta+O(n^{-1})
\\
&=
A_zB_\delta
+\frac{1}{2+\varpi\chi_u}
B_{\mathrm{com}}\Lambda^{-1}
\Theta(H(t))E_{\mathrm{com}}B_\delta
+O(n^{-1})
\\
&=O(n^{-1}).
\end{aligned}
\end{equation}
Likewise, the estimates from Step~1, together with
\(DB_\delta=\frac12B_\delta\), \(E_\iota B_\delta=0\), and
\(B_\delta^\top H(t)B_\delta=P(t)\), give
\begin{equation}\label{eq:finite_common_theta_product_bound}
\begin{aligned}
2\left(
\Theta_n(H_n(t))D_nB_n
\right)^\top
-B_n^\top H_n(t)B_n
&=
2\left(
\Theta(H(t))DB_\delta
\right)^\top
-B_\delta^\top H(t)B_\delta
+O(n^{-1})\\
&=
\left(
B_\delta^\top H(t)B_\delta
\right)^\top
-B_\delta^\top H(t)B_\delta
+O(n^{-1})
\\
&=P(t)-P(t)
+O(n^{-1})
=O(n^{-1}).
\end{aligned}
\end{equation}
Estimates \eqref{eq:finite_common_affine_uniform_bound}--\eqref{eq:finite_common_theta_product_bound},
together with \(d_n^{-1}-(2+\varpi\chi_u)^{-1}=O(n^{-1})\) and the uniform
bound \eqref{eq:finite_riccati_uniform_bound}, show that
\(\upsilon_n^k=O(n^{-1})\), uniformly in time and on compact subsets of
\(\mathbb R^{3d}\). Since \(\upsilon_n^k(t,\cdot)\) is affine, its linear and
constant coefficients are \(O(n^{-1})\) as well.
We have thus shown that the augmented finite-player common state \(Y_n^k\)
satisfies the linear ODE \(\dot Y_n^k=\mathcal A_n(t)Y_n^k+b_n^k(t)\) on
\([0,T_{\mathrm{ann}}]\) for \(k=0\) and \([T_{\mathrm{ann}},T]\) for
\(k=1,\ldots,m\), with coefficients satisfying
\begin{equation}\label{eq:finite_common_coefficient_convergence}
\sup_{t\in[0,T]}
\left\|
\mathcal A_n(t)-\mathcal A(t)
\right\|
+
\max\left\{
\begin{aligned}
&\sup_{t\in[0,T_{\mathrm{ann}}]}
\left\|b_n^0(t)-b^0(t)\right\|,\\
&\max_{k=1,\ldots,m}\sup_{t\in[T_{\mathrm{ann}},T]}
\left\|b_n^k(t)-b^k(t)\right\|
\end{aligned}
\right\}
\leq
\frac{C}{n}.
\end{equation}

We next recall the boundary and matching conditions. For the mean-field
system, for \(k=1,\ldots,m\), they are
\begin{equation}\label{eq:mf_common_boundary_matching_conditions}
\begin{aligned}
\bar X^0(0)&=I^0(0)=0,
&
\begin{pmatrix}
\bar X^k(T_{\mathrm{ann}})\\
I^k(T_{\mathrm{ann}})
\end{pmatrix}
&=
\begin{pmatrix}
\bar X^0(T_{\mathrm{ann}})\\
I^0(T_{\mathrm{ann}})
\end{pmatrix},
\\
\ell^k(T)&=0,
&
\ell^0(T_{\mathrm{ann}})
&=
\sum_{k'=1}^m
\bar\pi^{k'}\ell^{k'}(T_{\mathrm{ann}}).
\end{aligned}
\end{equation}
The finite boundary conditions and
\eqref{eq:finite_common_auxiliary_ell_definition} give, for
\(k=1,\ldots,m\),
\begin{equation}\label{eq:finite_common_auxiliary_ell_matching}
\begin{aligned}
\bar X_n^0(0)&=I_n^0(0)=0,
&
\begin{pmatrix}
\bar X_n^k(T_{\mathrm{ann}})\\
I_n^k(T_{\mathrm{ann}})
\end{pmatrix}
&=
\begin{pmatrix}
\bar X_n^0(T_{\mathrm{ann}})\\
I_n^0(T_{\mathrm{ann}})
\end{pmatrix},
\\
\ell_n^k(T)&=0,
&
\ell_n^0(T_{\mathrm{ann}})
&=
\sum_{k'=1}^m
\bar\pi_n^{k'}\ell_n^{k'}(T_{\mathrm{ann}}).
\end{aligned}
\end{equation}

In summary, \eqref{eq:finite_common_coefficient_convergence} implies that the
finite and mean-field ODE vector fields are \(O(n^{-1})\) apart on bounded
sets. The boundary and matching conditions in
\eqref{eq:mf_common_boundary_matching_conditions} and
\eqref{eq:finite_common_auxiliary_ell_matching} differ only by replacing
\(\bar\pi\) with \(\bar\pi_n\) in the announcement matching condition, so
the matching error is \(O(\|\bar\pi_n-\bar\pi\|)\).

\emph{Step 3: Stability of the Augmented Mean-Field Common-State
Boundary-Value Problem.}
The goal of this step is to show that the \(O(n^{-1})\) vector-field error in
\eqref{eq:finite_common_coefficient_convergence} and the
\(O(\|\bar\pi_n-\bar\pi\|)\) matching error between
\eqref{eq:mf_common_boundary_matching_conditions} and
\eqref{eq:finite_common_auxiliary_ell_matching} imply that the corresponding
solutions are \(O(n^{-1}+\|\bar\pi_n-\bar\pi\|)\) apart.
Unlike for an initial-value problem, this does not follow directly from
Gronwall's inequality: \((\bar X,I)\) have initial conditions, the
post-announcement affine coefficients have terminal conditions, and the
pre- and post-announcement paths are coupled through the matching conditions
at \(T_{\mathrm{ann}}\).

The uniform bounds \eqref{eq:finite_riccati_uniform_bound} and
\eqref{eq:finite_common_affine_uniform_bound} imply that
\eqref{eq:finite_common_original_system} has uniformly bounded coefficients
and forcing. Gronwall's inequality and
\eqref{eq:finite_common_auxiliary_ell_definition} therefore give uniformly over $n$
\begin{equation}\label{eq:finite_common_state_uniform_bound}
\max\left\{
\sup_{t\in[0,T_{\mathrm{ann}}]}\|Y_n^0(t)\|,
\max_{k=1,\ldots,m}\sup_{t\in[T_{\mathrm{ann}},T]}\|Y_n^k(t)\|
\right\}
\leq
C.
\end{equation}

Set
\[
e_n^k
:=
Y_n^k-Y^k
=:
\begin{pmatrix}
e_{n,X}^k\\
e_{n,I}^k\\
e_{n,\ell}^k
\end{pmatrix}.
\]
Subtracting the finite and mean-field ODEs gives
\(\dot e_n^k=\mathcal Ae_n^k+\Delta b_n^k\), where
\(\Delta b_n^k:=(\mathcal A_n-\mathcal A)Y_n^k+b_n^k-b^k\).
By \eqref{eq:finite_common_coefficient_convergence} and the uniform bound on
\(Y_n^k\), \(\Delta b_n^k=O(n^{-1})\) uniformly on the corresponding time
intervals.
For \(k=1,\ldots,m\), the boundary conditions for the error
system are
\begin{equation*}
e_{n,X}^0(0)=0,
\qquad
e_{n,I}^0(0)=0,
\qquad
e_{n,\ell}^k(T)=0.
\end{equation*}
At the announcement, the matching conditions are
\begin{equation}\label{eq:finite_common_error_matching_conditions}
\begin{aligned}
e_{n,X}^k(T_{\mathrm{ann}})
&=e_{n,X}^0(T_{\mathrm{ann}}),
\qquad
e_{n,I}^k(T_{\mathrm{ann}})
=e_{n,I}^0(T_{\mathrm{ann}}),
\qquad k=1,\ldots,m,
\\
e_{n,\ell}^0(T_{\mathrm{ann}})
&=
\sum_{k=1}^m
\bar\pi^{k} e_{n,\ell}^{k}(T_{\mathrm{ann}})
+
e_{n,\pi},
\end{aligned}
\end{equation}
where
\(e_{n,\pi}:=\sum_{k=1}^m
(\bar\pi_n^{k}-\bar\pi^{k})\ell_n^{k}(T_{\mathrm{ann}})\).
The uniform bound on \(Y_n^k\) gives
\(\|e_{n,\pi}\|\leq C\|\bar\pi_n-\bar\pi\|\).
Set \(q_n:=e_{n,\ell}^0(0)\). Variation of constants on the
pre-announcement interval gives
\begin{equation}\label{eq:finite_common_pre_error_representation}
e_n^0(T_{\mathrm{ann}})
=
\mathcal E(T_{\mathrm{ann}},0)
\begin{pmatrix}
0\\0\\q_n
\end{pmatrix}
+\mathcal R_n^0,
\qquad
\mathcal R_n^0
:=
\int_0^{T_{\mathrm{ann}}}
\mathcal E(T_{\mathrm{ann}},s)\Delta b_n^0(s)\,ds.
\end{equation}
For each \(k=1,\ldots,m\), variation of constants on the
post-announcement interval and \(e_{n,\ell}^k(T)=0\) give
\begin{equation}\label{eq:finite_common_post_error_terminal_identity}
0
=
\mathcal E_{\ell X}(T,T_{\mathrm{ann}})e_{n,X}^0(T_{\mathrm{ann}})
+\mathcal E_{\ell I}(T,T_{\mathrm{ann}})e_{n,I}^0(T_{\mathrm{ann}})
+\mathcal E_{\ell\ell}(T,T_{\mathrm{ann}})e_{n,\ell}^k(T_{\mathrm{ann}})
+\mathcal R_{n,\ell}^k,
\end{equation}
where
\(\mathcal R_{n,\ell}^k
:=
\int_{T_{\mathrm{ann}}}^T
\mathcal E_{\ell:}(T,s)
\Delta b_n^k(s)\,ds\).
Averaging \eqref{eq:finite_common_post_error_terminal_identity} with weights
\(\bar\pi^k\) and using \eqref{eq:finite_common_error_matching_conditions} yields
\begin{equation}\label{eq:finite_common_averaged_error_terminal_identity}
\begin{aligned}
0
={}&
\mathcal E_{\ell X}(T,T_{\mathrm{ann}})e_{n,X}^0(T_{\mathrm{ann}})
+\mathcal E_{\ell I}(T,T_{\mathrm{ann}})e_{n,I}^0(T_{\mathrm{ann}})
\\
&+\mathcal E_{\ell\ell}(T,T_{\mathrm{ann}})
\left(e_{n,\ell}^0(T_{\mathrm{ann}})-e_{n,\pi}\right)
+\sum_{k=1}^m\bar\pi^k\mathcal R_{n,\ell}^k.
\end{aligned}
\end{equation}
Combining \eqref{eq:finite_common_pre_error_representation} and
\eqref{eq:finite_common_averaged_error_terminal_identity} with
\(\mathcal E(T,T_{\mathrm{ann}})\mathcal E(T_{\mathrm{ann}},0)=\mathcal E(T,0)\)
gives
\(\mathcal E_{\ell\ell}(T,0)q_n
=\mathcal E_{\ell\ell}(T,T_{\mathrm{ann}})e_{n,\pi}+\zeta_n\),
where
\[
\zeta_n
:=
-
\mathcal E_{\ell:}(T,T_{\mathrm{ann}})
\mathcal R_n^0
-
\sum_{k=1}^m\bar\pi^k\mathcal R_{n,\ell}^k.
\]
The bound on \(\Delta b_n^k\) and boundedness of \(\mathcal E\) imply
\(\|\mathcal R_n^0\|+\max_{1\leq k\leq m}\|\mathcal R_{n,\ell}^k\|
=O(n^{-1})\) and thus \(\|\zeta_n\|=O(n^{-1})\).
Since \(\mathcal E_{\ell\ell}(T,0)\) is invertible by
Theorem~\ref{thm:mf_implementability}, the uniform bound on
\(\ell_n^k\) from \eqref{eq:finite_common_state_uniform_bound} and the estimate
for \(\zeta_n\) give
\(\|q_n\|=O(\varepsilon_n)\), where
\(\varepsilon_n=\frac1n+\|\bar\pi_n-\bar\pi\|\).
The pre-announcement error starts from
\((0,0,q_n)^\top\), so variation of constants and boundedness of
\(\mathcal E\) give
\(\sup_{t\in[0,T_{\mathrm{ann}}]}\|e_n^0(t)\|=O(\varepsilon_n)\).

Solving \eqref{eq:finite_common_post_error_terminal_identity} for
\(e_{n,\ell}^k(T_{\mathrm{ann}})\) gives
\[
e_{n,\ell}^k(T_{\mathrm{ann}})
=
-\left(\mathcal E_{\ell\ell}(T,T_{\mathrm{ann}})\right)^{-1}
\!\left(
\mathcal E_{\ell X}(T,T_{\mathrm{ann}})e_{n,X}^0(T_{\mathrm{ann}})
+\mathcal E_{\ell I}(T,T_{\mathrm{ann}})e_{n,I}^0(T_{\mathrm{ann}})
+\mathcal R_{n,\ell}^k
\right)
\]
where we again used that by Theorem~\ref{thm:mf_implementability},
\(\mathcal E_{\ell\ell}(T,T_{\mathrm{ann}})\) is invertible. The right-hand side
is \(O(\varepsilon_n)\), since \(e_n^0(T_{\mathrm{ann}})\) and
\(\mathcal R_{n,\ell}^k\) are both \(O(\varepsilon_n)\). Variation of constants
and boundedness of \(\mathcal E\) on each post-announcement interval yield
\begin{equation}\label{eq:finite_common_state_convergence}
\max\left\{
\begin{aligned}
&\sup_{t\in[0,T_{\mathrm{ann}}]}
\left\|Y_n^0(t)-Y^0(t)\right\|,\\
&\max_{k=1,\ldots,m}\sup_{t\in[T_{\mathrm{ann}},T]}
\left\|Y_n^k(t)-Y^k(t)\right\|
\end{aligned}
\right\}
=O(\varepsilon_n).
\end{equation}

Since \(\dot{\bar X}_n^k=\bar u_n^k\) and
\(\dot{\bar X}^k=\bar u^k\), taking the first components of the finite and
mean-field ODEs and using
\eqref{eq:finite_common_coefficient_convergence} and
\eqref{eq:finite_common_state_convergence} gives
\begin{equation}\label{eq:finite_common_flow_convergence}
\max\left\{
\begin{aligned}
&\sup_{t\in[0,T_{\mathrm{ann}}]}
\|\bar u_n^{0}(t)-\bar u^0(t)\|,\\
&\max_{k=1,\ldots,m}\sup_{t\in[T_{\mathrm{ann}},T]}
\|\bar u_n^{k}(t)-\bar u^k(t)\|
\end{aligned}
\right\}
=O(\varepsilon_n).
\end{equation}
Equation~\eqref{eq:finite_common_flow_convergence} proves
\eqref{eq:finite_to_mf_common_estimate}, while the \((\bar X,I)\)-components
of \eqref{eq:finite_common_state_convergence} give the corresponding mean
inventory and impact-state estimates.

\emph{Step 4: Player-Level Convergence.}
In this step, we prove that each finite-player inventory deviation \(\delta^{i,k}_n(t)\)
from Subsection~\ref{sec:equilibrium_dynamics} is
\(O(\varepsilon_n)\)-close, uniformly in \(i\), to the mean-field deviation \(\delta^{\pi_n^i,k}(t)\)
of type \(\pi_n^i\) from Remark~\ref{rem:mf_deviation_policy}, both before
and after the announcement in every scenario $k$. Combining these estimates for the deviations
with the common-state estimate in \eqref{eq:finite_common_state_convergence}
gives the player-level inventory estimates.
We then prove the convergence of individual trading rates in
\eqref{eq:finite_to_mf_control_estimate} by comparing the right-hand sides
of the deviation ODEs and using \eqref{eq:finite_common_flow_convergence}.

Set \(\widetilde\ell_n^{i,0}(t)
:=B_n^\top\bigl(r_n^{i,0}(t)-r_n^0(t)\bigr)\) and
\(\widetilde\ell^{\pi_n^i,0}(t):=\ell^{\pi_n^i,0}(t)-\ell^0(t)\) for
\(i=1,\ldots,n\).
By \eqref{eq:pre_ann_delta_dynamics}--\eqref{eq:post_ann_delta_dynamics},
the finite-player pre- and post-announcement deviations satisfy, for
\(k=1,\ldots,m\),
\begingroup
\allowdisplaybreaks[4]
\begin{alignat*}{2}
\dot\delta_n^{i,0}(t)
&=
-\frac1{a_n}\Lambda^{-1}
\left(
[\Theta_n(H_n(t))]_{\delta}\delta_n^{i,0}(t)
+\widetilde\ell_n^{i,0}(t)
\right),
\qquad
&\delta_n^{i,0}(0)&=0,
\\
\dot\delta_n^{i,k}(t)
&=
-\frac1{a_n}\Lambda^{-1}
[\Theta_n(H_n(t))]_{\delta}\delta_n^{i,k}(t),
\qquad
&\delta_n^{i,k}(T_{\mathrm{ann}})
&=\delta_n^{i,0}(T_{\mathrm{ann}}),
\end{alignat*}
while, by \eqref{eq:mf_pre_deviation_dynamics}--\eqref{eq:mf_post_deviation_dynamics},
the mean-field counterparts satisfy
\begin{alignat*}{2}
\dot\delta^{\pi_n^i,0}(t)
&=
-\frac12\Lambda^{-1}
\left(
P(t)\delta^{\pi_n^i,0}(t)
+\widetilde\ell^{\pi_n^i,0}(t)
\right),
\qquad
&\delta^{\pi_n^i,0}(0)&=0,
\\
\dot\delta^{\pi_n^i,k}(t)
&=
-\frac12\Lambda^{-1}P(t)\delta^{\pi_n^i,k}(t),
\qquad
&\delta^{\pi_n^i,k}(T_{\mathrm{ann}})
&=\delta^{\pi_n^i,0}(T_{\mathrm{ann}}).
\end{alignat*}
\endgroup
To prove
\begin{equation}\label{eq:finite_deviation_state_convergence}
\begin{aligned}
&\max_{1\le i\le n}
\sup_{t\in[0,T_{\mathrm{ann}}]}
\left\|
\delta_n^{i,0}(t)-\delta^{\pi_n^i,0}(t)
\right\|
\\
&\quad+
\max_{1\le i\le n}
\max_{1\le k\le m}
\sup_{t\in[T_{\mathrm{ann}},T]}
\left\|
\delta_n^{i,k}(t)-\delta^{\pi_n^i,k}(t)
\right\|
\leq
C\varepsilon_n,
\end{aligned}
\end{equation}
it suffices to establish the following two bounds. First,
\begin{equation}\label{eq:finite_deviation_coefficient_convergence}
\sup_{t\in[0,T]}
\left\|
\frac1{a_n}[\Theta_n(H_n(t))]_{\delta}
-\frac12P(t)
\right\|
\leq
\frac{C}{n}
\end{equation}
which immediately follows from \eqref{eq:finite_theta_delta_convergence} and
\(a_n^{-1}-\frac12=\frac{\varpi\chi_u}{2na_n}=O(n^{-1})\).
Second,
\begin{equation}\label{eq:finite_centered_projection_convergence}
\max_{1\le i\le n}
\sup_{t\in[0,T_{\mathrm{ann}}]}
\left\|
\widetilde\ell_n^{i,0}(t)-\widetilde\ell^{\pi_n^i,0}(t)
\right\|
\leq
C\varepsilon_n.
\end{equation}
Indeed, \eqref{eq:finite_deviation_coefficient_convergence} makes the linear
coefficients of the finite-player and mean-field deviation ODEs
\(O(n^{-1})\) apart. Together with
\(a_n^{-1}-\frac12=O(n^{-1})\) and boundedness of the mean-field solutions,
\eqref{eq:finite_centered_projection_convergence} makes the pre-announcement
forcing terms \(O(\varepsilon_n)\) apart. Subtracting the deviation ODEs as
in Step~3 and applying Gronwall before and after the announcement then gives
\eqref{eq:finite_deviation_state_convergence}.

It remains to prove \eqref{eq:finite_centered_projection_convergence}.
Subtracting \eqref{eq:pre_ann_common_r_ode} from
\eqref{eq:pre_ann_individual_r_ode} and applying \(B_n^\top\) gives
\begin{gather}
\dot{\widetilde\ell}_n^{i,0}
=
-\left(M_n(H_n)B_n\right)^\top\bigl(r_n^{i,0}-r_n^0\bigr)
+\frac1{a_n}
\left[
2\left(\Theta_n(H_n)D_nB_n\right)^\top
-B_n^\top H_nB_n^-
\right]
\Lambda^{-1}\widetilde\ell_n^{i,0},
\notag\\
\widetilde\ell_n^{i,0}(T_{\mathrm{ann}})
=
\sum_{k=1}^m
(\pi_n^{i,k}-\bar\pi_n^k)B_n^\top r_n^k(T_{\mathrm{ann}}).
\label{eq:finite_centered_affine_dynamics}
\end{gather}
Subtracting \eqref{eq:mf_pre_common_affine_system_derived} from
\eqref{eq:mf_pre_ell} gives
\begin{equation}\label{eq:mf_centered_affine_dynamics}
\dot{\widetilde\ell}^{\pi_n^i,0}(t)
=
\frac12P(t)\Lambda^{-1}\widetilde\ell^{\pi_n^i,0}(t),
\qquad
\widetilde\ell^{\pi_n^i,0}(T_{\mathrm{ann}})
=
\sum_{k=1}^m
(\pi_n^{i,k}-\bar\pi^k)\ell^k(T_{\mathrm{ann}}).
\end{equation}
To establish \eqref{eq:finite_centered_projection_convergence}, we show that
(i) the linear coefficients of \eqref{eq:finite_centered_affine_dynamics}
and \eqref{eq:mf_centered_affine_dynamics} differ by \(O(n^{-1})\),
(ii) the forcing $-\left(M_n(H_n)B_n\right)^\top\bigl(r_n^{i,0}-r_n^0\bigr)$ in
\eqref{eq:finite_centered_affine_dynamics} is \(O(n^{-1})\), and (iii) their
terminal conditions at \(T_{\mathrm{ann}}\) differ by \(O(\varepsilon_n)\).
Once these bounds are established, we obtain
\eqref{eq:finite_centered_projection_convergence} by subtracting
\eqref{eq:mf_centered_affine_dynamics} from
\eqref{eq:finite_centered_affine_dynamics}, using boundedness of the
mean-field solution, and applying Gronwall's inequality backward in time.

First, using \(B_n-B_n^-=B_\delta\), \(B_n=B_\delta+O(n^{-1})\),
\(a_n^{-1}-1/2=O(n^{-1})\), and estimates
\eqref{eq:finite_riccati_limit_to_three_block} and
\eqref{eq:finite_common_theta_product_bound}, the linear coefficients
satisfy, uniformly in time,
\[
\begin{aligned}
\frac1{a_n}
\left[
2\left(\Theta_n(H_n)D_nB_n\right)^\top-B_n^\top H_nB_n^-
\right]\Lambda^{-1}
&=\frac1{a_n}B_n^\top H_nB_\delta\Lambda^{-1}+O(n^{-1})\\
&=\frac12P\Lambda^{-1}+O(n^{-1}).
\end{aligned}
\]

We next bound the forcing in \eqref{eq:finite_centered_affine_dynamics}. For this purpose, we
first obtain a uniform bound for \(r_n^{i,0}-r_n^0\). Since
\(\pi_n^i,\bar\pi_n\in\Delta_m\),
\(
\sum_{k=1}^m
|\pi_n^{i,k}-\bar\pi_n^k|
\leq 2
\).
The terminal conditions in \eqref{eq:pre_ann_individual_r_ode} and
\eqref{eq:pre_ann_common_r_ode}, together with
\eqref{eq:finite_common_affine_uniform_bound}, give, uniformly in
\(n\), \(i\), and the finite belief profile,
\begin{equation}\label{eq:finite_centered_r_terminal_bound}
\begin{aligned}
\left\|r_n^{i,0}(T_{\mathrm{ann}})-r_n^0(T_{\mathrm{ann}})\right\|
&=
\left\|
\sum_{k=1}^m
(\pi_n^{i,k}-\bar\pi_n^k)r_n^k(T_{\mathrm{ann}})
\right\|
\\
&\leq
\left(
\sum_{k=1}^m
|\pi_n^{i,k}-\bar\pi_n^k|
\right)
\max_{1\leq k\leq m}
\left\|r_n^k(T_{\mathrm{ann}})\right\|
\\
&\leq
2\max_{1\leq k\leq m}
\left\|r_n^k(T_{\mathrm{ann}})\right\|
\leq C.
\end{aligned}
\end{equation}
Subtracting \eqref{eq:pre_ann_common_r_ode} from
\eqref{eq:pre_ann_individual_r_ode} gives
\(\dot r_n^{i,0}-\dot r_n^0
=-\mathsf A_n^{\mathrm d}(H_n)\bigl(r_n^{i,0}-r_n^0\bigr)\).
By \eqref{eq:pre_ann_matrix_maps_def} and the bounds from Step~1,
\(\mathsf A_n^{\mathrm d}(H_n)\) is uniformly bounded in \(n\) and time.
Gronwall's inequality and the terminal bound
\eqref{eq:finite_centered_r_terminal_bound} therefore give
\begin{equation}\label{eq:finite_centered_r_uniform_bound}
\max_{1\leq i\leq n}
\sup_{t\in[0,T_{\mathrm{ann}}]}
\|r_n^{i,0}(t)-r_n^0(t)\|
\leq C.
\end{equation}
Estimates \eqref{eq:finite_common_m_product_bound} and
\eqref{eq:finite_centered_r_uniform_bound} now yield, uniformly in \(i\) and
time,
\(
\left(M_n(H_n)B_n\right)^\top\bigl(r_n^{i,0}-r_n^0\bigr)
=O(n^{-1}).
\)
Thus the additive forcing is \(O(n^{-1})\).

Finally, we bound the difference between the terminal conditions in
\eqref{eq:finite_centered_affine_dynamics} and
\eqref{eq:mf_centered_affine_dynamics} at \(T_{\mathrm{ann}}\).
By \eqref{eq:finite_common_auxiliary_ell_definition} and
\eqref{eq:finite_common_auxiliary_ell_matching},
\(\ell_n^k(T_{\mathrm{ann}})-B_n^\top r_n^k(T_{\mathrm{ann}})\) is
independent of \(k\). Since \(\sum_{k=1}^m(\pi_n^{i,k}-\bar\pi_n^k)=0\),
\eqref{eq:finite_centered_affine_dynamics} gives
\(\widetilde\ell_n^{i,0}(T_{\mathrm{ann}})
=\sum_{k=1}^m(\pi_n^{i,k}-\bar\pi_n^k)\ell_n^k(T_{\mathrm{ann}})\).
Subtracting the terminal condition in \eqref{eq:mf_centered_affine_dynamics} gives
\begin{equation}\label{eq:finite_centered_terminal_error_identity}
\begin{aligned}
&\widetilde\ell_n^{i,0}(T_{\mathrm{ann}})
-\widetilde\ell^{\pi_n^i,0}(T_{\mathrm{ann}})
\\
&\quad=
\sum_{k=1}^m
(\pi_n^{i,k}-\bar\pi_n^k)
\left(
\ell_n^k(T_{\mathrm{ann}})
-\ell^k(T_{\mathrm{ann}})
\right)
+
\sum_{k=1}^m
(\bar\pi^k-\bar\pi_n^k)
\ell^k(T_{\mathrm{ann}}).
\end{aligned}
\end{equation}
By \eqref{eq:finite_common_state_convergence},
\(\|\ell_n^k(T_{\mathrm{ann}})-\ell^k(T_{\mathrm{ann}})\|
=O(\varepsilon_n)\) uniformly in \(k\).
Since \(\sum_{k=1}^m|\pi_n^{i,k}-\bar\pi_n^k|\leq2\), the first sum in
\eqref{eq:finite_centered_terminal_error_identity} is \(O(\varepsilon_n)\)
uniformly in \(i\). The second sum is \(O(\|\bar\pi_n-\bar\pi\|)\) by
boundedness of \(\ell^k\).
Adding the two estimates gives
\begin{equation}\label{eq:finite_centered_terminal_error_bound}
\max_{1\leq i\leq n}
\left\|
\widetilde\ell_n^{i,0}(T_{\mathrm{ann}})
-\widetilde\ell^{\pi_n^i,0}(T_{\mathrm{ann}})
\right\|
\leq
C\varepsilon_n,
\end{equation}
where \(C\) is independent of \(n\) and the finite belief profile.

Finally, comparing the right-hand sides of the deviation ODEs using
\eqref{eq:finite_deviation_coefficient_convergence},
\eqref{eq:finite_centered_projection_convergence}, and
\eqref{eq:finite_deviation_state_convergence} gives
\(\dot\delta_n^{i,k}-\dot\delta^{\pi_n^i,k}=O(\varepsilon_n)\), uniformly
over players, scenarios, and the corresponding time intervals. Since
\(u_n^{i,k}=\bar u_n^k+\dot\delta_n^{i,k}\) and
\(u^{\pi_n^i,k}=\bar u^k+\dot\delta^{\pi_n^i,k}\),
\eqref{eq:finite_common_flow_convergence} gives
\eqref{eq:finite_to_mf_control_estimate}.
\end{proof}

\subsection{Proof of Theorem~\ref{thm:mf_approximate_nash}}
\label{subsec:proof_mf_approximate_nash}

\begin{proof}
The ODEs for \(\ell^{\pi,0}\) and the inventories are linear, with data
affine in \(\pi\). Thus the rates \(u^{\pi,k}(t)\) in
\eqref{eq:mf_induced_rates} are affine in \(\pi\), with coefficients
uniformly bounded in \(k=0,\ldots,m\) and time. The policies
\(\widehat\Phi_n\) are Borel and non-anticipative. Their bounded rates give
unique solutions of the inventory and impact equations, so
\(\widehat\Phi_n\in\mathcal U_n\).

To show the approximate Nash inequality \eqref{eq:mf_profile_approximate_nash},
fix player \(i\) and retain the opponents' policies
\(\widehat\Phi_n^{-i}\). For \(j\ne i\), \(\widehat\Phi_n^j\) prescribes the
rates \(u^{\pi_n^j,0}\) before the announcement and \(u^{\pi_n^j,k}\) after
scenario \(k\) is announced, independently of player \(i\)'s deviation.
Now let \(\Psi_n^i\) be an arbitrary admissible deviation. By independence of
\(S^0\) and \(\Xi\), the conditioning argument used in the proof of
Theorem~\ref{thm:stitched_hjb_verification} allows us to condition on \(S^0\)
and integrate at the end. For almost every price path under \(\mathbb P_n^i\),
the induced rates form a deterministic trading schedule \(v=(v^0,v^1,\ldots,v^m)\),
with a common pre-announcement component \(v^0\in L^2\) and a
post-announcement component \(v^k\in L^2\) in each scenario \(k\) with
\(\pi_n^{i,k}>0\). Set \(v^k=0\) when
\(\pi_n^{i,k}=0\), without changing the conditional cost.
Define the seminorm \(
\|v\|_i^2
:=
\int_0^{T_{\mathrm{ann}}}\|v^0(t)\|^2\,dt
+
\sum_{k=1}^m\pi_n^{i,k}
\int_{T_{\mathrm{ann}}}^T\|v^k(t)\|^2\,dt
\).

Write \(J_n^i(v;\widehat\Phi_n^{-i})\) and
\(J^{\pi_n^i}(v;\bar u)\) for the finite-player and mean-field objectives
\eqref{eq:reduced_continuation_objective} and \eqref{eq:mf_type_objective},
respectively, evaluated at the deterministic trading schedule \(v\).
A direct calculation using \eqref{eq:reduced_running_cost} gives
\begin{equation}\label{eq:mf_profile_cost_impact_difference}
\begin{aligned}
J_n^i(v;\widehat\Phi_n^{-i})
={}&J^{\pi_n^i}(v;\bar u)
+\sum_{k=1}^m\pi_n^{i,k}
\Bigg[
\int_0^{T_{\mathrm{ann}}}
\left(I_{n,v}^0-I^0+\chi_u\Lambda e_{n,i}^k\right)^\top v^0\,dt
\\
&\qquad\quad+
\int_{T_{\mathrm{ann}}}^T
\left(I_{n,v}^k-I^k+\chi_u\Lambda e_{n,i}^k\right)^\top v^k\,dt
\Bigg],
\end{aligned}
\end{equation}
where \(I^0,I^k\) are the mean-field impact paths from
\eqref{eq:mf_component_impact_dynamics}, and \(I_{n,v}^0,I_{n,v}^k\) are
the finite-player impact paths when player \(i\) uses \(v\) against
\(\widehat\Phi_n^{-i}\).

For \(k=1,\ldots,m\), variation of constants in
\eqref{eq:transient_impact} and \eqref{eq:mf_component_impact_dynamics}
gives, on \([0,T_{\mathrm{ann}}]\) and \([T_{\mathrm{ann}},T]\), respectively,
\begin{equation}\label{eq:mf_profile_impact_differences}
\begin{aligned}
I_{n,v}^0(t)-I^0(t)
&=
\int_0^t e^{-R(t-s)}\Gamma
\left(e_{n,i}^k(s)+\frac{\varpi}{n}v^0(s)\right)\,ds,
\\
I_{n,v}^k(t)-I^k(t)
&=
e^{-R(t-T_{\mathrm{ann}})}
\left(I_{n,v}^0(T_{\mathrm{ann}})-I^0(T_{\mathrm{ann}})\right)
\\
&\quad+
\int_{T_{\mathrm{ann}}}^t e^{-R(t-s)}\Gamma
\left(e_{n,i}^k(s)+\frac{\varpi}{n}v^k(s)\right)\,ds,
\end{aligned}
\end{equation}
where
\begin{equation}\label{eq:mf_profile_opponent_flow_error}
e_{n,i}^{k}(t)
:=
\begin{cases}
\displaystyle
\frac{\varpi}{n}\sum_{j\ne i}u^{\pi_n^j,0}(t)-\varpi\bar u^0(t),
&0\leq t<T_{\mathrm{ann}},\\[6pt]
\displaystyle
\frac{\varpi}{n}\sum_{j\ne i}u^{\pi_n^j,k}(t)-\varpi\bar u^k(t),
&T_{\mathrm{ann}}\leq t\leq T.
\end{cases}
\end{equation}
As noted at the start of the proof, the map
\(\pi\mapsto u^{\pi,k}(t)\) is affine for \(k=0,\ldots,m\), with uniformly
bounded coefficients. Let \(A_u^k(t)\in\mathbb R^{d\times m}\) denote its
linear coefficient. The consistency condition
\eqref{eq:mf_average_consistency} gives \(
\bar u^k(t)
=
\int_{\Delta_m}u^{\pi,k}(t)\,\Pi(d\pi)
\).
Consequently, \(
\frac{\varpi}{n}\sum_{j\ne i}u^{\pi_n^j,k}(t)-\varpi\bar u^k(t)
=
\varpi A_u^k(t)(\bar\pi_n-\bar\pi)
-\frac{\varpi}{n} u^{\pi_n^i,k}(t)
\).
Uniform boundedness of \(A_u^k(t)\) and \(u^{\pi,k}(t)\) over types,
scenarios, and time therefore gives
\begin{equation}\label{eq:mf_profile_opponent_flow_bound}
\max_{1\leq i\leq n}
\max_{1\leq k\leq m}
\sup_{t\in[0,T]}
\|e_{n,i}^{k}(t)\|
\leq
C\left(
\|\bar\pi_n-\bar\pi\|+\frac1n
\right)
=C\varepsilon_n.
\end{equation}
Define \(
F_{n,i}^k(t)
:=
\int_0^t e^{-R(t-s)}\Gamma e_{n,i}^k(s)\,ds
+\chi_u\Lambda e_{n,i}^k(t)
\).
Substituting \eqref{eq:mf_profile_impact_differences} into
\eqref{eq:mf_profile_cost_impact_difference} gives
\begin{align}
J_n^i(v;\widehat\Phi_n^{-i})
={}&
J^{\pi_n^i}(v;\bar u)
+
\sum_{k=1}^m\pi_n^{i,k}
\left[
\int_0^{T_{\mathrm{ann}}}
\left(F_{n,i}^k(t)\right)^\top v^0(t)\,dt
+
\int_{T_{\mathrm{ann}}}^T
\left(F_{n,i}^k(t)\right)^\top v^k(t)\,dt
\right]
\notag\\
&+
\frac{\varpi}{n}
\sum_{k=1}^m\pi_n^{i,k}
\mathcal C\left(
\mathbf 1_{[0,T_{\mathrm{ann}})}v^0
+\mathbf 1_{[T_{\mathrm{ann}},T]}v^k
\right),
\label{eq:mf_profile_cost_decomposition}
\end{align}
where \(\mathcal C\) is the self-induced transient-impact cost defined in
\eqref{eq:transient_self_impact_cost}.
Equation~\eqref{eq:mf_profile_opponent_flow_bound} gives \(\|F_{n,i}^k(t)\|\leq C\varepsilon_n\), uniformly
over players, scenarios, and time.

Define the trading schedule prescribed by \(\widehat\Phi_n^i\) by \(
v^{i,*}
:=
\left(
u^{\pi_n^i,0},
u^{\pi_n^i,1},\ldots,u^{\pi_n^i,m}
\right)
\)
and set \(\Delta v:=v-v^{i,*}\). Applying
\eqref{eq:mf_profile_cost_decomposition} to \(v\) and \(v^{i,*}\) and
subtracting gives
\[
\begin{aligned}
&J_n^i(v;\widehat\Phi_n^{-i})
-J_n^i(v^{i,*};\widehat\Phi_n^{-i})
\\
&\quad=
J^{\pi_n^i}(v;\bar u)
-J^{\pi_n^i}(v^{i,*};\bar u)
\\
&\qquad+
\sum_{k=1}^m\pi_n^{i,k}
\left[
\int_0^{T_{\mathrm{ann}}}
\left(F_{n,i}^k(t)\right)^\top\Delta v^0(t)\,dt
+
\int_{T_{\mathrm{ann}}}^T
\left(F_{n,i}^k(t)\right)^\top\Delta v^k(t)\,dt
\right]
\\
&\qquad+
\frac{\varpi}{n}
\sum_{k=1}^m\pi_n^{i,k}
\left[
\mathcal C\left(
\mathbf 1_{[0,T_{\mathrm{ann}})}v^0
+\mathbf 1_{[T_{\mathrm{ann}},T]}v^k
\right)
-\mathcal C\left(
\mathbf 1_{[0,T_{\mathrm{ann}})}u^{\pi_n^i,0}
+\mathbf 1_{[T_{\mathrm{ann}},T]}u^{\pi_n^i,k}
\right)
\right].
\end{aligned}
\]
We bound each of the three terms on the right-hand side from below.
First, since every deterministic scenario-dependent trading schedule is admissible,
the type-\(\pi\) HJB verification in the proof of
Theorem~\ref{thm:mf_solution} implies that \(v^{i,*}\) minimizes the
quadratic functional \(J^{\pi_n^i}(\,\cdot\,;\bar u)\) over deterministic
scenario-dependent schedules. Its first variation vanishes at \(v^{i,*}\). Since
\(\Lambda\succ0\) and \(Q,Q_T\succeq0\), \(
J^{\pi_n^i}(v;\bar u)
-J^{\pi_n^i}(v^{i,*};\bar u)
\geq
\lambda_{\min}(\Lambda)\|\Delta v\|_i^2
\).
Second, the bound on \(F_{n,i}^k\) and Cauchy--Schwarz imply
\[
\sum_{k=1}^m\pi_n^{i,k}
\left[
\int_0^{T_{\mathrm{ann}}}
\left(F_{n,i}^k(t)\right)^\top\Delta v^0(t)\,dt
+
\int_{T_{\mathrm{ann}}}^T
\left(F_{n,i}^k(t)\right)^\top\Delta v^k(t)\,dt
\right]
\geq
-C\varepsilon_n\|\Delta v\|_i.
\]
Third, for each scenario \(k\), set \(
w
=
\mathbf 1_{[0,T_{\mathrm{ann}})}u^{\pi_n^i,0}
+\mathbf 1_{[T_{\mathrm{ann}},T]}u^{\pi_n^i,k}
\) and \(
h
=
\mathbf 1_{[0,T_{\mathrm{ann}})}\Delta v^0
+\mathbf 1_{[T_{\mathrm{ann}},T]}\Delta v^k
\).
The definition of \(\mathcal C\) gives
\[
\begin{aligned}
\mathcal C(w+h)-\mathcal C(w)
={}&
\mathcal C(h)
+
\int_0^T
w(t)^\top
\left(
\int_0^t e^{-R(t-s)}\Gamma h(s)\,ds
\right)dt
\\
&+
\int_0^T
h(t)^\top
\left(
\int_0^t e^{-R(t-s)}\Gamma w(s)\,ds
\right)dt.
\end{aligned}
\]
Assumption~\ref{ass:standing_model} gives \(\mathcal C(h)\geq0\). The two
remaining integrals are bounded using the transient-impact kernel and the
uniform bound on the mean-field rates. Cauchy--Schwarz therefore gives
\[
\sum_{k=1}^m\pi_n^{i,k}
\left[
\mathcal C\left(
\mathbf 1_{[0,T_{\mathrm{ann}})}v^0
+\mathbf 1_{[T_{\mathrm{ann}},T]}v^k
\right)
-\mathcal C\left(
\mathbf 1_{[0,T_{\mathrm{ann}})}u^{\pi_n^i,0}
+\mathbf 1_{[T_{\mathrm{ann}},T]}u^{\pi_n^i,k}
\right)
\right]
\geq
-C\|\Delta v\|_i.
\]
Combining the three bounds and using \(n^{-1}\leq\varepsilon_n\) gives,
for some \(C_1>0\), \(
J_n^i(v;\widehat\Phi_n^{-i})
-J_n^i(v^{i,*};\widehat\Phi_n^{-i})
\geq \lambda_{\min}(\Lambda)\|\Delta v\|_i^2-C_1\varepsilon_n\|\Delta v\|_i
\).
Completing the square bounds the right-hand side below by
\(-C_2\varepsilon_n^2\), where
\(C_2:=C_1^2/(4\lambda_{\min}(\Lambda))\).
The bound is uniform over deterministic scenario-dependent schedules. Integrating it over
the law of \(S^0\) gives
\eqref{eq:mf_profile_approximate_nash} for every admissible deviation.
\end{proof}

\end{document}